\documentclass[microtype]{gtpart}
\usepackage{a4wide}

\usepackage{textcmds} 
\usepackage{amsmath, amssymb, amsfonts, amstext, amsthm, amscd, mathrsfs}
\usepackage{mathtools}
\usepackage{verbatim}
\usepackage{graphicx}
\usepackage{color}
\usepackage{pinlabel}
\usepackage{mathrsfs} 
\usepackage{xypic}
\usepackage{enumitem}
\usepackage{tikz}
\usetikzlibrary{matrix}

\usepackage{etoolbox}
\makeatletter
\let\ams@starttoc\@starttoc
\makeatother
\usepackage[parfill]{parskip}
\makeatletter
\let\@starttoc\ams@starttoc
\patchcmd{\@starttoc}{\makeatletter}{\makeatletter\parskip\z@}{}{}
\makeatother

\usepackage{hyperref}

\newcommand{\CP}{\mathbb{C}P}
\newcommand{\RP}{\mathbb{R}P}
\newcommand{\T}{\mathbb{T}}

\newcommand{\id}{\mathrm{Id}}

\newcommand{\OP}{\operatorname}
\newcommand{\indx}[1]{\operatorname{index}(#1)}
\newcommand{\pt}{\operatorname{pt}}

\newsavebox{\textvisiblespacebox}
\savebox{\textvisiblespacebox}{\texttt{aa}}
\newcommand\vartextvisiblespace[1][\wd\textvisiblespacebox]{%
  \makebox[#1]{\kern.1em\rule{.4pt}{.3ex}%
  \hrulefill%
  \rule{.4pt}{.3ex}\kern.1em}%
}

\numberwithin{equation}{section}

\newtheorem{thm}{Theorem}[section]
\newtheorem{lma}[thm]{Lemma}
\newtheorem{prp}[thm]{Proposition}
\newtheorem{cor}[thm]{Corollary}
\newtheorem{mainthm}{Theorem}

\newtheoremstyle{TheoremNum}
    {\topsep}{\topsep}              
    {\itshape}                      
    {}                              
    {\bfseries}                     
    {.}                             
    { }                             
    {\thmname{#1}\thmnote{ \bfseries #3}}
\theoremstyle{TheoremNum}

\theoremstyle{definition}

\theoremstyle{remark}
\newtheorem{rmk}[thm]{Remark}

\begingroup 
\makeatletter 
\@for\theoremstyle:=definition,remark,plain,TheoremNum\do{%
\expandafter\g@addto@macro\csname th@\theoremstyle\endcsname{%
\addtolength\thm@preskip\parskip 
}%
} 
\endgroup 

\title[Lagrangians nearby the Whitney immersion]{The classification of Lagrangians nearby the Whitney immersion}

\author{Georgios Dimitroglou Rizell}
\givenname{Georgios}
\surname{Dimitroglou Rizell}
\address{Department of Mathematics\\
Uppsala University\\
Box 480\\
SE-751 06 UPPSALA\\
SWEDEN}
\email{georgios.dimitroglou@math.uu.se}

\keyword{Nearby Lagrangian conjecture, Whitney sphere, Clifford torus, Chekanov torus}
\subject{primary}{msc2010}{53D12}

\arxivreference{math.SG/1712.01182}

\begin{document}

\begin{abstract}
The Whitney immersion is a Lagrangian sphere inside the four-dimensional symplectic vector space which has a single transverse double point of Whitney self-intersection number $+1.$ This Lagrangian also arises as the Weinstein skeleton of the complement of a binodal cubic curve inside the projective plane, and the latter Weinstein manifold is thus the `standard' neighbourhood of Lagrangian immersions of this type. We classify the Lagrangians inside such a neighbourhood which are homologically essential, and which either are embedded or immersed with a single double point; they are shown to be Hamiltonian isotopic to either product tori, Chekanov tori, or rescalings of the Whitney immersion.
\end{abstract}

\maketitle
\setcounter{tocdepth}{1}
\tableofcontents

\section{Introduction}
In the following $(\CP^2,\omega_{\OP{FS}})$ is taken to denote the complex projective plane endowed with the Fubini--Study symplectic form, where the latter has been normalised so that a line is of symplectic area equal to $\int_\ell \omega_{\OP{FS}}=\pi.$ Our main result concerns classification up to Hamiltonian isotopy of embedded Lagrangian tori and immersed Lagrangian spheres inside the open symplectic manifold
$$V \coloneqq \CP^2 \setminus (\ell_\infty \cup C) \subset (\CP^2,\omega_{\OP{FS}}),$$
where $\ell_\infty \subset \CP^2$ denotes the line at infinity and where $C \coloneqq \overline{\{z_1z_2=1\}} \subset \CP^2$ is a smooth conic. In other words, $V$ is the complement of the binodal cubic curve $\ell_\infty \cup C.$ The fact that $(V,\omega_{\OP{FS}})$ is a Liouville manifold for a family of inequivalent Liouville forms $\lambda_r,$ $d\lambda_r=\omega_{\OP{FS}},$ parametrised by $r \in (0,3/2)$ will play an important role in our proof; see Section \ref{sec:cp2liouville} for their construction. (In fact, it is well-known that $V$ even admits the structure of a Weinstein manifold, but this will not be needed.)

In Section \ref{sec:LagrangianFibrationsIntro} below we give an explicit description of a one-parameter family $\Pi_s \colon V \to (-1,1) \times (0,+\infty),$ $s \in (0,\pi/2),$ of Lagrangian fibrations, the fibres of which project to simple closed curves in $\C$ that encircle the value $1 \in \C$ under the standard Lefschetz fibration $(z_1,z_2)\mapsto z_1z_2.$ All fibres of $\Pi_s$ are embedded Lagrangian tori except $\Pi_s^{-1}(0,1)$ which is singular; it consists of a Lagrangian sphere having one transverse double point of Whitney self-intersection number equal to $+1,$ and for which the symplectic action class
$$\int_{[\cdot]} \omega_{\OP{FS}} \colon H_2(\CP^2 \setminus \ell_\infty,\Pi_s^{-1}(0,1)) \to \R$$
assumes precisely the values $ns,$ $n \in \Z.$ This singular fibre is a Lagrangian incarnation of the so-called Whitney immersion, which becomes exact inside $V$ for the aforementioned Liouville form $\lambda_{3s/\pi}$; see Part (1) of Lemma \ref{lma:LiouvilleCP2}.

All Lagrangian fibres of $\Pi_s$ are well studied objects, going back to work by Y. Chekanov \cite{Chekanov:LagrangianTori}, as well as Y. Eliashberg and L. Polterovich \cite{Eliashberg:ProblemLagrangian}; for that reason we call them {\bf standard}. The embedded Lagrangians tori are of two different types: product tori, including the monotone Clifford torus, as well as monotone Chekanov tori. (Monotonicity here refers to the tori when considered inside $(\CP^2 \setminus \ell_\infty, \omega_{\OP{FS}}).$) Our main result can be roughly states as follows: any Lagrangian inside $(V,\omega_{\OP{FS}})$ with the same classical properties as those of a fibre $\Pi_s^{-1}(u_1,u_2)$ is actually Hamiltonian isotopic to a fibre.

The Lagrangian fibrations $\Pi_s$ can be well understood by using the theory of almost toric systems as developed in the work by M. Symington \cite{Symington:FourDimensions}; the unique singular fibre corresponds to the unique node of the base diagram (corresponding to a singularity of ``focus-focus'' type), while varying the parameter $s$ is equivalent to performing a so-called ``nodal slide''. The almost toric systems whose singularity consists of one single such node constitute the simplest examples of nontrivial Lagrangian torus fibrations, and they have therefore been an important example when studying the SYZ conjecture in mirror symmetry; see e.g.~work \cite{Auroux:SpecialLagrangian} by D. Auroux.

The Hamiltonian classification results for Lagrangian submanifold are scarce (except, of course, when the symplectic manifold is of dimension two and the Lagrangian thus is a curve). The only known results for closed Lagrangians exist in the present setting of four-dimensional symplectic manifolds, where strong results have been obtained for embedded discs \cite{Eliashberg:LocalLagrangian} due to Y.~Eliashberg and L.~Polterovich, spheres \cite{Hind:LagrangianSpheres} due to R.~Hind, and tori \cite{Dimitroglou:Isotopy} due to the author together with E.~Goodman and A.~Ivrii. All of these these three works utilise the technique of positivity of intersection in different ways; recall that positivity of intersection for pseudoholomorphic curves is a purely four-dimensional phenomenon.

In higher dimensions the only currently known Hamiltonian isotopy classifications hold for Lagrangians on the \emph{flexible side} of symplectic topology, by the work \cite{LagCaps} due to Y.~Eliashberg and E.~Murphy. These result apply to Lagrangians with conical singularities over \emph{loose} Legendrians. Without going into the details concerning these flexible Lagrangians, we would just like to point out that their singularities are more complicated than a transverse double point, which is the only type of singularity that we consider here.

One of the results proven in \cite{Dimitroglou:Isotopy} was the nearby Lagrangian conjecture for $T^*\T^2$; this is a Weinstein manifold with skeleton being an embedded Lagrangian torus. Since $(V,\omega_{\OP{FS}})$ can be endowed with a Weinstein structure for which the Lagrangian Whitney immersion is the skeleton, our result can be interpreted as a result in line with the nearby Lagrangian conjecture for a generically immersed Lagrangian sphere. Namely, our result in particular provides the following Hamiltonian classification of e.g.~the Lagrangian tori which are homologically essential in some small neighbourhood of such an immersed Lagrangian.

\subsection{Preliminaries}
We begin by swiftly covering the notions needed to formulate our results. The experienced reader can safely skip this subsection.

For convenience, we will often utilise the standard symplectic identification
\begin{gather*}
\varphi \colon (\CP^2 \setminus \ell_\infty,\omega_{\OP{FS}}) \xrightarrow{\cong} (B^4,\omega_0)\\
[z_1:z_2:1] \mapsto \frac{1}{\sqrt{1+\|z_1\|^2+\|z_2\|^2}}(z_1,z_2),
\end{gather*}
with inverse 
$$ (\widetilde{z}_1,\widetilde{z}_2) \mapsto \frac{1}{\sqrt{1-\|\widetilde{z}_1\|^2-\|\widetilde{z}_2\|^2}}(\widetilde{z}_1,\widetilde{z}_2),$$
in order to realise $V$ as an embedding
$$V \cong \widetilde{V} \coloneqq B^4 \setminus \varphi(C) \subset (B^4,\omega_0)$$
into the standard symplectic unit ball. The linear symplectic form $\omega_0=dx_1\wedge dy_1 +dx_2 \wedge dy_2$ is exact with primitive $\lambda_{\OP{std}}=x_1dy_1+x_2dy_2,$ which thus is a Liouville form for the symplectic form on $V$ as well. (This Liouville form does not, however, make $V$ into a Liouville domain.) We will often switch between these two realisations of $(V,\omega_{\OP{FS}})$ in order to work with the description which is most suitable for our different needs.

Recall that a two-dimensional immersion $\iota \colon L \hookrightarrow V$ is {\bf Lagrangian} if $\iota^*\omega_{\OP{FS}}\equiv 0.$ A Lagrangian immersion is {\bf weakly exact} if $\int_\alpha \omega_{\OP{FS}}=0$ for all $\alpha \in \pi_2(V,L).$ Given a choice of Liouville form $\lambda$ for $\omega_{\OP{FS}}=d\lambda,$ we say that the Lagrangian is {\bf exact} in the case when $\iota^*\lambda$ is an exact one-form, and {\bf strongly exact} if the primitive moreover can be chosen to be constant when restricted to each preimage set $\iota^{-1}(\pt),$ $\pt \in \iota(L).$ Note that exact Lagrangian embeddings, as well as \emph{strongly} exact Lagrangian immersions, necessarily also are weakly exact. More generally, the {\bf symplectic action class} is given by $[\lambda|_{TL}] \in H^1(L,\R);$ this class also depends on the choice of Liouville form.

Another important class associated to a Lagrangian submanifold is the {\bf Maslov class}
$$ \mu_L \colon H_2(V,L) \to \Z $$
which takes values in the even integers for an oriented Lagrangian; see e.g.~\cite{McDuff:Introduction} for more details. For general closed curves on $L$ there is also a notion of Maslov class induced by the trivialisation of $\C^2 \supset B^4 \supset \widetilde{V};$ this Maslov class will be denoted by $\mu_L^{\C^2} \colon H^1(L) \to \Z.$ Note that the equality $\mu_L^{\C^2} \circ \partial=\mu_L$ holds, where $\partial \colon H_2(V,L) \to H_1(L)$ is the connecting homomorphism.

The classification that we are pursuing is that of Lagrangians up to {\bf Hamiltonian isotopy} $\phi^t_{H_t},$ i.e.~a smooth isotopy whose infinitesimal generator satisfies $\iota_{X_t}\omega_{\OP{FS}}=-dH_t$ for a smooth family of functions $H_t \colon V \to \R;$ this function is called the generating {\bf Hamiltonian.} A standard result shows that a smooth path of Lagrangian embeddings $L_t \subset (V,\omega_{\OP{FS}}=d\lambda),$ also called a {\bf Lagrangian isotopy}, is generated by a global Hamiltonian isotopy if and only if the symplectic action is \emph{constant} for the path relative an arbitrary choice of Liouville form $\lambda.$ The difference
$$\OP{Flux}(\{L_s\}_{s \in [a,b]}) \coloneqq [\lambda|_{TL_b}]-[\lambda|_{TL_a}] \in H^1(\T^2)$$
of symplectic action classes (suitably identified) is called the {\bf symplectic flux} of the Lagrangian isotopy $\{L_t\}_{t \in [a,b]}.$ A Lagrangian isotopy $L_t,$ $t \in [0,1]$ is thus generated by a Hamiltonian if and only if the corresponding {\bf symplectic flux-path} defined as
$$t \mapsto \OP{Flux}(\{L_s\}_{s \in [0,t]}) \in H^1(\T^2)$$
vanishes equivalently for all $t \in [0,1].$ 

\subsection{Result}

The result that we show here is a classification of the Lagrangians inside $(V,\omega_{\OP{FS}})$ up to Hamiltonian isotopy under the assumption that they satisfy properties similar to those of the fibres of $\Pi_s.$ Our main result is as follows.
\begin{mainthm}
\label{thm:main}
Let $L \subset (V,\omega_{\OP{FS}})$ be either an embedded Lagrangian torus or an immersed Lagrangian sphere with a single transverse double point. Assume that \emph{at least} one of the following two conditions are satisfied:
\begin{enumerate}
\item the class $[L] \in H_2(V) \cong \Z$ is nonzero in homology; {\bf or}
\item
\begin{enumerate}
\item {\em In the case when $L$ is an embedded torus:} for any homotopy class $\alpha \in \pi_2(V,L)$ the implication
$$\mu_L(\alpha)=2 \Rightarrow \int_\alpha \omega_0 \le 0$$
holds,
\item {\em In the case when $L$ is an immersed sphere:} property (a) holds for any Lagrangian torus resulting from a Lagrange surgery applied to its double point.
\end{enumerate}
\end{enumerate}
Then $L$ is Hamiltonian isotopic inside $(V,\omega_{\OP{FS}})$ to a standard Lagrangian. In other words, there exists a nonempty subset of values $s \in (0,\pi/2)$ (possibly the entire interval) such that $L$ is Hamiltonian isotopic to a unique fibre of the Lagrangian fibration $\Pi_s \colon V \to (-1,1) \times (0,+\infty).$
\end{mainthm}
It is not difficult to see that either of the conditions in the above Theorem \ref{thm:main} are satisfied for standard Lagrangians, i.e.~the fibres of the fibrations $\Pi_s;$ we refer to Section \ref{sec:LagrangianFibrationsIntro} below for more details.
\begin{rmk}
We point out the following, with more details given below in Proposition \ref{prp:mainprp}.
\begin{enumerate}
\item Condition (2) of Theorem \ref{thm:main} is automatically satisfied in the case when the Lagrangian (embedded or immersed) is weakly exact, or if the Lagrangian is embedded and has vanishing Maslov class. Further, note that all fibres $\Pi_s^{-1}(u_1,u_2)$ have vanishing Maslov class when considered inside $(V,\omega_{\OP{FS}}),$ while they are weakly exact if and only if $u_1=0$;
\item The Hamiltonian isotopy classes of the strongly exact immersed spheres $\Pi_{s}^{-1}(0,1)$ are all different, uniquely determined by the parameter $s \in (0,\pi/2).$ Contrary to this, every fixed non-weakly exact torus fibre $\Pi_{s_0}^{-1}(u_1,u_2)$ (i.e.~$u_1 \neq 0$) is Hamiltonian isotopic to a unique fibre $\Pi_s^{-1}(u_1,u_2')$ for any other choice of $s \in (0,\pi/2)$ as well. For the weakly exact tori the situation is more complicated. A given weakly exact torus fibre of Clifford type is Hamiltonian isotopic to a fibre of the fibrations $\Pi_s$ only for a certain strict sub-interval of parameters $s \in (0,s_0) \subsetneq (0,\pi/2),$ while for a torus of Chekanov type the corresponding sub-interval is of the form $(s_0,\pi/2).$
\end{enumerate}
\end{rmk}

Theorem \ref{thm:main} gives conditions for when a Lagrangian torus inside $(\CP^2,\omega_{\OP{FS}})$ is Hamiltonian isotopic to a torus of either Clifford or Chekanov type in terms of the linking properties with a binodal cubic. In particular, Theorem \ref{thm:main} shows there are precisely two different monotone Lagrangian tori which are exact in the complement of the binodal cubic up to Hamiltonian isotopy. It has been shown by R. Vianna \cite{Vianna:second} that there are \emph{infinitely} many different Hamiltonian isotopy classes of monotone Lagrangian tori inside $(\CP^2,\omega_{\OP{FS}})$ which, moreover, can be realised as exact Lagrangians inside the complement of the \emph{smooth} cubic curve.

The central technique used in the proof of Theorem \ref{thm:main} is to consider the limit of pseudoholomorphic foliations by conics when stretching the neck around the Lagrangian torus. Note that there is a natural holomorphic conic fibration on $V$ being the restriction of the Lefschetz fibration $z_1z_2$ on $\CP^2 \setminus \ell_\infty$ to the complement of the smooth fibre $C$ above $1 \in \C.$ Here we study the foliations given by such conics that satisfy an additional tangency to $C$ at $\ell_\infty$ for arbitrary almost complex structures (which still required to be standard near $\ell_\infty$).

The idea to use pseudoholomorphic foliations and neck stretching to classify Lagrangian tori goes back to H.~Hofer and K.~M.~Luttinger. This program was carried out in the recent work \cite{Dimitroglou:Isotopy} by the author together with E.~Goodman, and A.~Ivrii, where foliations by degree one spheres were considered; see \cite[Section 1.2]{Dimitroglou:Isotopy} for more details concerning the history of the problem. One notable result obtained using these techniques was the positive answer to the so called nearby Lagrangian conjecture for $\T^2$; see \cite[Theorem B]{Dimitroglou:Isotopy}. In the course of proving Theorem \ref{thm:main} we also need to provide a sharpening of this result:
\begin{mainthm}
\label{thm:nearby}
Suppose that $L \subset (T^*\T^2,d\lambda_{\T^2})$ is a Lagrangian embedding which is either weakly exact, homologically essential, or a Lagrangian torus of vanishing Maslov class. In all these cases, $L$ is Hamiltonian isotopic to the graph of a closed one-form in $\Omega^1(\T^2).$ Moreover, for any convex subset $U \subset \R^2$ it is the case that:
\begin{enumerate}
\item If $L \subset \T^2 \times U \subset T^*\T^2,$ then the Hamiltonian isotopy can be taken to be supported inside $\T^2 \times U;$ and
\item For any $ \theta \in S^1$ consider the properly embedded Lagrangian disc
\begin{gather*}
 \dot{D}_{\mathbf{p}^0}(\theta) \coloneqq (S^1 \times \{\theta\}) \times (\{p_1^0\} \times (-\infty,p_2^0]) \subset \T^2 \times \R^2 =T^*\T^2, \:\:\mathbf{p}^0:=(p_1^0,p_2^0),
\end{gather*}
with one interior point removed. If it is the case that
$$ L \cap \dot{D}_{\mathbf{p}^0}(e^{is})=\partial \dot{D}_{\mathbf{p}^0}(e^{is}) = S^1 \times \{ e^{is} \} \times \{\mathbf{p}^0\}$$
holds for all $|s|<\epsilon,$ then the Hamiltonian isotopy can be assumed to be supported outside of the subset
$$\bigcup_{|s|<\delta} \dot{D}_{\mathbf{p}^0}( e^{is} ) = S^1 \times e^{i[-\delta,\delta]} \times \{p_1^0\} \times (-\infty,p_2^0 ],$$
for some $0<\delta<\epsilon$ sufficiently small (note that for symplectic action reasons, we may not be able to Hamiltonian isotope the Lagrangian to the \emph{constant} section $\T^2 \times \{\mathbf{p}^0\}$); and
\item If $L \subset \T^2 \times U$ holds in addition to the assumptions of (2), then the Hamiltonian isotopy produced there can moreover be taken to have support contained inside $\T^2 \times U.$
\end{enumerate}
\end{mainthm}
We prove this result in Section \ref{sec:proofnearby}. Note that Part (1) is a fairly straight forward consequence of \cite[Theorem 7.1]{Dimitroglou:Isotopy}, while Part (2) requires a more careful study of its proof. Part (3) finally follows without too much additional work by simply combining Parts (1) and (2).

In Section \ref{sec:SelfPlumbing} we show that $(V,\omega_{\OP{FS}})$ is a Liouville domain with completion $(\widehat{W},d\lambda),$ which is a Liouville manifold whose skeleton is the Whitney sphere itself. A version of the Weinstein neighbourhood theorem shows that $\widehat{W}$ serves as a standard neighbourhood for any immersed Lagrangian sphere having a single self-intersection point of positive sign. The classification given by Theorem \ref{thm:main} can thus be interpreted as a result in line with the nearby Lagrangian conjecture for such an \emph{immersed} Lagrangian sphere. In addition, note that $(V,\omega_{\OP{FS}})$ is a Weinstein manifold that, in some sense, is not too distant from the cotangent bundle of a torus, since it can obtained from $(DT^*\T^2,d\lambda_{\T^2})$ by attaching a single Weinstein two-handle along the conormal lift of a simple closed geodesic.

The complete Liouville manifold $(\widehat{W},d\lambda)$ admits a surjective Lagrangian \emph{almost toric} fibration $\hat{\pi} \colon \widehat{W} \to \R^2$ with a single nodal fibre in the sense of \cite{Symington:FourDimensions}. The nodal fibre is an immersed Lagrangian sphere, while all other fibres are embedded tori. The process of a \emph{nodal slide} introduced in the aforementioned work can be applied to the node, thereby producing a one-parameter family $\hat{\pi}_s$ of almost toric fibrations of the same type. The Hamiltonian isotopy class of the nodal fibres for different values of $s \in \R$ live in distinct Hamiltonian isotopy classes. We fix our convention so that $\hat{\pi}_0^{-1}(0)$ is a strongly exact immersed Lagrangian sphere for the Liouville form $\lambda.$ Recall that, since the fibration is assumed to be almost toric, the map $\hat{\pi}_0$ is determined up to an affine transformation by the corresponding Lagrangian foliation. 
\begin{cor}
A Lagrangian $L \subset (\widehat{W},d\lambda)$ that satisfies either of the assumptions in Theorem \ref{thm:main} is Hamiltonian isotopic to a fibre of $\hat{\pi}_s$ for some $s \in \R.$ If $L$ moreover is strongly exact with respect to the Liouville form $\lambda,$ and if $ \hat{\pi}_0(L)$ is contained inside a subset $O \subset \R^2$ which is star-shaped with respect to the origin (in the above affine coordinates), then the Hamiltonian isotopy may be taken to have support inside the preimage $\hat{\pi}_0^{-1}(O).$
\end{cor}
\begin{rmk}
 
We do not know whether it is possible to confine the above Hamiltonian isotopy to the preimage of the star-shaped subset $O \subset \R^2$ without the assumption of $L$ being strongly exact. A direct consequence of Theorem \ref{thm:main} tells us that this stronger result is true at least in case when the preimage $\pi^{-1}_0(O)$ is symplectomorphic to $(V,c\omega_{\OP{FS}})$ for some $c>0.$
\end{rmk}
\begin{proof}
Consider the fibre of $\hat{\pi}_s$ that has the same classical invariants as those of $L,$ and observe that:
\begin{itemize}
\item in the case when $L$ is a sphere, this is the nodal fibre for a uniquely determined value of $s \in \R,$
\item in the case when $L$ is an embedded torus which is not weakly exact, there is a unique such representative up to Hamiltonian isotopy, which moreover appears as fibres in the fibrations $\hat{\pi}_s$ for any choice of $s \in \R,$ and
\item in the case when $L$ is a weakly exact embedded torus, there are precisely two such Hamiltonian isotopy classes, and both can be assumed to be obtained by appropriate action-preserving Lagrange surgeries on the immersed sphere $\hat{\pi}_s^{-1}(0)$ for a uniquely determined value of $s \in \R.$
\end{itemize}

Using Proposition \ref{prp:embedding} one constructs a symplectic embedding $\iota_s \colon (V,c\omega_{\OP{FS}}) \hookrightarrow (\widehat{W},d\lambda)$ for some $c \gg 0$ satisfying the properties that both $L$ and the the Lagrangian fibre of $\hat{\pi}_s$ pinpointed above are contained inside the image of $\iota_s.$ Theorem \ref{thm:main} finally implies that $L$ is Hamiltonian isotopic inside $V$ to a fibre of $\Pi_{s'}$ for some $s' \in (0,\pi/2).$ Since the same is true also for the aforementioned fibre of $\hat{\pi}_s,$ we have thereby managed to produce our Hamiltonian isotopy contained entirely inside $\iota_s(V).$

For the last point, it is sufficient to apply the negative Liouville flow of $\lambda$ to the Hamiltonian isotopy, to make sure that it stays inside the required subset. To that end, note that the Liouville form preserves exact Lagrangian submanifolds, and that it retracts the subset $\hat{\pi}_0^{-1}(O)$ onto the immersed Lagrangian sphere $\hat{\pi}_0^{-1}(0).$ See Section \ref{sec:Liouville} for more details.
\end{proof}

\subsection{A family of Lagrangian fibrations}
\label{sec:LagrangianFibrationsIntro}

Since the work of Y.~Chekanov \cite{Chekanov:LagrangianTori} it has been known that $(B^4,\omega_0)$ admits two types of monotone Lagrangian tori that live in different Hamiltonian isotopy classes but whose classical invariants agree; these are the so-called Clifford and Chekanov tori. These two types of Lagrangian tori can also be realised as weakly exact Lagrangian tori inside $(V,\omega_{\OP{FS}});$ see \cite{Eliashberg:ProblemLagrangian} by Y.~Eliashberg L.~Polterovich, as well as \cite{Auroux:SpecialLagrangian} by D.~Auroux. Furthermore, as shown in the latter article, these tori arise as the leaves of Lagrangian torus fibrations on $\CP^2,$ which also were studied in \cite{Pascaleff:FloerCohomology} by J.~Pascaleff as well as \cite{Vianna:first} by R. Vianna. In particular, the latter article makes use of the convenient language of almost toric fibrations and their deformations as introduced by M.~Symington \cite{Symington:FourDimensions}.

In this subsection we recall an explicit description of a one-parameter family of such Lagrangian fibrations on $(V,\omega_{\OP{FS}}).$ Alternatively, this fibration can be constructed by using the language of almost-toric base diagrams. More precisely, one can deform the standard moment polytope of $B^4$ (i.e.~the standard fibration by product tori) by a so-called nodal trade, followed by a nodal slide; the parameter of the slide induces the parameter in our family of fibrations. The reason for instead choosing the more explicit approach is as follows. First, we want the Lagrangian fibrations to be {\bf compatible} with the standard Lefschetz fibration
\begin{gather*}
f \colon (\CP^2 \setminus \ell_\infty,\omega_{\OP{FS}}) \to \C,\\
(z_1,z_2) \mapsto z_1\cdot z_2,
\end{gather*}
by holomorphic conics in the following sense: the restriction of $f$ to a smooth fibre $L$ of the Lagrangian fibration is a smooth $S^1$-bundle over the simple closed curve $f(L) \subset \C.$ Second, we want the fibration to restrict to a fibration on $V = \CP^2 \setminus (\ell_\infty \cup C),$ i.e.~we want the smooth Lagrangian fibres to be disjoint from the smooth conic $C.$

Before starting, we say a few more words about the Lefschetz fibration $f.$ Note that it has a unique singular fibre
$$ C_{\OP{nodal}} \coloneqq f^{-1}(0)=\ell_1 \cup \ell_2$$
consisting of a union
$$ \ell_1\coloneqq \overline{\C \times \{0\}} \:\:\: \text{and} \:\:\: \ell_2\coloneqq \overline{\{0\} \times \C}$$
of two lines. All conic fibres of $f$ pass through the two points $q_i \coloneqq \ell_i \cap \ell_\infty$ while being tangent to $v_i \coloneqq T_{q_i}\ell_i.$ We will call $f$ a {\bf symplectic Lefschetz fibration}, even if we make no claims concerning the symplectic triviality outside of a compact subset.

Now we are ready to commence with the construction of the Lagrangian fibrations $\Pi_s.$ For convenience we will work with the identification $\widetilde{V} \subset (B^4,\omega_0)$ of $(V,\omega_{\OP{FS}}),$ where the fibration takes the form
\begin{gather*}
\widetilde{f} \coloneqq f \circ \varphi^{-1} \colon B^4 \to \C,\\
(\widetilde{z}_1,\widetilde{z}_2) \mapsto \frac{\widetilde{z}_1\widetilde{z}_2}{1-\|\widetilde{z}_1\|^2-\|\widetilde{z}_2\|^2},
\end{gather*}
where $(\widetilde{z}_1,\widetilde{z}_2)$ are the standard complex coordinates on $B^4 \subset \C^2.$ (This notation is useful for distinguishing the two different complex coordinates $\widetilde{z}_i,z_i \colon V \to \C.$)

Begin by constructing a smooth one-parameter family of diffeomorphisms
$$\Psi_s \colon \C \to \C, \:\: s \in (0,\pi/2)$$
that satisfy
\begin{align*}
& \Psi_s(t)=t, \:\: t \in [0,1],\\
& \overline{\Psi_s(z)}=\Psi_s(\overline{z}),
\end{align*}
and which all are equal to the identity outside of some compact subset (that necessarily depends on the parameter $s$). We moreover demand that they satisfy the following. Let $\eta_c\coloneqq\{\|z-1\|=c\} \subset \C$ be the foliation of $\C \setminus \{1\}$ by concentric circles centered at $1 \in \C,$ where $\eta_1$ is the leaf passing through the origin. Then:
\begin{enumerate}
\item The curve
$$(\Psi_s \circ \widetilde{f})^{-1}(\eta_1) \cap \{\widetilde{z}_1=\widetilde{z}_2\} \subset \widetilde{V} \subset B^4 $$
is an immersion of the form $(\alpha_s,\alpha_s),$ for an immersed figure-8 curve $\alpha_s \subset D_{1/\sqrt{2}}^2$ satisfying $-\alpha_s=\alpha_s.$ We require that the symplectic area of either of the two bounded components of $\alpha_s \subset (D_{1/\sqrt{2}}^2,dx \wedge dy)$ is equal to $s/2 \in (0,\pi/4)$ (i.e.~the symplectic area inside $B^4$ bounded by $(\alpha_s,\alpha_s)$ is equal to $s \in (0,\pi/2)$); and
\item There is a smooth extension of the family to include also diffeomorphisms
\begin{align*}
& \Psi_0 \colon \C \setminus [0,1] \xrightarrow{\cong} \C \setminus D^2(1),\\
& \Psi_{\pi/2} \colon \C \setminus (-\infty,0] \xrightarrow{\cong} B^2(1) \subset \C,
\end{align*}
where $\Psi_0$ is the identity outside of some compact subset, and $\Psi_{\pi/2}$ is the identity near $1 \in \C.$
\end{enumerate}
The dependence of the curves $\gamma_c\coloneqq\Psi_s^{-1}(\eta_c)$ on the parameter $s \in (0,\pi/2)$ is exhibited in Figure \ref{fig:foliation}.
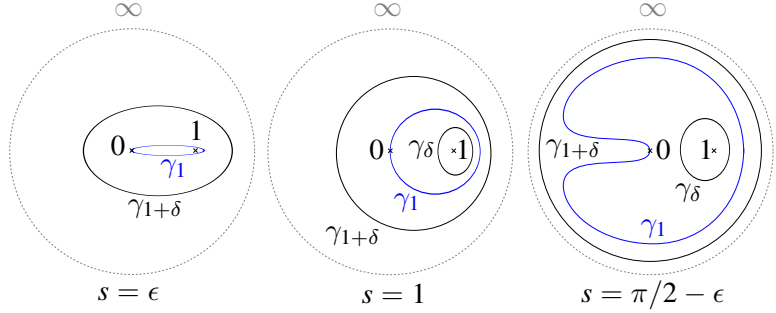
\begin{figure}[htp]
\begin{center}
\vspace{6mm}
\labellist
\pinlabel $0$ at 82 98
\pinlabel $0$ at 277 95
\pinlabel $\gamma_{1+\delta}$ at 259 29
\pinlabel $0$ at 493 95
\pinlabel $1$ at 525 95
\pinlabel $s=\epsilon$ at 90 -14
\pinlabel $\color{blue}\gamma_1$ at 124 80
\pinlabel $1$ at 141 109
\pinlabel $\color{blue}\gamma_1$ at 300 55
\pinlabel $\color{blue}\gamma_1$ at 486 35
\pinlabel $\gamma_{1+\delta}$ at 110 50
\pinlabel $\gamma_{1+\delta}$ at 425 95
\pinlabel $\gamma_{\delta}$ at 311 95
\pinlabel $1$ at 342 95
\pinlabel $\gamma_{\delta}$ at 512 62
\pinlabel $\color{gray}\infty$ at 90 200
\pinlabel $\color{gray}\infty$ at 285 200
\pinlabel $\color{gray}\infty$ at 485 200
\pinlabel $s=1$ at 290 -14
\pinlabel $s=\pi/2-\epsilon$ at 485 -15
\endlabellist
\includegraphics[scale=0.5]{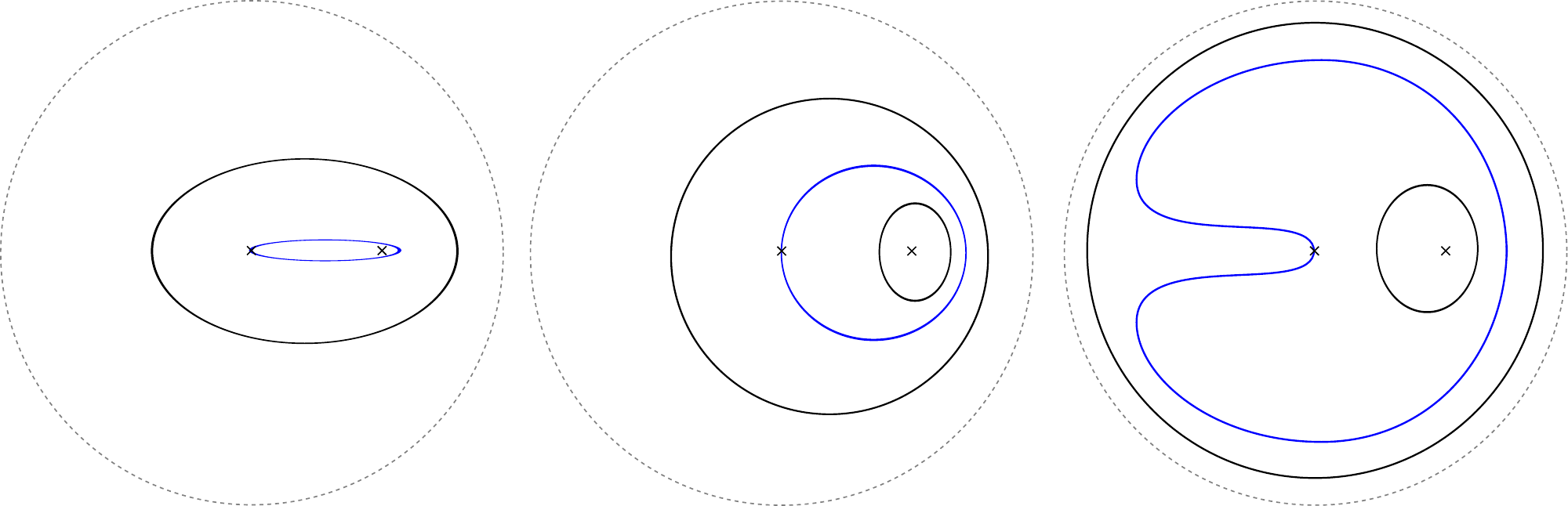}
\vspace{6mm}
\caption{The image $\gamma_c\coloneqq\Psi_s^{-1}(\eta_c)$ of the leaves in the foliation of $\C$ by the concentric circles $\eta_c=\{\|z-1\|=c\}$ for different values of the parameter $s \in (0,\pi/2).$ }
\label{fig:foliation}
\end{center}
\end{figure}

We can now finally define the smooth one-parameter family of Lagrangian fibrations to be
\begin{center}
\boxed{\parbox[t]{13cm}{
\begin{align*}
& \widetilde{\Pi}_s \colon \widetilde{V} \to (-1,1) \times (0,+\infty),\\
& (\widetilde{z}_1,\widetilde{z}_2) \mapsto (\|\widetilde{z}_1\|^2-\|\widetilde{z}_2\|^2,\|\Psi_s(\widetilde{f}(\widetilde{z}_1,\widetilde{z}_2))-1\|^2).
\end{align*}}}
\end{center}
We also write $\Pi_s \coloneqq \widetilde{\Pi}_s \circ \varphi$ for the corresponding fibration on $V.$

It is immediate by the construction that these fibrations are compatible with the standard symplectic Lefschetz fibration $f;$ See Figures \ref{fig:fibration-clifford}, \ref{fig:fibration-chekanov}, and \ref{fig:fibration-sphere}, for a schematic depiction of this. The fibres $\Pi_s^{-1}(u_1,u_2)$ with $u_2 > 1$ and $u_2 < 1$ will be called tori of {\bf Clifford type} and {\bf Chekanov type}, respectively. Note that a torus of Clifford type is fibred over a curve whose winding number around $0\in\C$ is equal to one, while a torus of Chekanov type is fibred over such a curve with zero winding. 

Later we will also make use of the limit case when $s=0.$ In this case, by Property (2) above, we obtain an induced Lagrangian fibration
$$ \Pi_0 \colon V \setminus f^{-1}[0,1] \to (-1,1) \times (1,+\infty),$$
all whose fibres will turn out to be embedded Lagrangian tori.

We summarise the important properties of the Lagrangian fibres of $\Pi_s$ in the following proposition, the proof of which we postpone to Section \ref{sec:mainprpproof}.
\begin{prp}
\label{prp:mainprp}
All fibres of $\Pi_s$ are compact Lagrangian immersions contained inside $(V,\omega_{\OP{FS}}).$ Furthermore, it is the case that:
\begin{enumerate}
\item The fibres $L_{\OP{Wh}}(s)\coloneqq\Pi_s^{-1}(0,1),$ $s\in(0,\pi/2),$ constitute a one-parameter family of weakly exact Lagrangian immersions of the sphere, each having a single transverse double-point and Whitney self-intersection number equal to $+1.$ It is the case that $[L_{\OP{Wh}}(s)] \in H_2(V) \cong \Z$ is a generator. Moreover, the primitive of the pull-back $\lambda_{\OP{std}}|_{TL_{\OP{Wh}}(s)}$ has potential difference equal to $\Delta(L_{\OP{Wh}}(s))=s$ at the double point; two different such spheres are hence not Hamiltonian isotopic.
\item A fibre $\Pi_s^{-1}(u_1,u_2)$ for $(u_1,u_2) \neq (0,1)$ is an embedded Lagrangian torus which is homologous to the generator $[L_{\OP{Wh}}(s)] \in H_2(V),$ and its Maslov class evaluates to zero on any element in $H_2(V,L).$ Such a fibre is weakly exact (and hence also monotone) if and only if $ u_1 = 0.$ Moreover:
\begin{enumerate}
\item A weakly exact torus of Clifford type (resp. Chekanov type) is Hamiltonian isotopic inside $B^4 \supset \widetilde{V}$ to a Clifford torus (resp. Chekanov torus); while
\item Any non weakly exact fibre $\Pi^{-1}_s(u_1,u_2),$ $u_1 \neq 0,$ is Hamiltonian isotopic inside $V$ to a unique fibre of the form $(\Pi_1)^{-1}(u_1,u_2')$ for some appropriate $u_1' \in (0,+\infty).$
\end{enumerate}
\end{enumerate}
\end{prp}
We point out that the nonzero homology groups for this space are $H_i(V)=\Z$ for $i=0,1,2.$ By Part (2.a) of the above proposition combined with Y.~Chekanov's classical result \cite{Chekanov:LagrangianTori} we deduce the following: Weakly exact (or monotone) tori of Clifford and Chekanov type are never Hamiltonian isotopic, while the non weakly-exact (or non-monotone) such tori all are Hamiltonian isotopic to product tori.

\begin{figure}[htp]
\begin{center}
\vspace{6mm}
\labellist
\pinlabel $f^{-1}(2+\epsilon)$ at 120 116
\pinlabel $f^{-1}(0)$ at 70 116
\pinlabel $iy$ at 107 45
\pinlabel $1$ at 85 22
\pinlabel $2$ at 105 32
\pinlabel $x$ at 145 27
\pinlabel $\color{blue}f(L_{\OP{Cl}})$ at 88 8
\endlabellist
\includegraphics{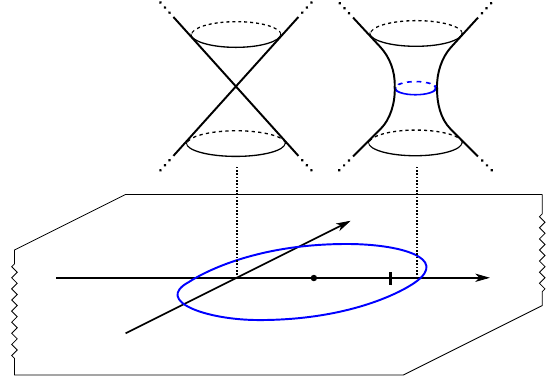}
\caption{A Clifford torus $L_{\OP{Cl}}\coloneqq\Pi_s^{-1}(0,u_2)$ with $u_2 > 1$ is fibred over a curve in the base of the Lefschetz fibration which encircles the critical value.}
\label{fig:fibration-clifford}
\end{center}
\end{figure}
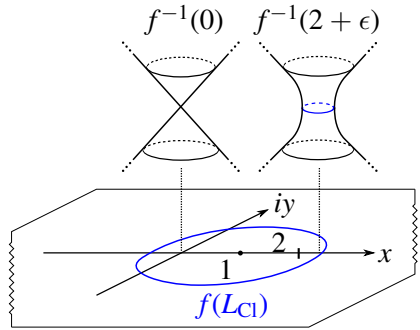

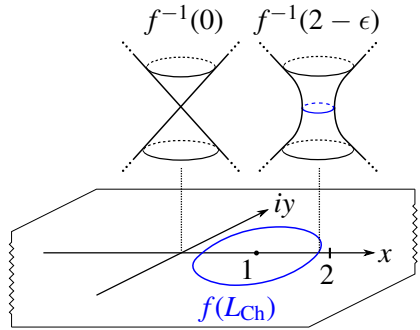
\begin{figure}[htp]
\begin{center}
\vspace{6mm}
\labellist
\pinlabel $f^{-1}(2-\epsilon)$ at 120 116
\pinlabel $f^{-1}(0)$ at 70 116
\pinlabel $iy$ at 107 45
\pinlabel $1$ at 93 22
\pinlabel $2$ at 123 20
\pinlabel $x$ at 145 27
\pinlabel $\color{blue}f(L_{\OP{Ch}})$ at 90 7
\endlabellist
\includegraphics{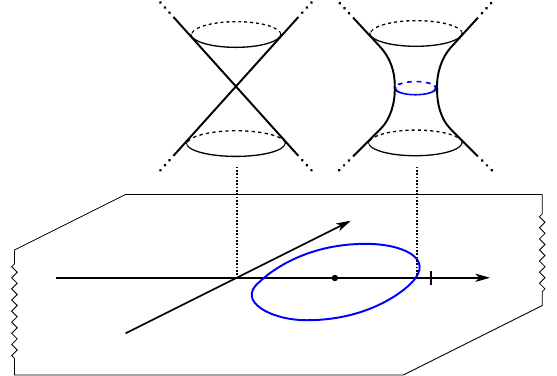}
\caption{A Chekanov torus $L_{\OP{Ch}}\coloneqq\Pi_s^{-1}(0,u_2)$ with $u_2 < 1$ is fibred over a curve in the base of the Lefschetz fibration which does not encircle the critical value.}
\label{fig:fibration-chekanov}
\end{center}
\end{figure}

\begin{figure}[htp]
\begin{center}
\vspace{6mm}
\labellist
\pinlabel $f^{-1}(2)$ at 120 116
\pinlabel $f^{-1}(0)$ at 70 116
\pinlabel $iy$ at 107 45
\pinlabel $1$ at 91 22
\pinlabel $2$ at 123 22
\pinlabel $x$ at 145 27
\pinlabel $\color{blue}f(L_{\OP{Wh}})$ at 100 7
\endlabellist
\includegraphics{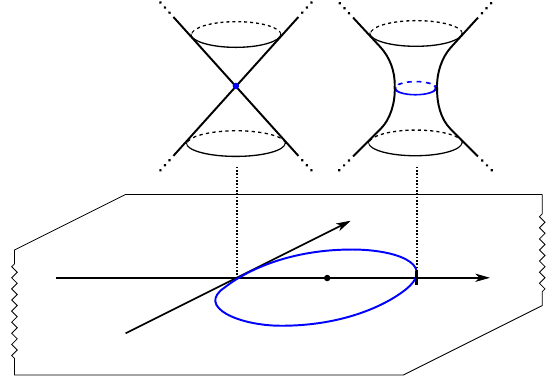}
\caption{A Whitney immersion $L_{\OP{Wh}}\coloneqq\Pi_s^{-1}(0,1)$ is fibred over a curve in the base of the Lefschetz fibration which passes though the critical value.}
\label{fig:fibration-sphere}
\end{center}
\end{figure}
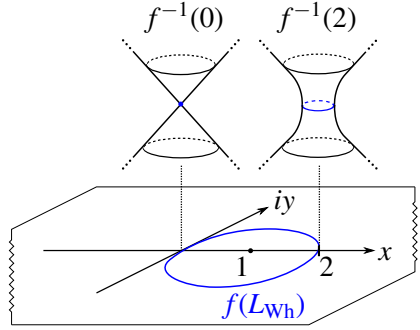
Finally, we make the following comment regarding the apparent non-symmetry that holds between the tori of Clifford and Chekanov type when considered in conjunction with the Lefschetz fibration $f;$ this non-symmetry disappears when one considers the Liouville completion of $V$ as constructed in Section \ref{sec:Liouville}.
\begin{prp}[Propositions \ref{prp:embedding} and \ref{prp:W}]
\label{prp:involution}
 The Liouville completion $(\widehat{W},d\lambda)$ of $(V,\omega_{\OP{FS}})$ admits a global symplectomorphism $I \colon (\widehat{W},d\lambda) \to (\widehat{W},d\lambda)$ which is of order two, fixes the Whitney immersion set-wise, while it interchanges the two sheets at the double-point. Moreover, the involution $I$ interchanges the Hamiltonian isotopy classes of the Clifford and Chekanov tori.
\end{prp}
As shown by Chekanov \cite{Chekanov:LagrangianTori}, we cannot find a similar symplectomorphism defined on on all of $(\CP^2,\omega_{\OP{FS}})$ or $(B^4,\omega_0);$ the Clifford and Chekanov tori cannot be interchanged by a globally defined symplectomorphism there.

\subsection*{Acknowledgements}
I am grateful for having a very stimulating environment at the University of Cambridge where the main ideas leading to this paper were born; thanks to Ivan Smith, Jack Smith, Dmitry Tonkonog, and Renato Vianna for useful discussions and a shared interest in Lagrangian tori. I would also like to thank my former Ph.~D.~advisor Tobias Ekholm for introducing these questions to me during my thesis, and for stressing the importance of the technique of neck stretching in symplectic topology; it is obvious that the latter tool has played a very crucial role here. The author is supported by the grant KAW 2016.0198 from the Knut and Alice Wallenberg Foundation

\section{Properties of the Lagrangian fibration}
\label{sec:LagrangianFibrations}

Here we investigate crucial properties of the family
$$ \Pi_s \colon \widetilde{V} \to (-1,1) \times (0,+\infty), \:\: s \in (0,\pi/2), $$
of fibrations. We will see that the unique singular fibre is the immersed Lagrangian sphere $L_{\OP{Wh}}(s) \coloneqq \Pi_s^{-1}(0,1),$ while all other fibres are smoothly embedded Lagrangian tori.

We start by treating the Lagrangian condition of the fibres, which will turn out to be a direct consequence of the following statement.
\begin{lma}
\label{lma:CharDist}
For any embedded path $\gamma \subset \C$ and $u_1 \in (-1,1),$ the intersection
$$\widetilde{f}^{-1}(\gamma) \cap \{ \|\widetilde{z}_1\|^2-\|\widetilde{z}_2\|^2=u_1\} \subset (B^4,\omega_0)$$
is a Lagrangian submanifold outside of the origin. (The origin is in general a singular point.)
\end{lma}
\begin{proof}
The characteristic distribution of the hypersurface $\{ \|\widetilde{z}_1\|^2-\|\widetilde{z}_2\|^2=u_1\},$ i.e.~the line field being the kernel of the restriction of $\omega_0,$ is generated by the vector field $\frac{d}{dt}(e^{it}\widetilde{z}_1,e^{-it}\widetilde{z}_2).$ (This can be checked e.g.~by computing the Euclidean gradient of the function $\|\widetilde{z}_1\|^2-\|\widetilde{z}_2\|^2,$ which is equal to $-2i\frac{d}{dt}(e^{it}\widetilde{z}_1,e^{-it}\widetilde{z}_2).$) Since $\widetilde{f}^{-1}(\gamma) \cap \{ \|z_1\|^2-\|z_2\|^2=u_1\}$ is foliated by closed integral curves of the aforementioned characteristic distribution away from the origin, the claim now follows.
\end{proof}
Below we will see that the fibres are compact subsets of $\CP^2 \setminus \ell_\infty.$ The topology of the fibres can then be investigated by hand without too much difficulty; $L_{\OP{Wh}}(s)$ is an immersed sphere having a transverse self-intersection whose Whitney self-intersection number is equal to $+1,$ while the smooth fibres all are tori. (Recall that any smooth Lagrangian foliation by closed surfaces must consist of torus leaves by a version of the Arnol'd--Liouville theorem.)
\begin{lma}
\label{lma:compactfibres}
\begin{itemize}
\item The fibre $\Pi_s^{-1}(u_1,u_2),$ $(u_1,u_2) \neq (0,1),$ is a smooth and compact Lagrangian torus living inside $\widetilde{V} \subset (B^4,\omega_0).$ In fact these tori with fixed $u_1=u_1^0 \neq 0$ provide a smooth foliation of the noncompact hypersurface $\Sigma_{u_1^0} \cap \widetilde{V},$ where
$$ \Sigma_{u_1^0} \coloneqq \{\|\widetilde{z}_1\|^2-\|\widetilde{z}_2\|^2= u_1^0\} \cap B^4$$
is an open solid torus.
\item The fibre $L_{\OP{Wh}}(s)=\Pi_s^{-1}(0,1)$ is a Lagrangian sphere with a single transverse double point of Whitney self-intersection number equal to $+1.$
\end{itemize}
\end{lma}
\begin{proof}
The Lagrangian condition was proven in Lemma \ref{lma:CharDist}.

The fibres are compact subsets of $B^4,$ since the Lefschetz fibration $\widetilde{f}(\widetilde{z}_1,\widetilde{z}_2)=\frac{\widetilde{z}_1\widetilde{z}_2}{1-\|\widetilde{z}_1\|^2-\|\widetilde{z}_2\|^2}$ is proper when restricted to any subset $\Sigma_{u_1^0}.$ Indeed, for any sequence $\mathbf{z}(k) \coloneqq (\widetilde{z}_1(k),\widetilde{z}_2(k)) \in B^4$ inside $\Sigma_{u_1^0}$ which converges to a point in $\partial D^4$ must satisfy the positive lower bounds
$$\epsilon< \|\widetilde{z}_1(k)\|^2,\|\widetilde{z}_2(k)\|^2 < 1, \:\: k \gg 0,$$
for some $\epsilon>0$ (here we use $u_1^0 \in (-1,1)$). However, such a sequence is now seen to map to an unbounded sequence $\widetilde{f}(\mathbf{z}(k))\in\C,$ since the numerators of $\widetilde{f}(\mathbf{z}(k))$ are positive and bounded from below while the denominators tend to $0.$
\end{proof}

\subsection{Action properties}
\label{sec:ActionProperties}

Recall that $\lambda_{\OP{std}}=x_1dy_1+x_2dy_2$ is the standard Liouville form defined on $(B^4,\omega_0=d\lambda_{\OP{std}}) \supset \widetilde{V}.$ We start with the following action computation for the immersed sphere.
\begin{lma}
\label{lma:whitneyaction}
The Lagrangian sphere $L_{\OP{Wh}}(s)=\Pi_s^{-1}(0,1)$ has an action difference equal to $\Delta=s \in (0,\pi/2)$ at the two preimages of its double point, where the action is computed by taking a primitive of the pullback of the Liouville form $\lambda_{\OP{std}}.$
\end{lma}
We proceed to investigate the symplectic action of the torus fibres. But first, we need to fix the choice of a basis of each such fibre.
\begin{lma}
\label{lma:basis0}
\begin{enumerate}
\item For any Lagrangian torus fibre $L=\Pi_s^{-1}(u_1,u_2)$ there is a canonical choice of generator $\mathbf{e}_0 \in H_1(L)$ of $\ker(H_1(L) \to H_1(V)) \cong \Z$ induced by the inclusion of a fibre, which is determined uniquely by the requirement that
$$\int_{\mathbf{e}_0}\lambda_{\OP{std}} = \pi \cdot u_1$$
is satisfied. Moreover, the Maslov class evaluates to $\mu_L^{\C^2}(\mathbf{e}_0)=0$ on this element.
\item In the case when $L \subset V \setminus C_{\OP{nodal}}$ moreover is satisfied, i.e.~when $u_2 \neq 1,$ then any relative cycle $(D,\partial D) \to (V,L)$ with $[\partial D]=\mathbf{e}_0$ has the property that $D \bullet \ell_i=(-1)^i$ for each of the two lines $f^{-1}(0)=\ell_1 \cup \ell_2$ in the nodal conic.
\end{enumerate}
\end{lma}
\begin{proof}
The statements are straight forward to check. We simply note that $\int_{\mathbf{e}_0} \lambda_{\OP{std}}=\pi\cdot u_1$ holds if we represent the class $\mathbf{e}_0$ by a closed curve of the form $\theta \mapsto (e^{i\theta}a,e^{-i\theta}b) \in B^4$ for suitable $a,b, \in \C$ satisfying $(a,b) \in L,$ and thus $\|a\|^2-\|b\|^2=u_1.$
\end{proof}
The choice of the homology class $\mathbf{e}_0$ can be seen to vary continuously with the fibres. In the following manner, we then proceed to extend $\mathbf{e}_0$ to (a not globally defined) basis $\langle \mathbf{e}_0,\mathbf{e}_1 \rangle = H_1(L)$ for the fibres, where $\mu_L^{\C^2}(\mathbf{e}_1)=2.$ Recall that, as shown in \cite[Section 4.3]{Symington:FourDimensions}, the bundle of standard tori is nontrivial due to the presence of the fibre with a nodal singularity. For that reason, no global and continuous choice of basis exists; also see Remark \ref{rmk:monodromy} below.
\begin{lma}
\label{lma:homologybasis}
In the following manner we can uniquely specify classes $\mathbf{e}_1 \in H_1(L)$ of the torus fibres such that $\langle \mathbf{e}_0,\mathbf{e}_1 \rangle$ is a basis: $\mathbf{e}_1$ satisfies the property that, for any relative cycle $(D,\partial D) \to (B^4,L)$ with $[\partial D]=\mathbf{e}_1,$ we have:
\begin{enumerate}
\item $D \bullet C=1;$ and
\item \begin{enumerate}
\item In the case $u_2<1$: we have $D \bullet \ell_i=0$ for each of the two lines $\ell_i \subset f^{-1}(0),$ $i=1,2,$ in the nodal conic,
\item in the case $u_2>1$ and $u_1 \ge 0$: we have $D \bullet \ell_1=1$ and $D \bullet \ell_2=0;$ while in the case $u_2>1$ and $u_1 < 0$: we have $D \bullet \ell_2=1$ and $D \bullet \ell_1=0.$
\end{enumerate}
\end{enumerate}
The Maslov class satisfies $\mu_L^{\C^2}(\mathbf{e}_1)=2$ in all cases above.
\end{lma}
\begin{proof}
The proof consists of an explicit verification, and is left to the reader.
\end{proof}
An immediate consequence of the above two lemmas is that a Lagrangian torus fibre $\Pi_s^{-1}(u_1,u_2)$ is monotone if and only if $u_1 = 0.$

\begin{rmk}
\label{rmk:monodromy}
In the complement of the ray $\Pi_s^{-1}(0,u_2),$ $u_2>1,$ of torus fibres, the basis determined by the above lemmas can be seen to coincide and vary continuously. However, the basis vector $\mathbf{e}_1$ is not unambiguously defined over the monotone tori $\Pi_s^{-1}(0,u_2),$ $u_2>1,$ of Clifford type; for these tori there are the two choices induced by the continuous extension for the tori $u_1 > 0$, and another one being the continuous extension of the basis for the tori with $u_1<0.$ Nevertheless, since these are monotone tori, and since the two different choices of basis differ by a cycle of Maslov index zero, this ambiguity is irrelevant as far as computations of the symplectic action are concerned.
\end{rmk}

Using the basis constructed above we are now ready to describe the global behaviour of the symplectic actions of the different fibres.
\begin{lma}
\label{lma:action}
The symplectic action of a torus $\Pi_s^{-1}(u_1,u_2)$ satisfies
$$ \int_{\mathbf{e}_0} \lambda_{\OP{std}}=\pi\cdot u_1 \:\:\: \text{and} \:\:\: \int_{\mathbf{e}_1} \lambda_{\OP{std}}=A_s(u_1,u_2)$$
for a function
$$A_s \colon (-1,1) \times (0,+\infty) \to (0,\pi/2),$$
which is smooth in all variables $(s,u_1,u_2)$ and which satisfies:
\begin{enumerate}
\item For each fixed $s$ and $u_1,$ we have
$$A_s(\{u_1\} \times (0,+\infty))=(0,\pi(1-|u_1|)/2) \subset \R$$
\item $\partial_{u_2}A_s >0$ away from $(u_1,u_2)=(0,1),$ 
\item $A_s$ has a continuous extension to the compactification $[-1,1] \times [0,+\infty]$ for which the properties
\begin{enumerate}
\item $A_s(u_1,+\infty)=\pi(1-|u_1|)/2,$
\item $A_s(u_1,0)=0,$
\item $A_s(\pm 1,u_2)=0,$ and
\end{enumerate}
\item $\lim_{s \to 0} A_s(u_1,1)=0$ while $\lim_{s \to \pi/2}A_s(u_1,1)=\pi(1-|u_1|)/2.$
\end{enumerate}
See Figure \ref{fig:action}.
\end{lma}

\begin{figure}[htp]
\begin{center}
\vspace{6mm}
\labellist
\pinlabel $A_s$ at 84 66
\pinlabel $\pi/2$ at 100 45
\pinlabel $\pi u_1$ at 182 6
\pinlabel $-\pi$ at 9 -2
\pinlabel $\pi$ at 156 -2
\endlabellist
\includegraphics{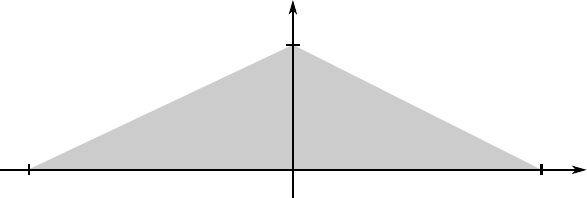}
\vspace{3mm}
\caption{ The image of the values $(\pi u_1,A_s)$ of the symplectic action class evaluated on the pair $(\mathbf{e}_0,\mathbf{e}_1)$ of basis vectors of $H_1(\Pi_s^{-1}(u_1,u_2)).$}
\label{fig:action}
\end{center}
\end{figure}
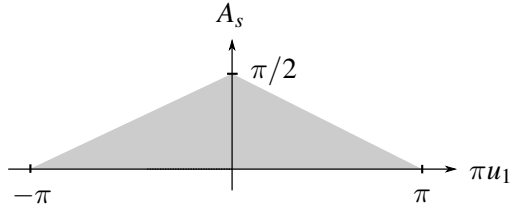

\begin{proof}
By continuity, it suffices to establish the claims whenever $u_1 \neq 0.$ In fact, the case $u_1=0$ is easy to check by hand since the fibration is sufficiently explicit for those parameters. Furthermore, one can argue by symmetry and restrict to the case $u_1>0.$ We continue to exhibit these torus fibres with additional care.

Recall that the one-parameter family of tori for fixed $s \in (0,\pi/2)$ and $u_1 \neq 0$ provides a foliation of the hypersurface
$$ \Sigma_{u_1} \coloneqq \{\|\widetilde{z}_1\|^2-\|\widetilde{z}_2\|^2=u_1\} \cap \widetilde{V}$$
by Lemma \ref{lma:compactfibres}, where $\Sigma_{u_1}$ is a solid torus. Further, we exhibit an explicit foliation of this solid torus by symplectic discs
$$ D_\theta(u_1)\coloneqq\{(e^{i\theta}\sqrt{u_1+\|z\|^2},z); \:\: 2\|z\|^2+u_1<1 \} \subset B^4$$
of total symplectic area equal to $(\pi/2)(1-u_1).$ (Recall that we here are considering the case $u_1>0$.) Also, note that the smooth conic $\widetilde{C} \coloneqq \widetilde{f}^{-1}(1)$ intersects each symplectic disc $D_\theta(u_1)$ transversely in a single point $(e^{i\theta}\sqrt{u_1+t^2},e^{-i\theta}t)$ for a unique value of $t \in (0,\sqrt{(1-u_1)/2}).$

The one-parameter family of Lagrangian torus fibres $\Pi_s^{-1}(u_1,\cdot)$ for fixed $s$ and $u_1$ can be seen to foliate the solid torus, and also to intersect each symplectic disc $D_\theta(u_1)$ in a foliation by simple closed curves encircling the unique intersection point $D_\theta(u_1) \cap \widetilde{C} = \{p_\theta(u_1)\}\subset D_\theta(u_1).$ (A Lagrangian obviously cannot have a full tangency to a symplectic leaf, and hence it is automatically transverse to each such symplectic disc when considered inside the solid torus.) Since $A_s(u_1,u_2)$ simply is the area inside the symplectic disc bounded by the curve $D_\theta(u_1) \cap \Pi_s^{-1}(u_1,u_2),$ Properties (1), (2) and (3) can now be seen to follow.

(4): This is a consequence of Property (2) satisfied by the family $\Psi_s$ of diffeomorphisms in the construction of $\Pi_s$ given in Section \ref{sec:LagrangianFibrationsIntro}. We argue as above by considering the symplectic area inside the discs foliating the solid torus $\Sigma_{u_1}.$ To that end we note that, for any $\epsilon>0,$ there exists a sufficiently small neighbourhood $U \subset \C$ of either of the subsets $[0,1]$ (the case $s=0$) or $(-\infty,0]$ (the case $s=\pi/2$), such that the discs $D_\theta(u_1)$ intersected with $f^{-1}U$ may be assumed to be of symplectic area bounded from above by $\epsilon>0$ independently of $u_1 \in (-1,1).$
\end{proof}

\subsection{Proof of Proposition \ref{prp:mainprp}}
\label{sec:mainprpproof}
The Lagrangian property for the fibres was shown in Lemma \ref{lma:compactfibres} in the beginning of this section. The claims concerning the action in (1) is a consequence of Lemma \ref{lma:whitneyaction}. The topological considerations can be investigated by hand.

For Property (2.a) we refer to the work of A. Gadbled \cite{Gadbled:OnExotic} concerning the different guises of the Chekanov torus.

Property (2.b) is finally shown with the below lemma, which finishes the proof.\qed

\begin{lma}
Consider any Lagrangian torus fibre $\Pi_s^{-1}(u_1,u_2)$ which is \emph{not} weakly exact, i.e.~with $u_1 \neq 0.$ For any $s' \in (0,\pi/2)$ this Lagrangian torus is Hamiltonian isotopic inside $(V,\omega_{\OP{FS}})$ to a unique torus fibre of the form $\Pi_{s'}^{-1}(u_1,u_2').$
\end{lma}
\begin{proof}
The Hamiltonian isotopy is constructed by hand, by considering a suitable family of Lagrangian tori contained inside the solid torus
$$\Sigma_{u_1} = \{ \|\widetilde{z}_1\|^2-\|\widetilde{z}_2\|^2=u_1\} \cap \widetilde{V}, \:\: u_1 \neq 0,$$
all which can be taken to be standard fibres for different values of $s \in (0,\pi/2).$

For the uniqueness, it is sufficient to note that any Hamiltonian isotopy must preserve the basis element $\mathbf{e}_0$ as constructed in Lemma \ref{lma:basis0}; recall that this is a preferred generator of $\ker(H_1(\T^2) \to H_1(V)) \cong \Z$ induced by the inclusion of the torus. The action computation in Lemma \ref{lma:action} thus implies that any Hamiltonian isotopy connecting two fibres $\Pi_s^{-1}(u_1,u_2)$ and $\Pi_{s'}^{-1}(u_1',u_2')$ must satisfy $u_1=u_1'.$ The uniqueness of the parameter $u_2'$ is also a consequence of Lemma \ref{lma:action}.
\end{proof}

\section{Standard neighbourhood of a sphere with self-intersection number $+1$}
\label{sec:SelfPlumbing}

In this section we give a careful description of the symplectic geometry of the standard neighbourhood of a Lagrangian sphere with a single transverse double point. In dimension $2k$ there are two different such Lagrangian spheres up to diffeomorphism; the two cases are determined by their Whitney self-intersection number which can be either $I=\pm 1.$ Note that the situation is different in odd dimensions, where the local model for a Lagrangian sphere with a single transverse self-intersection is unique. Here we are only interested in the case when $k=1$ and the self-intersection number is equal to $+1.$ This is the only sphere which appears as an isolated singular fibre in a Lagrangian \emph{torus} fibration.

We present the neighbourhood of the immersed sphere as a self-plumbing of the cotangent bundle of an embedded sphere. We moreover show that it naturally carries the structure of a Liouville domain, which thus can be completed to a Liouville manifold. We also construct a Lagrangian torus fibrations on this space, together with an induced symplectic embedding of $(V,\omega_{\OP{FS}})$ that preserves the torus fibres.

\subsection{The self-plumbing of the cotangent bundle of a sphere}
We commence with construction of the symplectic manifold of interest. First, consider the standard symplectic unit cotangent bundle
$$ DT^*D^2 \coloneqq \{\|\mathbf{q}\| \le 1, \: \|\mathbf{p}\| \le 1\} \subset (T^*\R^2,d\lambda_{\R^2})$$
of the unit disc. This is a smooth manifold with boundary with corners, where
\begin{align*}
& \partial DT^*D^2 = \partial_VDT^*D^2 \: \cup \: \partial_H DT^*D^2,\\
& \partial_V DT^*D^2 \coloneqq DT_{\partial D^2}^*D^2 = \{ \|\mathbf{q}\| = 1, \| \mathbf{p}\| \le 1 \},\\
& \partial_H DT^*D^2 \coloneqq ST^*D^2 = \{ \|\mathbf{q}\|\le 1, \| \mathbf{p}\|=1 \}.
\end{align*}
It is also convenient to use the canonical identification with the corresponding subset
\begin{gather*}
DT^*D^2 \hookrightarrow \C^2,\\
(\mathbf{q},\mathbf{p}) \mapsto i\mathbf{q}+\mathbf{p},
\end{gather*}
of the complex plane.

We also need the unit cotangent bundle
$$DT^*(S^1 \times [-1,1]) \coloneqq \{ \| \mathbf{p}_{\theta,q} \| \le 1 \} \subset (T^*(S^1 \times [-1,1]),d\lambda_{S^1 \times [-1,1]})$$
of the cylinder, with the coordinate $\mathbf{p}_{\theta,q}=(p_\theta,p_q)$ on the cotangent fibres induced by the standard coordinates $(\theta,q) \in S^1 \times [-1,1].$ Note that we here have made implicit use of the flat product metric on $S^1 \times [-1,1]$ when speaking about the unit co-disc bundle.

For some small $\epsilon>0$ we consider the open neighbourhoods
\begin{align*}
& \{\mathbf{q}=(q_1,q_2)=r(\cos \theta,\sin \theta), \: r \in [1-\epsilon,1+\epsilon]\} \subset DT^*D_{1+\epsilon}^2,\\
& \{\mathbf{p}=(p_1,p_2)=r(\cos \theta,\sin \theta), \: r \in [1-\epsilon,1+\epsilon]\} \subset D_{1+\epsilon}T^*D^2,
\end{align*}
of the pieces $\partial_V DT^*D^2$ and $\partial_H DT^*D^2$ of the boundary, respectively. There are symplectic inclusions
$$\phi_V \colon \{\mathbf{q}=r(\cos \theta,\sin \theta), \: r \in [1-\epsilon,1+\epsilon]\} \hookrightarrow D_2T^*(S^1 \times [-1-\epsilon,-1+\epsilon])$$
defined by
\begin{align*}
& (\mathbf{q},\mathbf{p}) \mapsto \\
& ((\theta,-1+(r-1)),(p_\theta=r(-\sin \theta,\cos\theta)\bullet \mathbf{p},p_q=(\cos \theta,\sin\theta)\bullet \mathbf{p})),
\end{align*}
as well as
$$\phi_H \colon \{\mathbf{p}=r(\cos \theta,\sin \theta), \: r \in [1-\epsilon,1+\epsilon]\} \hookrightarrow D_2T^*(S^1 \times [1-\epsilon,1+\epsilon])$$
defined by
\begin{align*}
& (\mathbf{q},\mathbf{p}) \mapsto \\
& ((\theta+\pi/2,1-(r-1)),(p_\theta=-r(-\sin \theta,\cos\theta)\bullet \mathbf{q},p_q=(\cos \theta,\sin\theta)\bullet \mathbf{q})).
\end{align*}
Note that the restrictions
\begin{align*}
& \phi_V|_{\partial_V DT^*D^2} \xrightarrow{\cong } DT_{\{q=-1\}}^*(S^1 \times [-1,1]),\\
& \phi_H|_{\partial_H DT^*D^2} \xrightarrow{\cong } DT_{\{q=1\}}^*(S^1 \times [-1,1]),
\end{align*}
are diffeomorphisms

Using the above symplectomorphisms it is possible to perform the gluing
$$W=\frac{DT^*D^2 \: \sqcup \: D_\epsilon T^*(S^1 \times [-1,1])}{x \sim \phi_V(x), \:\: x \sim \phi_H(x)}, \:\: 0<\epsilon\ll 1,$$
producing a symplectic manifold $(W,\omega).$ This symplectic manifold is naturally identified with the self-plumbing of a two-sphere. In Section \ref{sec:Liouville} below we will exhibit a natural primitive of the symplectic form, giving it the structure of a Liouville domain.

Observe that
$$L_0 \coloneqq 0_{D^2} \: \cup \: DT^*_{(0,0)} D^2 \: \cup \: 0_{S^1 \times [-1,1]} \subset (W,\omega)$$
is the Lagrangian immersion of a sphere with one transverse double point
$$\{ \mathbf{q}=0=\mathbf{p} \} = 0_{D^2} \cap DT^*_{(0,0)}D^2.$$
The Whitney self-intersection number of this sphere can be computed to be equal to $+1$. 

There is a symplectomorphism
\begin{gather*}
I_1 \colon (DT^*D^2,d\lambda_{\R^2}) \xrightarrow{\cong} (DT^*D^2,d\lambda_{\R^2}),\\
I_1((q_1,q_2),(p_1,p_2))=((-p_2,p_1),(q_2,-q_1)),
\end{gather*}
that satisfies $I_1^2=\id_{DT^*D^2}$. Likewise, the symplectomorphism
\begin{gather*}
I_2 \colon (DT^*(S^1 \times [-1,1]),d\lambda_{S^1 \times [-1,1]}) \xrightarrow{\cong} (DT^*(S^1 \times [-1,1]),d\lambda_{S^1 \times [-1,1]}),\\
 I_2((\theta,q),(p_\theta,p_q))=((\theta,-q),(p_\theta,-p_q))
\end{gather*}
is of order two, i.e.~$I_2^2=\id_{DT^*(S^1 \times [-1,1])}$. Since one can check that
$$I_2 \circ \phi_V=\phi_H \circ I_1 \:\:\:\:\text{and}\:\:\:\: I_2 \circ \phi_H=\phi_V \circ I_1.$$
is satisfied, we obtain an induced symplectomorphism $I$ of $(W,\omega)$ whose unique fixpoint is the origin $(0,0) \in DT^*D^2.$ To summarise:
\begin{prp}
\label{prp:involution2}
The induced symplectomorphism
$$ I \colon (W,\omega) \to (W,\omega) $$
is of order two, fixes the immersed Lagrangian sphere $L_0$ setwise, i.e.~$ I(L_0) = L_0,$ while reversing its orientation as well as the two sheets at its unique double point.
\end{prp}

Weinstein's symplectic neighbourhood theorem \cite{Weinstein:Lagrangian} (also, see \cite{McDuff:Introduction}) shows that any Lagrangian submanifold $L$ has a neighbourhood symplectomorphic to a neighbourhood of the zero-section of $(T^*L,d\lambda_L)$ of the cotangent bundle equipped with the standard symplectic form, while moreover identifying the Lagrangian with the zero-section. It is well-known that the result generalises to establish a standard symplectic neighbourhood also of a Lagrangian immersion with transverse double-points; see e.g.~\cite[Proposition 4.8]{Symington:FourDimensions}. In particular, it is the case that
\begin{prp}
\label{prp:Weinstein}
For any two-dimensional Lagrangian immersion $L \subset (W',\omega')$ of a sphere with a single transverse double-point and whose Whitney self-intersection number equal to $+1$, there is a symplectic embedding
$$ \phi \colon (U,\omega') \hookrightarrow (W,\omega)$$
of a neighbourhood $U \subset W'$ of $L$ which moreover satisfies $\phi(L)=L_0.$
\end{prp}
Later in Proposition \ref{prp:embedding} this result is extended to also preserve locally defined Lagrangian torus fibrations.

\subsection{Extending the neighbourhood to a complete Liouville manifold}
\label{sec:Liouville}
Here we construct a suitable Liouville form defined on all of $(W,\omega)$ by interpolating between natural Liouville forms on the pieces $DT^*D^2$ and $DT^*(S^1 \times [-1,1]).$ First, consider the Liouville form
$$\lambda_0\coloneqq\frac{1}{2}\sum_{i=1}^2(p_i\,dq_i-q_i\,dp_i)=\lambda_{\R^2}-\frac{1}{2}d\left(\sum_{i=1}^2q_ip_i\right)$$
defined on $T^*\R^2$ and which clearly satisfies $I_1^*\lambda_0=\lambda_0.$ 

Fix a smooth function $h \colon [-1,1] \to \R$ that satisfies
\begin{itemize}
\item $h(-q)=-h(q),$
\item $h(q)=2+q$ near $q=-1$ and $h(q)=-2+q$ near $q=1,$
\item $h'(q)<2$ for $q \in [-1,1],$ $h'(q) \ge 0$ for $q \le -2/3,$ $h'(q)=0$ for $q \in [-2/3,-1/3],$ and $h'(q) \le 0$ for $q \in [-1/3,0].$
\end{itemize}
In particular, we note that $h(q)$ is necessarily non-vanishing outside the subset $[-1/3,1/3],$ while it is constant on $\{ q;\: 1/3 \le |q| \le 2/3\}.$

Then, using the function $h,$ we construct the Liouville form
$$\lambda\coloneqq\lambda_{S^1 \times [-1,1]}-\frac{1}{2} d(h(q)p_q)=p_qdq+p_\theta d\theta-\frac{1}{2}d(h(q)p_q)$$
defined on $T^*(S^1 \times [-1,1]),$ which thus also clearly satisfies $I_2^*\lambda_{S^1 \times [-1,1]}=\lambda_{S^1 \times [-1,1]}.$ The Liouville vectorfield induced by $\lambda$ can be seen to be given by
\begin{equation}
\label{eq:liouville}
 \zeta_\lambda \coloneqq p_\theta \partial_{p_\theta}+p_q\left(1-\frac{1}{2}h'(q)\right)\partial_{p_q}+\frac{1}{2}h(q)\partial_q,
\end{equation}
and we use $\phi^t_{\lambda}$ to denote the corresponding Liouville flow. (This flow is well-defined at least for negative times $t \le 0.$)

\begin{lma}
We have
\begin{gather*}
\phi_V^*\lambda=\lambda_0,\\
\phi_H^*\lambda=\lambda_0,
\end{gather*}
and these Liouville form thus combines to a smooth Liouville form $\lambda$ on all of $(W,\omega=d\lambda)$ invariant under the symplectomorphism $I.$
\end{lma}
\begin{proof}
Observe that
\begin{align*}
& dq=dr=\frac{q_1\,dq_1+q_2\,dq_2}{\sqrt{q_1^2+q_2^2}}=\cos{\theta}\,dq_1+\sin{\theta}\,dq_2,\\
& d\theta=\frac{-q_2\,dq_1+q_1\,dq_2}{q_1^2+q_2^2}=r^{-1}(-\sin{\theta}\,dq_1+\cos{\theta}\,dq_2),
\end{align*}
from which we deduce that
$$\phi_V^*\lambda_{S^1 \times [-1,1]}=\sum_{i=1}^2p_i\,dq_i=\lambda_{\R^2}.$$
Combining this with the relation $h(q)p_q \circ \phi_V=\sum_{i=1}^2q_ip_i$ we deduce that
$$\phi_V^*\lambda=\lambda_{\R^2}-\frac{1}{2}d\left(\sum_{i=1}^2q_ip_i\right)=\lambda_0$$
as sought.

Since both $\lambda_0$ and $\lambda$ are invariant under $I_1$ and $I_2,$ respectively, and since $I_2 \circ \phi_H=\phi_V \circ I_1$ the claim
$$\phi_H^*\lambda=\lambda_0$$
is now a direct consequence as well.
\end{proof}

Next we construct the smooth function
$$\rho \colon W \to \R_{\ge 0}$$
specified by the following:
\begin{itemize}
\item Inside $DT^*D^2$ it is given by
$$\rho(\mathbf{q},\mathbf{p})=(p_1q_1+p_2q_2)^2+(q_1p_2-q_2p_1)^2;$$
and
\item Inside $DT^*(S^1 \times [-1,1])$ it is given by
$$\rho((\theta,q),(p_\theta,p_q))=g(q)^2p_q^2+p_\theta^2$$
for a smooth function $g(|q|)>0$ defined as follows: $g(q)=|h(q)|$ outside of $[-1/3,1/3],$ while $g(q) \equiv h(-2/3)$ is constant for $q \in [-2/3,2/3].$ Observe that $g(q)$ thus in particular is non-vanishing.
\end{itemize}
The function $\rho$ is smooth on $W.$ Furthermore, it is the case that
\begin{lma}
\label{lma:ContactBoundary}
The Liouville vector field $\zeta_\lambda$ on $(W,\omega=d\lambda)$ corresponding to the primitive $\lambda$ satisfies the following properties in a neighbourhood of $L_0$:
\begin{enumerate}
\item The Liouville form $\lambda$ vanishes on $TL_0,$ and its backwards flow satisfies
$$ \bigcap_{n \in \Z_{\ge 0}} \phi_\lambda^{-n}(W) = L_0;$$ 
\item The form $\lambda,$ and hence $\zeta_\lambda,$ as well as the function $\rho$ are all preserved by the symplectomorphism $I;$ and
\item $\rho^{-1}(0)=L_0$, and the vector field $\zeta_\lambda$ satisfies $d\rho(\zeta_\lambda)=f\cdot\rho$ for some function $f \colon W \to \R_{>0}$ which is constantly equal to $f \equiv 2$ outside of $\{(\theta,q,p_\theta,p_q)\; \: q \in [-1/3,1/3]\} \subset DT^*(S^1 \times [-1,1]).$
\end{enumerate}
In particular, for all $a>0$ sufficiently small, the level-set $Y_a\coloneqq\rho^{-1}(a) \subset (W,d\lambda)$ is a smooth contact-type hypersurface with induced contact form $\alpha_a\coloneqq\lambda|_{TY_a},$ and the restriction
$$ I|_{Y_a} \colon (Y_a,\alpha_a) \xrightarrow{\cong} (Y_a,\alpha_a) $$
is a strict contactomorphism.
\end{lma}
\begin{proof}
(1): This follows from Property (3) (to be proven below).

(2): We leave this to the reader to check.

(3): First we establish the claim $\rho^{-1}(0)=L_0.$ Inside $DT^*(S^1 \times [-1,1])$ it is clear that $\rho$ vanishes precisely along the zero section. Further, by the definition of $\rho,$ inside $DT^*D^2$ we have $\rho=0$ if and only if the two vectors $(q_1,q_2),(p_1,p_2) \in \R^2$ are simultaneously orthogonal and collinear. In other words $\rho=0$ if and only if either $(q_1,p_2)=0$ or $(p_1,p_2)=0$ in that subset. This shows the claim.

The computation of $d\rho(\zeta_\lambda)$ is left to the reader. Inside the subset $DT^*D^2$ we use the expression
$$\zeta_\lambda=\frac{1}{2}\sum_{i=1}^2(q_i\partial_{q_i}+p_i\partial_{p_i})$$
while in $DT^*(S^1 \times [-1,1])$ we use Equation \eqref{eq:liouville}.
\end{proof}

Finally, we use the Liouville flow $\phi^t_{\lambda}$ generated by $\zeta_\lambda$ in order to produce the following completion of the symplectic manifold $(W,d\lambda).$ Consider the sub-level set $W_a \coloneqq \rho^{-1}[0,a]$, which is a Liouville manifold with a contact boundary by Lemma \ref{lma:ContactBoundary}. Attaching half of the corresponding symplectisation, i.e.
$$ ([0,+\infty) \times Y_a,d(e^t\alpha_a)),$$
along its boundary, there is a smooth extension of the Liouville form $\lambda$ on $W_a$ by $e^t\alpha_a$ on this cylindrical end. This produces a smooth Liouville form with a complete Liouville flow, and we denote by
$$ (\widehat{W},d\lambda)=(W_a,\lambda) \cup ([0,+\infty) \times Y_a,d(e^t\alpha_a))$$
the resulting complete Liouville manifold which contains $(W_a,\omega)$ as a subdomain. Recollecting the previously established results, we can conclude that
\begin{prp}
\label{prp:W}
\begin{enumerate}
\item The Liouville form $\lambda$ vanishes along $TL_0$ of the immersed sphere $L_0 \subset (\widehat{W},d\lambda)$, which thus is exact, and its backwards flow satisfies
$$ \bigcap_{n \in \Z_{\le 0}} \phi_\lambda ^{-n} (\widehat{W} \cap \{ t \le N \}) = L_0$$
for any $N \ge 0.$
\item There is a smooth function $\hat{\rho} \colon \widehat{W} \setminus L_0 \to \R_{>0}$ uniquely defined by the property that $\hat{\rho}^{-1}(a)=Y_a$ together with $\hat{\rho}\circ\phi^t_\lambda=e^t\cdot \hat{\rho}$ for all $t \in \R$ (in particular, the level-sets $(\hat{\rho}^{-1}(a),\lambda|_{T(\hat{\rho}^{-1}(a))})$ are hypersurfaces being of contact type for $\lambda$); and
\item The symplectomorphism $I \colon (W_a,d\lambda) \to (W_a,d\lambda)$ extends to an exact symplectomorphism of $(\widehat{W},d\lambda)$ of order two which fixes $L_0$ set-wise, preserves the Liouville form $\lambda,$ and which preserves each level set $\hat{\rho}^{-1}(s),$ $s \in [a,+\infty),$ (where it consequently acts by contactomorphism preserving the contact form $\lambda|_{T(\hat{\rho}^{-1}(s))}$).
\end{enumerate}
\end{prp}
\begin{proof}
The properties follow more or less directly from Lemma \ref{lma:ContactBoundary}, and by construction. For Property (3) we have to use that the Liouville flow of $\lambda$ is invariant under $I$ on $W_a \subset \widehat{W}$ by Part (2) of Lemma \ref{lma:ContactBoundary}, and that $\hat{\rho}^{-1}(a)=\rho^{-1}(a)=\partial W_{a}$ is fixed by $I.$ Hence we can smoothly extend $I$ to all of $(\widehat{W},d\lambda)$ by requiring that $I$ commutes with the Liouville flow of $\lambda,$ i.e.~that
$$I \circ \phi^t_{\lambda}= \phi^t_{\lambda} \circ I$$
is satisfied.
\end{proof}

\subsection{A singular Lagrangian torus fibration}
\label{sec:TorusFibration}

Following Symington's construction in \cite[Section 4.2]{Symington:FourDimensions} we consider the map
$$\pi=(\pi_1,\pi_2) \colon W_a \to \R^2$$
which is defined by
\begin{align*}
& \pi(\mathbf{q},\mathbf{p})=(q_1p_2-q_2p_1,q_1p_1+q_2p_2),\\
& \pi((\theta,q),\mathbf{p}_{(\theta,q)})=(p_\theta,g(|q|)p_q),
\end{align*}
for $(\mathbf{q},\mathbf{p})\in DT^*D^2$ and $((\theta,q),\mathbf{p}_{(\theta,q)})\in DT^*(S^1 \times [-1,1]),$ respectively. Here we have used the previously defined smooth function $g(q)>0$ that satisfies $g(q)=|\mp 2+q|$ near $q=\pm 1,$ and we further assume that $a>0$ is sufficiently small.
\begin{lma}
\label{lma:involution}
The map $\pi$ is smooth, inducing a singular Lagrangian torus fibration, the unique singular fibre of which is given by our previously constructed Lagrangian immersion
$$\pi^{-1}(0,0)=L_0$$
of a sphere. Moreover, denoting the reflection of the second coordinate in $\R^2$ by $R(u_1,u_2)=R(u_1,-u_2),$ we have
\begin{gather*}
\pi \circ I=R\circ \pi
\end{gather*}
and in particular $I$ preserves the fibres of $\pi$ setwise.
\end{lma}
\begin{proof}
We show that the fibres of $(q_1p_2-q_2p_1,q_1p_1+q_2p_2) \in \R^2$ are Lagrangian inside $T^*\R^2.$ The remaining claims are straight forward to check.

The Lagrangian condition is most easily seen by using the fact that any complex curve inside $\C^2$ becomes Lagrangian for the symplectic form
$$ \mathfrak{Re}(dz_1 \wedge dz_2)=dx_1 \wedge dx_2-dy_1 \wedge dy_2.$$
Indeed, we can set $z_1=q_1-iq_2$ and $z_2=p_1+ip_2,$ thus turning
$$\pi(\mathbf{q},\mathbf{p})=z_1z_2$$
into a holomorphic Lefschetz fibration.
\end{proof}

\begin{prp}
\label{prp:ExtendedFibration}
There exists a Lagrangian torus fibration $\hat{\pi} \colon \widehat{W} \to \R^2,$ where $\hat{\pi}$ is onto $\R^2$ and submersive outside of the origin, and for which $\hat{\pi}^{-1}(0)=L_0$ is the unique singular fibre. Furthermore, the fibration can be taken to satisfy
\begin{enumerate}
\item $\hat{\pi}|_{Y_a}=\pi,$ and for any $\mathbf{v} \in \pi(Y_a)=\{u_1^2+u_2^2=a\}$ we have
$$\hat{\pi}(\phi^t_\lambda(\hat{\pi}^{-1}(\mathbf{v})))=e^{t/2}\cdot \mathbf{v},\:\:t \ge 0,$$
and in particular $\hat{\rho}=\|\hat{\pi}\|^2$ holds inside $\hat{\rho}^{-1}(a,+\infty)=\widehat{W} \setminus W_a,$
\item $\hat{\pi}=\pi$ in some neighbourhood of $L_0,$ and
\item the Liouville flow $\phi^t_\lambda$ applied to any fibre of $\hat{\pi}$ is again Hamiltonian isotopic to a fibre of $\hat{\pi}.$
\end{enumerate}
\end{prp}
\begin{proof}
Recall that $a \equiv \rho|_{Y_a}\equiv \|\pi\|^2|_{Y_a}$ is satisfied by construction. In view of Part (2) of Proposition \ref{prp:W}, the Liouville flow applied to the family of tori $\pi^{-1}(\mathbf{v}) \subset Y_a$ for $\mathbf{v} \in \{ \|(u_1,u_2)\|^2 \equiv a\}$ produces a smooth torus fibration $f \colon \widehat{W} \setminus L_0 \to \R^2 \setminus \{0\}$ which coincides with $\pi$ when restricted to the hypersurface $Y_a.$ This torus fibration (defined only in the complement of $L_0$) can thus be made to satisfy Property (1) by construction.

What suffices is then to perform a suitable interpolation between these two fibrations. To that end, we argue as follows. First, using the fact that the two fibrations coincide along $Y_a$ by construction, we can use the classical Arnol'd--Liouville Theorem \cite[Theorem 2.3]{Symington:FourDimensions} in order to find a symplectomorphism $\phi_0$ defined inside $W_{a+\epsilon} \setminus W_{a-\epsilon}$ such that
\begin{itemize}
\item $\phi_0$ is the identity along $Y_a,$ and
\item $\pi \circ \phi_0=f.$
\end{itemize}
The construction of $\phi_0$ is standard; see the proof of Proposition \ref{prp:embedding} below for more details. 

The fact that $\phi_0$ is a symplectomorphism that restricts to the identity along $Y_a$ implies that the differential must satisfy $D\phi_0 = \id_{T\widehat{W}}$ along $T_{Y_a}\widehat{W}.$ Hence, after shrinking $\epsilon>0,$ a standard argument shows that $\phi_0=\phi^1_{H_t}$ holds for a Hamiltonian isotopy that again can be taken to satisfy $\phi^t_{H_t}|_{Y_a} =\id.$

After an appropriate cutoff of this Hamiltonian isotopy, we construct a symplectomorphism $\phi_1$ that satisfies
\begin{itemize}
\item $\phi_1=\id$ inside $W_{a+\epsilon} \setminus W_a$ while
\item $\phi_1=\phi_0$ inside $W_{a-\epsilon/2} \setminus W_{a-\epsilon}.$
\end{itemize}
This provides us with the sought interpolation of the two torus fibrations.

(3): One can readily find a path of Lagrangian fibres of $\hat{\pi}$ realising the same symplectic flux-paths as that induced by the Liouville flow applied to the given fibre. (In fact, above the subset $\{u_1^2+u_2^2\ge a \},$ the positive Liouville flow maps fibres to fibres by construction.) The result is then a consequence of Lemma \ref{lma:LiouvilleFlow}.

\end{proof}
Recall that Lagrangian fibrations with a unique singular fibre was constructed for the symplectic manifold $(V,\omega_{\OP{FS}})$ in Section \ref{sec:LagrangianFibrationsIntro}. The fact that the fibration constructed here has similar properties gives us a convenient way to construct a symplectic embedding of $(V,\omega_{\OP{FS}}) \cong (\CP^2 \setminus (\ell_\infty \cup C),\omega_{\OP{FS}})$ by utilising the Arnol'd--Liouville theorem.
\begin{prp}
\label{prp:embedding}
Given any $s \in (0,\pi/2)$ and $c>0,$ there exists a symplectic embedding
\begin{gather*}
\iota_s \colon(V,c\omega_0) \hookrightarrow (\widehat{W},d\lambda),\\
\iota_s(\Pi_s^{-1}(0,1))=L_0,
\end{gather*}
for which the following is satisfied:
\begin{enumerate}
\item Given an arbitrarily small neighbourhood $U \subset (-1,1) \times (0,+\infty)$ of $(0,1) \in (-1,1) \times (0,+\infty)$ (i.e.~the unique critical value of $\Pi_s$), we may moreover assume that $\iota_s$ maps fibres $\Pi_s^{-1}(p)$ for $p \notin U\setminus \{(0,1)\}$ to fibres of $\hat{\pi};$
\item For $c \gg 0$ sufficiently large, the image $\iota_s(V)$ projects to a starshaped subset $\hat{\pi}(\iota(V_s)) \subset \R^2.$
\end{enumerate}
\end{prp}
\begin{proof}
The Arnol'd--Liouville Theorem \cite[Theorem 2.3]{Symington:FourDimensions} together with its generalisation \cite[Proposition 4.8]{Symington:FourDimensions} to Lagrangian fibrations with nodal singular fibres, shows the existence of the embeddings.

More precisely, the generalised version of the Arnol'd--Liouville theorem provides us with a symplectomorphism $\iota_s \colon (O,c\omega_0) \hookrightarrow (\widehat{W},d\lambda)$ from a neighbourhood $O \subset V$ of the singular Lagrangian fibre $\Pi_s^{-1}(0,1) \subset V$ to a neighbourhood $\iota(O) \ni \pi^{-1}(0,0),$ which moreover
\begin{itemize}
\item maps the singular fibre $\Pi_s^{-1}(0,1)$ to the singular fibre $\pi^{-1}(0,0),$ and
\item maps all fibres of $\Pi_s$ to fibres of $\pi$ outside of some neighbourhood $\Pi_s^{-1}(U \setminus \{(0,1)\})$ as required.
\end{itemize}
To that end, we use the fact that the unique singular fibres $\Pi_s^{-1}(0,0)$ and $\pi^{-1}(0,0)$ of the two fibrations are `nodes' as defined in \cite[(4.3)]{Symington:FourDimensions}. (Observe that it is not possible to find a symplectomorphism that preserves also the fibres near the singular fibre in general; see \cite[Remark 3.9]{Symington:FourDimensions}.)

Next we must extend the map $\iota_s$ from $O \subset V$ to all of $V.$ This is a simple matter of applying the classical Arnol'd--Liouville theorem. Namely, for sufficiently small and simply connected $U_1 \subset (-1,1) \times (+\infty) \setminus \{0,1\}$ and $U_2 \subset \R^2 \setminus \{(0,0)\},$ it provides us with symplectic identifications of neighbourhoods $\Pi_s^{-1}(U_1)$ and $\hat{\pi}^{-1}(U_2)$ with neighbourhoods of the form $\T^2 \times A_i \subset T^*\T^2.$ The extension is then created by patching together these identifications. Recall that the identification supplied by the Arnol'd--Liouville theorem is canonical up to fibre-wise translations in $T^*\T^2$ and the \emph{discrete} action of $\OP{Gl}(2,\Z)$ by pull-backs of the corresponding linear diffeomorphisms of $\T^2.$ (The discreteness is crucial for this argument.)

We are left with showing Property (2). We show that if $\mathbf{v} \in \R^2$ is in the image of $\hat{\pi} \circ \iota_s,$ then necessarily $e^{-t}\mathbf{v},$ $t \ge 0,$ is in the image as well, from which the sought stare-shaped property follows.

For $c \gg 0$ the torus fibres above the complement of some given compact neighbourhood of $(0,1) \in (-1,1) \times (0,+\infty)$ may be assumed to map to torus fibres of $\hat{\pi}$ contained inside $\widehat{W} \setminus W_a.$ By Part (1) of Proposition \ref{prp:ExtendedFibration}, the forwards Liouville flow $\phi^t_{\lambda}$ preserves the fibres in the same subset. Furthermore, whenever $L \coloneqq \hat{\pi}^{-1}(\mathbf{v}) \subset \widehat{W} \setminus W_a,$ the image $\hat{\pi}(\phi^t_\lambda(L))$ is simply the radial rescaling $e^{t/2}\cdot\mathbf{v}$ by the same result.

 For this reason it now suffices to show that the Liouville flow preserves also the torus fibres of $\Pi_s,$ at least up to Hamiltonian isotopy. Indeed, the convexity properties satisfied by the values of the symplectic action on these fibres, as shown in Figure \ref{fig:action}, combined with Lemma \ref{lma:LiouvilleFlow} implies that all tori $\phi^{-t}_{\iota_s^*\lambda}(\Pi_s^{-1}(u_1,u_2))$ for small $t \ge 0$ again are Hamiltonian isotopic to fibres of $\Pi_s.$ From this it then follows that the image of $\hat{\pi}(\iota_s(V))$ is invariant under multiplication by $e^{-t}$ outside of the subset $\hat{\pi}(W_a)=\{\|u_1\|^2+\|u_2\|^2 \le a\} \subset \R^2$ as sought.
\end{proof}

\section{Pencils of pseudoholomorphic conics}
\label{sec:pencils}

In this section we assume that we are given a tame almost complex structure on $(\CP^2,\omega_{\OP{FS}})$ which coincides with the standard integrable complex structure $i$ near the divisor $\ell_\infty \subset \CP^2$ at infinity. A pseudoholomorphic {\bf line} inside $\CP^2$ is a pseudoholomorphic curve of degree one. One of the first examples of the power of the technique of pseudoholomorphic curves in symplectic topology was Gromov's result from \cite{Gromov:Pseudo} which shows that $\CP^2$ is foliated by pseudoholomorphic lines for any choice of tame almost complex structure.
\begin{thm}[Gromov \cite{Gromov:Pseudo}]
\label{thm:gromov}
For any tame almost complex structure, the pseudoholomorphic lines that pass through a given point $\pt \in \CP^2$ are embedded symplectic spheres which form the leaves of a smooth foliation of the complement $\CP^2 \setminus \{\pt\}$ of that point.
\end{thm}
A pseudoholomorphic {\bf conic} inside $\CP^2$ is a pseudoholomorphic curve of degree two. The adjunction formula, together with positivity of intersection \cite{McDuff:LocalBehaviour}, allows us to conclude that
\begin{lma}
\label{lma:conics}
A pseudoholomorphic conic is either: a smoothly embedded sphere; a nodal sphere consisting of the union of two different pseudoholomorphic lines; or a two-fold branched cover of a pseudoholomorphic line.
\end{lma}
\begin{proof}
We show that the only singularities are nodes and branch points, the rest follows from elementary applications of the adjunction formula and positivity of intersection.

Consider a line passing through a singular point as well as a smooth point on the conic (its existence follows from Gromov's result Theorem \ref{thm:gromov} concerning the classification of pseudoholomorphic lines). Unless the line is contained inside the conic, the singular point contributes with at least $+2$ to the algebraic intersection number (see \cite{McDuff:LocalBehaviour}), while the intersection at the smooth point contributes with $+1.$ Since a line and a conic intersect with algebraic intersection number $+2,$ positivity of intersection implies that the line must be contained inside the conic.
\end{proof}
Now we fix the two points $q_1=[1:0:0], q_2=[0:1:0] \in \ell_\infty$ at the line at infinity, together with the complex tangent vectors $v_i \subset T_{q_i}\CP^2,$ $i=1,2,$ to the two lines
\begin{gather*}
\ell_1\coloneqq\{ [Z_1:0:Z_3] \in \CP^2 \} \:\:\: \text{and} \:\:\: \ell_2\coloneqq\{ [0:Z_2:Z_3] \in \CP^2 \}
\end{gather*}
at the two respective points. Note that there is a Lefschetz fibration
\begin{gather*}
f \colon \CP^2 \setminus \ell_\infty \to \C,\\
f(z_1,z_2) = z_1z_2,
\end{gather*}
whose fibres are precisely those conics satisfying the specified tangencies $v_i \subset T_{q_i}\CP^2,$ $i=1,2.$ This Lefschetz fibration has a unique singular fibre $f^{-1}(0)$; this is the standard nodal conic $C_{\OP{nodal}} \subset \CP^2$ given as the union of the coordinate lines.

Here we show how Gromov's strategy for establishing a foliation by pseudoholomorphic lines extends to give an analogous result also for conics. Namely, for an arbitrary tame almost complex structure $J$ which is standard at infinity, there exists fibration $f_J$ by $J$-holomorphic conics having properties similar to the standard fibration $f.$

We first need to introduce a couple of notions. Denote by $\mathcal{M}_J$ the moduli space of $J$-holomorphic conics and $\mathcal{M}_J(v_1,v_2) \subset \mathcal{M}_J$ the subspace of conics satisfying the two tangencies $v_i,$ $i=1,2.$ For a smooth family $J_{\mathbf{s}},$ $\mathbf{s} \in I^k,$ of tame almost complex structures on $(\CP^2,\omega_{\OP{FS}}),$ all which are assumed to be standard near $\ell_\infty,$ we denote by $C^{\mathbf{s}}_{\OP{nodal}} \subset \CP^2$ the union of the two unique $J_{\mathbf{s}}$-holomorphic lines satisfying the tangencies $v_i \subset T_{q_i}\CP^2.$ (Recall Gromov's result Theorem \ref{thm:gromov}.) The lines $C^{\mathbf{s}}_{\OP{nodal}} \in \mathcal{M}_{J_{\mathbf{s}}}(v_1,v_2)$ can of course be considered to be a nodal conic. Further, the holomorphicity of $\ell_\infty$ implies that $C^{\mathbf{s}}_{\OP{nodal}}$ intersects the line at infinite transversely precisely in the two points $q_i.$ The node $x^{\mathbf{s}}_{\OP{nodal}} \in \CP^2$ of $C^{\mathbf{s}}_{\OP{nodal}},$ which must be different from the two points $q_i,$ is thus contained inside $\CP^2 \setminus \ell_\infty.$

We are now ready to formulate the existence result for conic foliations.
\begin{thm}
\label{thm:Lefschetz}
The conics in $\mathcal{M}_{J_{\mathbf{s}}}(v_1,v_2)$ form a smooth foliation of $\CP^2 \setminus (\ell_\infty \cup \{x^{\mathbf{s}}_{\OP{nodal}}\})$ with symplectic leaves, whose unique singular fibre is given by the nodal conic $C^{\mathbf{s}}_{\OP{nodal}}$ with node $x^{\mathbf{s}}_{\OP{nodal}}.$ There is an induced family of symplectic fibrations $f_{J_{\mathbf{s}}} \colon \CP^2 \setminus \ell_\infty \to \C,$ which are submersions outside of the singularity of the nodal conic, and which depend smoothly on the parameter ${\mathbf{s}} \in I^k.$

Under the further assumption that $J=i$ holds inside some given subset of the form $f^{-1}(U),$ $ U \subset \C,$ one can ensure that $f_{\widetilde{J}}|_{f^{-1}(U)}=f$ is satisfied there. In particular, this can be assumed to hold above some subset $U \subset \C$ whose complement is compact. 
\end{thm}

\begin{proof}
The proof of the existence of the foliation relies on the well known fact that tame almost complex structures form a contractible space \cite{Gromov:Pseudo}. As a consequence, also the tame almost complex structures being standard at infinity form a contractible space. We may thus extend the family $J_{\mathbf{s}},$ $\mathbf{s} \in I^k,$ to a smooth family $J_{(s,\mathbf{s})}$ parametrised by $(s,{\mathbf{s}}) \in I \times I^k,$ where $J_{(0,\mathbf{s})}\equiv i$ and $J_{(1,\mathbf{s})}= J_{\mathbf{s}}.$

The transversality of the space of conics for all almost complex structures $J_{(s,\mathbf{s})},$ as well as the foliation property, are then both consequences of the below automatic transversality result Lemma \ref{lma:auttrans}, together with a cobordism argument applied to the moduli space $\bigcup_{(s,\mathbf{s})\in I \times I^k} \mathcal{M}_{J_{s,\mathbf{s}}}(v_1,v_2)$ of conics satisfying the given tangencies. Note that Lemma \ref{lma:auttrans} implies that the evaluation map from the moduli space is a \emph{diffeomorphism} defined locally near any given conic (except at the node). The global foliation property is then a consequence of the facts that 
\begin{itemize}
\item for any fixed $\mathbf{s},$ we have $\mathcal{M}_{J_{0,\mathbf{s}}}(v_1,v_2) \setminus \{2\ell_\infty\} \cong \C,$ where we use $2\ell_\infty$ to denote the two-fold cover of the line at infinity branched at $\{q_1,q_2\},$
\item the conics $\mathcal{M}_{J_{0,\mathbf{s}}}(v_1,v_2) \setminus \{2\ell_\infty\}$ foliate $\C^2 \setminus \{0\},$ and
\item the almost complex structures $J_{(s,\mathbf{s})}$ are all standard near $\ell_\infty,$ and hence there exists a neighbourhood $2\ell_\infty \in U \subset \mathcal{M}_{J_{s,\mathbf{s}}}(v_1,v_2)$ of solutions which persist for all $(s,\mathbf{s}) \in I \times I^k.$
\end{itemize}

The third point above combined with positivity of intersection also implies that no line can satisfy both tangency conditions $v_i,$ $i=1,2,$ simultaneously. Hence, the two lines satisfying these tangency conditions join to form a nodal conic $C_{\OP{nodal}}^{(s,\mathbf{s})},$ as sought.

To produce the symplectic fibrations $f_{J_{(s,\mathbf{s})}} \colon \CP^2 \setminus \ell_\infty \to \C$ whose fibres are the leaves in our conic foliation, we use the fact that the evaluation map from the moduli space is a diffeomorphism away from the node. 

First, by positivity of intersection together with $[2\ell_\infty]\bullet[2\ell_\infty]=4,$ two conics in this family must intersect precisely in the two points $\{q_1,q_2\} \subset \ell_\infty.$

Then, we fix standard affine holomorphic coordinates $[1:z_1:z_2] \in \CP^2$ around $q_1=[1:0:0]$ in which $\ell_\infty$ is given by $\{z_2=0\}.$ Since the almost complex structures considered are standard near $q_1,$ each conic $u$ has a uniquely determined power series expansion of the form
$$ z \mapsto (z_1,z_2)=(0,z)+\sum_{k\ge 2} (a_k^{(s,\mathbf{s})}(u)z^k,0), \:\: a_k \in \C,$$
after a suitable reparametrisation of the domain (depending smoothly on the conic $C$). The map $f_{J_{(s,\mathbf{s})}} \colon \C^2 \to \C$ which along each $J_{(s,\mathbf{s})}$-holomorphic conic $u \subset \C^2,$ $u \in \mathcal{M}_{J_{(s,\mathbf{s})}}(v_1,v_2),$ takes as value the corresponding coefficient $a_2^{(s,\mathbf{s})}(u)$ is a smooth function by the foliation property. We end by arguing that this is a fibration of the sought form.

First we show that
$$a_2^{(s,\mathbf{s})} \colon \mathcal{M}_{J_{(s,\mathbf{s})}}(v_1,v_2) \to \C$$
is injective. To that end, note that two different conics would intersect with local intersection index at least $\ge 3$ at $q_1$ if they would have the same coefficient $a_2^{(s,\mathbf{s})}$ in the above expansion. Together with positivity of intersection, this is then in contradiction with $[2\ell_\infty]\bullet [2\ell_\infty]=4,$ taking into account that the local intersection index at the other point $q_2 \in \ell_\infty$ is at least $\ge 2.$

Then we claim that, when $s=0,$ this construction gives back the standard fibration $f=f_i.$ A topological argument now shows that $f_{J_{(s,\mathbf{s})}}$ is surjective for all $(s,\mathbf{s}).$

It remains to show that $f_{J_{(s,\mathbf{s})}}$ is submersive away from the node of the singular conic. This is a consequence of the foliation property, together with the last statement of Lemma \ref{lma:auttrans} by which $a_2^{(s,\mathbf{s})}$ is submersive.
\end{proof}

The following automatic transversality result was crucial in the above proof.

\begin{lma}
\label{lma:auttrans}
Any smooth conic (i.e.~a conic which is neither nodal nor a branched cover) inside $\mathcal{M}_J(v_1,v_2)$ is a regular solution to this moduli problem for an arbitrary tame $J.$ Consequently, $\mathcal{M}_J(v_1,v_2)$ is a smooth two-dimensional manifold. Furthermore, the section normal to some conic $u \in \mathcal{M}(v_1,v_2)$ corresponding to a nonzero vector in $T_u\mathcal{M}_J(v_1,v_2)$ vanishes precisely at the two points $q_i,$ $i=1,2,$ where it moreover has zeros of precisely order two.
\end{lma}
\begin{proof}
The statement is a fairly straight forward consequence of the automatic transversality result in \cite{Hofer:OnGenericity}; we proceed to give the argument.

Consider the space $\mathcal{M}_J$ of embedded $J$-holomorphic conics $u \colon \CP^1 \to \CP^2,$ together with a fixed solution $u_0 \in \mathcal{M}_J(v_1,v_2)$ satisfying the tangency conditions $Du_0(T_0\CP^1)=v_1$ and $Du_0(T_\infty \CP^1)=v_2.$ Recall that these conics are embedded by Lemma \ref{lma:conics}. The kernel of the linearised $\overline{\partial}$-problem \emph{disregarding reparametrisations} is thus a complex five-dimensional space by the aforementioned automatic transversality; indeed, the cokernel vanishes and the Fredholm index is equal to $(n-3)\chi(\CP^1)+2c_1(\CP^2)[u_0]=-1(2)+12=10.$

We need to show that the infinitesimal evaluation map
$$(u,\pt_1,\pt_2) \mapsto ((u(\pt_1),Du(T_{\pt_1}\CP^1)),(u(\pt_2),Du(T_{\pt_2}\CP^1)))$$
is transverse to the pair $((q_1,v_1),(q_2,v_2)).$ Since we consider an embedded conic, we can identify the solutions near $u_0 \in \mathcal{M}_J$ with certain sections $\sigma_u \in \Gamma(\nu_{u_0})$ in the normal bundle of $u_0.$ We make the choice of appropriate holomorphic coordinates near the two points $q_i$ (recall that $u_0$ is holomorphic near these two points), and can then assume that the normal bundle is holomorphic there, and that the equation $\overline{\partial}_J$ for the sections $\sigma_u$ actually is the standard Cauchy--Riemann operator near the two points $\{0,\infty\}=u_0^{-1}\{q_1,q_2\} \subset \CP^1.$

From this point of view, the problem boils down to showing that the map
$$ \mathcal{M}_J \ni u \mapsto ((\sigma_u(0),\sigma_u'(0),(\sigma_u(\infty),\sigma_u'(\infty))) \in ((\nu_{u_0})_{0})^2 \times ((\nu_{u_0})_{\infty})^2 \cong (\C^2)^2$$
is submersive to the origin for $u$ close to $u_0.$ The differential of this map is a linear map
$$ \Phi \colon \C^5 \cong T_{u_0}\mathcal{M}_J \ni \xi \mapsto ((\xi(0),\xi'(0)),(\xi(\infty),\xi'(\infty))) \in \C^4,$$
where $\xi \in T_{u_0}\mathcal{M}_J$ again can be seen as a section in the normal bundle of $u_0.$ Moreover, it satisfies the properties that
\begin{itemize}
\item $\xi$ can be identified with a holomorphic map to $\C$ near the two points $\{0,\infty\}=u_0^{-1}\{q_1,q_2\} \subset \CP^1$ in the above coordinates;
\item every geometric intersection of $\xi$ with the zero section contributes positively to the algebraic intersection number.
\end{itemize}
The second statement is the main technical result of \cite{Hofer:OnGenericity}.

In conclusion, when $\Phi$ is not surjective, one can readily find a section $\xi \in \Gamma(\nu_{u_0})$ with a sufficiently high vanishing at the points $\{0,1\}=u_0^{-1}\{q_1,q_2\},$ so that the algebraic intersection index there is equal to at least $+5$ there. Using the aforementioned result concerning positivity of intersection, this is in contradiction with the fact that the Euler number of $\nu_{u_0}$ is equal to $4.$ In order to see the claimed vanishing, we argue as follows. When $\Phi$ is not surjective, then the linear subspace $\ker \Phi \subset \C^2$ is of real dimension $\ge 3.$ In this situation one thus finds a one-dimensional subspace which satisfies the vanishing $\xi''(0)=0$ as well.

The claim concerning the vanishing of the section corresponding to the infinitesimal variation in $T_u\mathcal{M}_J(v_1,v_2)$ is shown similarly, using the positivity of intersection from \cite{Hofer:OnGenericity} together with the fact $[u] \bullet [u] =4.$ To that end, note that any section in the normal bundle coming from a nonzero variation in $T_u\mathcal{M}_J(v_1,v_2)$ automatically vanishes to order \emph{at least} two at both points $q_i,$ $i=1,2,$ due to the tangency condition.
\end{proof}

\subsection{Normalising the fibration}
For us it is necessary to perform a further normalisation of the conic fibration supplied by Theorem \ref{thm:Lefschetz} above. In particular, we want to make the fibration standard outside of a compact subset, and to make the nodal conic standard near its node.
\begin{rmk}
The normalised fibration will still not define a symplectic Lefschetz fibration in the complement of the line $\ell_\infty$ at infinity in the normal sense; see e.g.~\cite{McLean:LefschetzFibrations} for the definition. The reason is that the requirement for the fibration to be symplectically trivial outside of a compact subset is not satisfied even for the standard fibration $f.$
\end{rmk}
\begin{thm}
 
 \label{thm:normalise}
Assume that we are given a conic fibration $f_J$ as produced by Theorem \ref{thm:Lefschetz} above, where $f_J$ is standard inside a subset of the form $f^{-1}U,$ $U \subset \C$ where $\C \setminus U$ is compact. It is possible to find a one-parameter family $f_{J_t},$ $t \in [0,1],$ of such conic fibrations, where $J_0=i$ and $f_{J_0}=f$ both are standard, for which:
\item \begin{itemize}
\item the fibres of $f_{J_1}$ coincide with the fibres of $f_J$ outside some small neighbourhood of the form
$$O_\epsilon:=(B^4_\epsilon(q_1) \cup B^4_\epsilon(q_2)) \setminus (\ell_\infty \cup f^{-1}F), \:\: 1 \gg \epsilon>0,$$
where $F \subset U$ is an arbitrary closed subset. Moreover, the almost complex structure $J_1$ can be taken to coincide with $J$ outside of the neighbourhood $O_{\epsilon'}$ for some $0<\epsilon' \ll \epsilon$;
\item the nodal conics $C^{J_t}_{\OP{nodal}}$ coincide with the standard nodal conic $C_{\OP{nodal}}$ near $\ell_\infty$ and is moreover given as the preimage $C^{J_t}_{\OP{nodal}}= f^{-1}_{J_t}(0)$; and
\item $f_{J_t}=f$ holds outside of some compact subset of $\C^2.$
\end{itemize}
\end{thm}
\begin{proof}
There exists a path $J_t$ from $J_0=i$ to $J_1=J$ with corresponding fibrations $f_{J_t}$ by $J_t$-holomorphic conics. See e.g.~the proof of Theorem \ref{thm:Lefschetz}.

We start to normalise the foliations near the two points $\{q_1,q_2\},$ making them coincide with the standard foliation there. Note that the smooth foliation property outside of these two points then allows us to find a deformations $J'_t$ of the path of almost complex structures, where still $J'_0=i,$ for which the deformed foliations are $J'_t$-holomorphic. (Here may assume that $J'=i$ still holds in a possibly smaller neighbourhood of $\ell_\infty.$)

The symplectic foliation is deformed by carefully replacing the coefficients in the power series expansions near the points $\{q_1,q_2\}$ that was described in the proof of Theorem \ref{thm:Lefschetz}. For simplicity we will here consider the case of the fixed fibration $f_J$; the general one-parameter case is proven without any additional difficulty.

First we recall our choice of power series expansions near $q_i$ for the leaves. Take the standard affine holomorphic coordinates $[1:z_1:z_2] \in \CP^2$ around $q_1=[1:0:0]$ in which $\ell_\infty$ is given by $\{z_2=0\}.$ Since the almost complex structures considered are standard near $q_1,$ each conic $u \in \mathcal{M}_J(v_1,v_2)$ has a uniquely determined power series expansion of the form
$$ z \mapsto (z_1,z_2)=(0,z)+\sum_{k\ge 2} (a_k(u)z^k,0), \:\: a_k \in \C,$$
after a suitable reparametrisation of the domain. In analogous affine coordinates near $q_2,$ the leaves can be written as
$$ z \mapsto (z_1,z_2)=(z,0) + \sum_{k\ge 2} (0,b_k(u)z^k,0), \:\: b_k \in \C.$$
The coefficients $b_k,a_k$ of non-minimal degree in the above power series can be replaced by functions of the form
$$ \beta_r(\|z\|) \cdot a_k(u), \:\: \beta_r(\|z\|) \cdot b_k(u), \:\: k \ge 3,$$
where $\beta_r$ is a bump function satisfying $\beta_r'(t) \ge 0,$ $\beta_r(t)=1$ for $t \ge r,$ $\beta_r(t)=0$ near $t=0,$ while in addition $|\beta_r'| \le 2/r$ is satisfied. Note that such a deformation does not deform those leaves which already are standard near the points $q_i.$ We claim that, when $t_0=r>0$ is taken to be sufficiently small, then this deformation is still a symplectic foliation, since only the higher order terms are deformed. Here we use the facts that the inequality
$$\|d(\beta_r(\|z\|)\cdot z^k)\|\le 2r^{k-1}+kr^{k-1}=(k+2)r^{k-1},\:\: k \ge 3,$$
holds in the region $D^2_r$ containing the support of $\beta_r'.$ (Here $r>0$ is sufficiently small.) In this manner, we can make the leaves of the foliation coincide with leaves of the standard foliation near the two points $q_i.$ This finishes the claim in the first bullet point.

For the second bullet point, we need to make
$$a_2(C^J_{\OP{nodal}})=b_2(C^J_{\OP{nodal}})=0$$
satisfied. We proceed as follows. Recall that all conic fibres except $2\ell_\infty$ are graphical over the second and first affine coordinate around $q_1$ and $q_2,$ respectively, as described above. The foliation can then readily be deformed by replacing the coefficients $a_2(u)$ and $b_2(u)$ by coefficients of the form
$$\phi_{\beta_r(2\|z\|)}(a_2(u)) \:\:\:\: \text{and} \:\:\:\: \psi_{\beta_r(2\|z\|)}(b_2(u))$$
for suitable smooth and compactly supported isotopies $\phi_t,\psi_t \colon \C \to \C$ which both satisfy $\phi_t = \id_\C$ for all $t \ge 1/2.$ For the symplectic condition of the deformed foliation, note that
\begin{align*}
& \|d(\phi_{\beta_r(2\|z\|)}(a_1(u))\cdot z^2)\| \le c_1(u)\cdot r,\\
& \|d(\psi_{\beta_r(2\|z\|)}(b_2(u))\cdot z^2)\| \le c_1(u)\cdot r,
\end{align*}
is satisfied inside $\{\|z\| \le r \},$ for $r>0$ sufficiently small, and where
$$c_1 \colon \mathcal{M}(v_1,v_2) \setminus \{2\ell_\infty\} \to \R_{\ge 0}$$
is a continuous function depending on the choice of isotopy which is constant outside of some compact subset. Since the isotopies are compactly supported, it thus suffices to take $r>0$ sufficiently small in order to guarantee the symplectic condition.

 What remains is the last bullet point. We need to make each fibre of $f_J$ coincide with \emph{one and the same} fibre of $f$ near both points $q_i,$ $i=1,2.$ This can be done by the same method as in the proof of the second bullet point, using suitable isotopies $\phi_t,\psi_t$ of $\C.$ Once this has been done, it is a simple matter of reparametrising the target $\C$ of the fibration in order to achieve the sought property.
\end{proof}
In addition it will be useful to consider the following normalisation of the fibration over a path in the base starting at the image of the nodal conic. Let
$$\gamma \colon [0,1] \hookrightarrow \C$$
be an embedded path which
\begin{itemize}
\item coincides with the canonical inclusion $\gamma_0 \colon [0,1] \hookrightarrow \C$ near its boundary, and
\item is homotopic to the canonical inclusion $\gamma_0$ through embeddings $\gamma_t,$ $\gamma_1=\gamma,$ all coinciding in some neighbourhood of the boundary.
\end{itemize}

\begin{thm}
\label{thm:normalisepath}
 
Assume that $f_{J_t} \colon (\CP^2 \setminus \ell_\infty,\omega_{\OP{FS}}) \to \C$ is a smooth path of symplectic conic fibrations with $J_0=i$ and $J_1=J,$ where $f_{J_t}$ all are fixed outside of a compact subset of $\CP^2 \setminus (\ell_\infty \cup C \cup C_{\OP{nodal}}).$ Under the above assumptions on $\gamma,$ there then exists a compactly supported Hamiltonian isotopy
$$\phi^t_{H_t} \colon (\CP^2 \setminus (\ell_\infty \cup C_{\OP{nodal}}),\omega_{\OP{FS}}) \to (\CP^2 \setminus (\ell_\infty \cup C_{\OP{nodal}}),\omega_{\OP{FS}})$$
which maps $f_J^{-1}(\gamma(x))$ to $f^{-1}(x)$ for each $x \in [0,1].$
\end{thm}

\begin{proof}
Recall that the symplectic fibrations $f_{J_t}$ give rise to a parallel transport along the extended curves $\gamma_t$ by integrating a suitably normalised characteristic vector field inside $f_{J_t}^{-1}\gamma_t.$ Using this parallel transport, starting from the conic fibre $f_{J_t}^{-1}(\epsilon)=f^{-1}(\epsilon)$ for some small $\epsilon>0,$ we obtain a compactly supported isotopy
$$ \varphi_t \colon f_{J_0}^{-1}(\epsilon) \times [\epsilon,1+\epsilon] \hookrightarrow \CP^2 \setminus (\ell_\infty \cup C_{\OP{nodal}})$$
of hypersurfaces where
\begin{itemize}
 
\item $\varphi_t^*\,\omega_{\OP{FS}}$ is constant,
\item $\varphi_t(f_{J_0}^{-1}(\epsilon) \times \{x\})=f_{J_t}^{-1}(\gamma_t(x))$; and
\item $\varphi_t=\varphi_0$ holds outside of a compact subset, as well as near $C_{\OP{nodal}}.$
\end{itemize}
The independence of $t$ asserted in the last two points is a consequence of the assumption that the fibrations all are standard outside of a compact subset of $\CP^2 \setminus (\ell_\infty \cup C \cup C_{\OP{nodal}}).$

The standard symplectic neighbourhood theorem \cite{McDuff:Introduction} then provides us with an extension of the family $\varphi_t$ to a family
\begin{gather*}
\Phi_t \colon (f_{J_0}^{-1}(\epsilon) \times [\epsilon,1+\epsilon] \times [-\epsilon',\epsilon'],\omega) \hookrightarrow (\CP^2 \setminus \ell_\infty,\omega_{\OP{FS}}),\\
\Phi_t|_{f_{J_0}^{-1}(\epsilon) \times [\epsilon,1+\epsilon] \times \{0\}}=\varphi_t,
\end{gather*}
of open symplectic embeddings fixed outside of a compact subset of each $f_{J_0}^{-1}(\epsilon) \times \{(x,y)\}$. For the last property, we again rely on the assumption that the fibrations are standard outside of a compact subset.

Finally, since the support of the above family of symplectomorphisms is of the aforementioned form, a standard fact now shows that it can be generated by a compactly supported Hamiltonian, i.e. $\Phi_t=\phi^t_{G_t} \circ \Phi_0,$ where $G_t$ moreover can be taken to vanish in some neighbourhood of $f^{-1}[0,\epsilon].$ A suitable cut-off of this Hamiltonian then generates the sought global Hamiltonian isotopy of $(\CP^2 \setminus \ell_\infty,\omega_{\OP{FS}}).$
\end{proof}

\subsection{Hamiltonian isotopies of symplectic surfaces with smooth self-intersection}
\label{sec:nodal}
Here we recall and establish facts concerning Hamiltonian isotopy of nodal symplectic surfaces, as well as symplectic surfaces with more complicated discrete self-intersections. First recall that a smooth family of embedded symplectic surfaces can be generated by a Hamiltonian isotopy by the following basic result.
\begin{prp}[Proposition 0.3 in \cite{Siebert:OnTheHolomorphicity}]
\label{prp:SiebertTian}
Let $\Sigma_t \subset (X^4,\omega)$ be a smooth isotopy of symplectic surfaces, where the isotopy moreover is fixed inside some (possibly empty) subset $F \subset X.$ Then there exists a Hamiltonian $H_t \colon X \to \R$ for which $\Sigma_t=\phi^t_{H_t}(\Sigma_0).$ The Hamiltonian $H_t$ can moreover be taken to vanish on any given open subset $U \subset F.$ 
\end{prp}

There is no analogous results for \emph{nodal} symplectic surfaces; for example, the tangent planes at the node are not generically symplectically orthogonal. We now argue that we still can find a Hamiltonian isotopy that generates a path of e.g.~nodal symplectic surfaces, albeit after an initial deformation near its nodes. First, a discrete self-intersection locus may be assumed to remain fixed in the family after a Hamiltonian isotopy. Proposition \ref{prp:NormaliseNode} below then makes it possible to deform the surfaces near the self-intersection locus in order to make the symplectomorphism class of the singularities constant. Finally, we can apply the above Proposition \ref{prp:SiebertTian} to this deformed family.

Let $\Sigma_t \subset (X^4,\omega)$ be a smooth family of symplectic immersions of a finite union of closed surfaces whose self-intersection loci all consists of precisely the $k$ number of points $p^1,\ldots,p^k \in \Sigma_t$ fixed in the family. (These points are all isolated singularities by assumption, but they are not necessarily transverse double points.)
\begin{prp}
\label{prp:NormaliseNode}
After a deformation of $\Sigma_t$ through smooth paths of symplectic immersions of the same kind, where the deformation can be assumed to be supported in an arbitrarily small neighbourhood of $p^1,\ldots,p^k,$ we obtain a path $\widetilde{\Sigma}_t$ of symplectic immersions that satisfies $\widetilde{\Sigma}_0=\Sigma_0$ and which is fixed in a \emph{neighbourhood} of its self-intersections.
\begin{enumerate}
\item If, in addition, the path of immersion $\Sigma_t$ is fixed when restricted to a number of components $A \subset \Sigma_t,$ then we may assume $A \subset \widetilde{\Sigma}_t$ to be fixed as well.
\item If, in addition, $p^i \in \Sigma_t$ is a transverse double-point for all $t,$ and if the corresponding sheets of $\Sigma_0$ and $\Sigma_1$ coincide near $p^i,$ then $\widetilde{\Sigma}_1=\Sigma_1$ can be assumed to hold near $p^i.$
\end{enumerate}
\end{prp}
\begin{proof}
 
The deformation follows as in the proof of Corollary 3.7 in \cite{Dimitroglou:Isotopy}. Roughly speaking, we consider the links of the singularities of the immersions (i.e.~the points $p^1,\ldots,p^k$). The singularities consist of intersections of a number of smooth sheets. The links are thus transverse links inside the standard contact three-sphere $(S^3,\xi_{\OP{std}}),$ where each component of the link is a standard transverse unknot. The deformation is obtained by a suitable smooth interpolation, by attaching symplectic cylinders given as the trace of an isotopy through transverse links.

(2): Here we need to use the fact that the space consisting of linear symplectic two-planes that are transverse to another given linear symplectic two-plane is contractible.
\end{proof}

\section{Properties derived from broken conic fibrations}

In this section $L \subset V \subset (\CP^2,\omega_{\OP{FS}})$ will always be used to denote an embedded Lagrangian \emph{torus} satisfying at least one of the conditions of Theorem \ref{thm:main}. Our goal here is to establish Theorem \ref{thm:CompatibleFibration}, i.e.~to construct a pseudoholomorphic conic fibration in the sense of Section \ref{sec:pencils} that is compatible with $L.$ Recall that compatibility of the fibration with the torus implies that the latter is fibred over an embedded closed curve in the base. In the entire section we rely heavily on the technique of stretching the neck, which is performed to the unit normal bundle of the Lagrangian torus.

\subsection{A neck-stretching sequence}
\label{sec:neckstretch}

We follow \cite[Section 2]{Dimitroglou:Isotopy} in the construction of a sequence of almost complex structures which {\bf stretch the neck} around an embedding of the unit cotangent bundle of the Lagrangian torus $L \subset (\CP^2 \setminus \ell_\infty,\omega_{\OP{FS}}).$ This is a sequence $J_\tau,$ $\tau \ge 0,$ of compatible almost complex structures on $(\CP^2,\omega_{\OP{FS}})$ satisfying the following properties:
\begin{itemize}
\item $J_\tau=i$ in a neighbourhood of $\ell_\infty$ as well as in a neighbourhood of the smooth conic $C \subset \CP^2;$
\item in a fixed Weinstein neighbourhood
\begin{gather*}
\phi \colon (D_{3\epsilon}T^*\T^2,d\lambda_{\T^2}) \hookrightarrow (V,\omega_{\OP{FS}}),\\
\phi(0_{\T^2})=L,
\end{gather*}
the almost complex structure takes the form 
$$J_\tau \partial_{\theta_i}=-\rho_\tau(\|\mathbf{p}\|)\partial_{p_i}$$
for a function $\rho_\tau \colon \R_{\ge 0} \to \R_{> 0}$ satisfying $\rho_\tau(t)=\epsilon$ for $t \le \epsilon,$ $\rho_\tau(t)=t$ for $t \ge 2\epsilon,$ while $\int_\epsilon^{2\epsilon} \rho(t)dt \ge \tau;$ and
\item $J_\tau$ is fixed outside of the above Weinstein neighbourhood $\phi(D_{3\epsilon}T^*\T^2) \subset \CP^2.$
\end{itemize}
In Section \ref{sec:proof-whitney} a variation of the above construction will be used, where we stretch the neck around two disjoint Lagrangian tori simultaneously. In that case, the sequence is constructed in the analogous manner utilising disjoint Weinstein neighbourhoods of the two tori.

The above choices also specifies the following important compatible almost complex structures:
\begin{itemize}
\item The compatible almost complex structure $J_{\OP{std}}$ on $T^*\T^2$ which is given by $J_{\OP{std}}\partial_{\theta_i}=-\rho_1(\|\mathbf{p}\|)\partial_{p_i}$;
\item The compatible almost complex structures $J_{\OP{cyl}}$ defined on
$$(T^*\T^2 \setminus 0_{\T^2},d\lambda_{\T^2})=(\R \times ST^*\T^2,d(e^t\lambda_{\T^2}|_{ST^*\T^2}))$$
by $J_{\OP{cyl}}\partial_{\theta_i}=-\|\mathbf{p}\|\partial_{p_i}.$ It is important to note that this almost complex structure is cylindrical with respect to the contact form $\alpha\coloneqq\lambda_{\T^2}|_{ST^*\T^2}$ induced by the flat Riemannian metric on $\T^2$; and
\item The compatible almost complex structure $J_\infty$ defined on $\CP^2 \setminus L$ which coincides with $J_{\OP{cyl}}$ inside the above Weinstein neighbourhood $\phi(D_{3\epsilon}T^*\T^2)$ and with $J_\tau,$ where $\tau \ge 0$ is arbitrary, in the complement $\CP^2 \setminus \phi(D_{3\epsilon}T^*\T^2).$
\end{itemize}

Recall that the periodic Reeb orbits of $\alpha$ correspond to lifts of closed geodesics on $\T^2$ induced by the flat metric. A basic but very crucial fact is that these geodesics all live in Bott manifolds $\Gamma_\alpha \cong S^1,$ which are in bijection with the nonzero homology classes $\alpha \in H_1(L) \setminus \{0\}$ of the corresponding geodesics. In particular, there are no closed contractible geodesics for the flat metric.

When speaking about finite energy pseudoholomorphic spheres we mean pseudoholomorphic spheres inside either $(T^*\T^2,J_{\OP{std}}),$ $(\R \times ST^*\T^2,J_{\OP{cyl}}),$ or $(\CP^2 \setminus L,J_\infty)$ with a finite number of punctures asymptotic to periodic Reeb orbits on $(S^*\T^2,\alpha).$ By a classical result this condition is equivalent to that of having finite so-called Hofer energy; see \cite{Hofer:I} and \cite{Hofer:IV}. In the following all punctured pseudoholomorphic spheres will tacitly be assumed to be holomorphic for one of the almost complex structures as described above, and to be of finite Hofer energy.

A one-punctured pseudoholomorphic sphere is called a {\bf plane} while a two-punctured pseudoholomorphic sphere is called a {\bf cylinder}.

By a {\bf broken pseudoholomorphic conic} we mean a pseudoholomorphic building of at least two levels whose components satisfy the following topological conditions: gluing all the domains of the components at their nodes we obtain a sphere without punctures, on which the maps compactify to give a continuous cycle $S^2 \to \CP^2$ of degree \emph{two}. Of course there also exist closed $J_\infty$-holomorphic curves inside $\CP^2 \setminus L$ without punctures; e.g.~the conic $C \subset \CP^2$ is $J_\infty$-holomorphic by the assumptions made on $J_\infty.$ Such curves will be called {\bf unbroken}. Similarly we also consider {\bf (un)broken pseudoholomorphic lines} inside $\CP^2 \setminus L.$

It is immediate from the SFT compactness theorem that the limit of a sequence of $J_\tau$-holomorphic conics (resp. lines) as $\tau \to +\infty$ is a broken or unbroken pseudoholomorphic conic (resp. line); see \cite{Bourgeois:Compactness} or \cite{Cieliebak:Compactness}.
\begin{rmk}
It is not clear a priori that the converse holds, i.e.~that all broken pseudoholomorphic conics or lines arise as such limits; this would require a rather strong form of pseudoholomorphic gluing in the Bott setting under consideration. Since we do not rely on such a result, we must instead use somewhat roundabout arguments based upon (asymptotic) positivity of intersection of pseudoholomorphic buildings, combined with automatic transversality for the components in the building, in order to rule out certain unwanted broken configurations.
\end{rmk}
When considering limits of pseudoholomorphic conics or lines in the four-dimensional setting, the following crucial property is a consequence of positivity of intersection \cite{McDuff:LocalBehaviour}.
\begin{lma}
\label{lma:embedded-components}
The limit of pseudoholomorphic lines or conics under a neck-stretching sequence consists of components all being (possibly trivial) branched covers of embedded punctured pseudoholomorphic spheres. In the case of a limit of lines, it is moreover the case that two different components have disjoint interiors.
\end{lma}

Another important technical feature of the dimension where we are working is that the Fredholm index of a punctured sphere can be forced to be nonnegative under certain mild assumptions; this facilitates transversality arguments and analysis significantly. More precisely, we have
\begin{lma}[Lemma 3.3 in \cite{Dimitroglou:Isotopy}]
\label{lma:posindex}
If all simple $J_\infty$-holomorphic punctured spheres inside $\CP^2 \setminus L$ are of non-negative Fredholm index, then the same is true for all $J_\infty$-holomorphic punctured spheres. Since a plane has odd Fredholm index, its index is thus at least one in this case. Moreover, if the index of a such a punctured plane inside $\CP^2 \setminus (\ell_\infty \cup L)$ is equal to one, then it is simply covered with a simply covered asymptotic. Finally, for a generic almost complex structures $J_\infty$ as in Section \ref{sec:neckstretch}, all simply covered curves can indeed be assumed to be transversely cut out and hence of nonnegative index.
\end{lma}

\begin{lma}[Proposition 3.5 in \cite{Dimitroglou:Isotopy}]
\label{lma:planecyl}
Assume that there exists no punctured pseudoholomorphic spheres inside $\CP^2 \setminus L$ of negative index. After perturbing $J_\infty$ inside some compact subset of $U \setminus \ell_\infty,$ where $U \subset \CP^2$ is an arbitrarily small neighbourhood of $\ell_\infty,$ we may assume that any curve being either
\begin{itemize}
\item a broken line satisfying a fixed point constraint at $q \in \ell_\infty,$ or
\item a broken conic satisfying two fixed tangency condition at two points $q_1,q_2 \in \ell_\infty,$
\end{itemize}
has a top level consisting of precisely two planes of index \emph{one}, and possibly several cylinders of index \emph{zero}. (Here we take appropriate constraint(s) into account when considering the index of a component.) Moreover, the components passing through $\ell_\infty$ are transversely cut out when considered with the appropriate point and tangency condition, respectively.
\end{lma}
\begin{rmk}
In the above lemma we do not assume that the broken curve is an actual limit under a neck stretching sequence.
\end{rmk}
\begin{proof}
First we note that the Fredholm index for the corresponding non-broken solutions is equal to two in both of the cases, where the moduli spaces are considered with the corresponding point or tangency constraint. For that reason, the proof of \cite[Proposition 3.5]{Dimitroglou:Isotopy} carries over immediately to the current situation, once the following transversality result has been established: for a generic almost complex structures of the type considered, the constrained moduli spaces are transversely cut out and assume their expected dimensions.

Since we are in a very particular situation the sought transversality is not difficult to establish. In the unbroken case, this is a consequence of automatic transversality; see Lemma \ref{lma:auttrans}. In the case of a broken curve, positivity of intersection implies that the component satisfying the additional constraint(s) necessarily is simply covered and traverse to $\ell_\infty.$ The transversality properties can now be achieved by finding explicit deformations by hand, while using standard transversality techniques \cite[Section 3]{McDuff:J}.
\end{proof}

 Most work in the remaining part of this section is to sharpen the above result, by showing that a broken conic consists of precisely two components in its top level, where each component moreover intersects $\ell_\infty,$ as is depicted in Figure \ref{fig:breaking}. In order for this to hold, it is crucial to use one of the two conditions in the assumptions of Theorem \ref{thm:main}. 

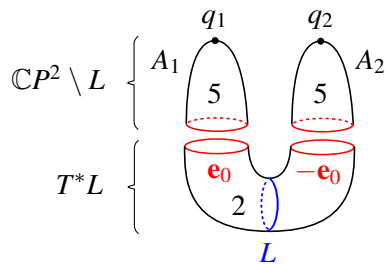
\begin{figure}[htp]
\centering
\vspace{3mm}
\hspace{20mm}
\labellist
\pinlabel $\CP^2\setminus L$ at -26 58
\pinlabel $T^*L$ at -18 19
\pinlabel $q_2$ at 73 81
\pinlabel $q_1$ at 33 81
\pinlabel $5$ at 73 53
\pinlabel $A_2$ at 92 65
\pinlabel $A_1$ at 15 65
\pinlabel $5$ at 33 53
\pinlabel $2$ at 42 11
\pinlabel $\color{red}\mathbf{e}_0$ at 35 23
\pinlabel $\color{red}-\mathbf{e}_0$ at 72 23
\pinlabel $\color{blue}L$ at 53 -7
\endlabellist
\includegraphics{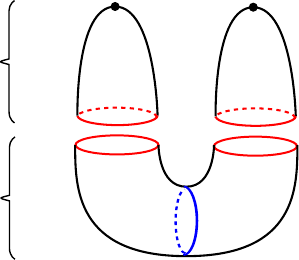}
\vspace{5mm}
\caption{The broken conic in the generic case. The two planes $A_i$ each intersect $\ell_\infty$ transversely in the point $q_i,$ where they moreover are tangent to the smooth conic $C.$ The two planes join to form a continuous embedding of a sphere.}
\label{fig:breaking}
\end{figure}

\subsection{Solid tori foliated by pseudoholomorphic planes and consequences}

When a broken line or conic consists of a plane that is disjoint from $\ell_\infty,$ we can under certain assumptions use the techniques from \cite{Dimitroglou:Isotopy} in order to construct an embedded solid torus with boundary equal to $L,$ which is foliated by discs of which one of is close to the aforementioned plane.

\begin{prp}{\cite[Section 5]{Dimitroglou:Isotopy}}
\label{prp:solidtorus}
 
Let $L \subset (\CP^2 \setminus \ell_\infty,\omega_{\OP{FS}})$ be an embedded Lagrangian torus and consider any compatible almost complex structure $J_\infty \subset \CP^2 \setminus L$ that satisfies the properties of Lemma \ref{lma:planecyl}. Assume that there exists a pseudoholomorphic plane $A \subset \CP^2 \setminus (L \cup \ell_\infty)$ asymptotic to $L$ which arises a component of a generic broken line satisfying a fixed point constraint at one of $q_i \in \ell_\infty.$ Then there exists a smooth embedding of a solid torus
$$ (S^1 \times D^2,S^1 \times S^1) \hookrightarrow (\CP^2 \setminus \ell_\infty,L)$$
foliated by embedded $J$-holomorphic discs in the same homotopy class as $A,$ and where $J=J_\infty$ can be taken to be satisfied outside of an arbitrarily small neighbourhood of $L.$
\end{prp}

Using the above we show that a Lagrangian torus that satisfies either of the conditions in the assumptions of Theorem \ref{thm:main} can be disjoined from the standard nodal conic by a Hamiltonian isotopy.

\begin{lma}
\label{lma:DisjoinNodal}
Assume that either Condition (1) or (2) of Theorem \ref{thm:main} is satisfied. For an almost complex structure $J_\infty$ that satisfies the properties of Lemma \ref{lma:planecyl}, the image of the nodal conic is not broken. It follows that the embedded Lagrangian torus $L \subset (V,\omega_{\OP{FS}})$ can be disjoined from the standard nodal conic $C_{\OP{nodal}}=\ell_1 \cup \ell_2$ by a Hamiltonian isotopy supported inside $V.$
\end{lma}
\begin{proof}
 Assume that the nodal conic is broken. Lemma \ref{lma:planecyl} implies that the broken line in the nodal conic has a component that is a plane of index one which is disjoint from $\ell_\infty.$ Using either Condition (2), or Condition (1) in conjunction with the existence of the solid torus produced by Proposition \ref{prp:solidtorus} that bounds $L$ (and contains a perturbation of the pseudoholomorphic plane as a compressing disc), we conclude that the plane necessarily intersects $C.$ Since such a line already is tangent to $C$ at one of the points $q_i \in \ell_\infty$ this finally leads to a contradiction with the total intersection number $[\ell_\infty] \bullet C = 2$ of a line and a conic. For the last part we need positivity of intersection \cite{McDuff:LocalBehaviour}.

What remains is to construct the Hamiltonian isotopy. Start by considering a compatible almost complex structure $J$ for which the aforementioned broken conic $\ell_1^{J} \cup \ell_2^{J}$ as well as $C$ and $\ell_\infty$ all are $J$-holomorphic simultaneously. Then interpolate between $J$ and $J_0$ through compatible almost complex structures, all for which $C \cup \ell_\infty$ remain pseudoholomorphic. Gromov's result Theorem \ref{thm:gromov} produces a smoothly varying family $\Sigma_t$ of symplectic immersions where $\Sigma_0=\ell_1 \cup \ell_2 \cup C \cup \ell_\infty,$ $\Sigma_1=\ell_1^J \cup \ell_2^J \cup C \cup \ell_\infty.$ Deforming this family at times $t>0$ using a Hamiltonian isotopy constructed by hand, we may assume that the position of the node remains fixed in the family. In particular, the node now coincides with $\ell_1 \cap \ell_2$ for all $t \ge 0.$

Part (1) of Proposition \ref{prp:NormaliseNode} applies to this family, yielding a family $\widetilde{\Sigma}_t$ to which Proposition \ref{prp:SiebertTian} can be applied. Here $\widetilde{\Sigma}_0=\Sigma_0,$ while we also have $C,\ell_\infty \subset \widetilde{\Sigma}_t,$ and where $\widetilde{\Sigma}_1 \subset V \setminus L$ is a small perturbation of $\Sigma_1.$ In other words, the global Hamiltonian isotopy produced by Proposition \ref{prp:SiebertTian} that generates the isotopy $\widetilde{\Sigma}_{1-t}$ of symplectic surfaces is the one we seek.
\end{proof}

\begin{lma}
\label{lma:goodbasis}
 Assume that either Condition (1) or (2) of Theorem \ref{thm:main} is satisfied for the Lagrangian torus $L \subset (V,\omega_{\OP{FS}}).$ There exists a basis $\langle \mathbf{e}_0, \mathbf{e}_1 \rangle = H_1(L)$ with the property that $\mathbf{e}_0$ generates the kernel of the canonical map $H_1(L) \to H_1(V)$ induced by the inclusion $L \subset V,$ while any disc $(D,\partial D) \to (\CP^2\setminus \ell_\infty,L)$ with boundary in the homology class $\mathbf{e}_1$ must satisfy $D \bullet C = a_1 >0.$
\end{lma}
\begin{proof}
 
Take any basis $\langle \mathbf{e}_0, \mathbf{e}_1 \rangle = H_1(L)=\Z^2$ where $\mathbf{e}_0$ is in the kernel of $H_1(L) \to H_1(V)=\Z.$ It suffices to show that $\mathbf{e}_1$ satisfies the property that any disc $(D,\partial D) \to (\CP^2 \setminus \ell_\infty,L)$ with boundary in the homology class $\mathbf{e}_1$ must satisfy $D \bullet C \neq 0.$

Using Proposition \ref{prp:solidtorus} applied to a suitable broken line, we produce a solid torus inside $\CP^2 \setminus \ell_\infty$ with boundary equal to $L,$ and which is foliated by $J$-holomorphic discs $D(\theta)$ of Maslov index two. The compatible almost complex structure $J$ on $\CP^2$ can moreover be chosen so that $C \cup \ell_\infty$ is $J$-holomorphic.

Under either of the assumptions of Theorem \ref{thm:main}, positivity of intersection \cite{McDuff:LocalBehaviour} implies that the discs foliating the above solid torus must intersect $C$ positively. Since $[\partial D(\theta)]=k\mathbf{e}_0+l\mathbf{e}_1$ the statement follows.
\end{proof}

\begin{lma}
\label{lma:geodclass}
Assume that either Condition (1) or (2) of Theorem \ref{thm:main} is satisfied for the Lagrangian torus $L \subset (V,\omega_{\OP{FS}}),$ and consider a basis $\langle \mathbf{e}_0,\mathbf{e}_1\rangle=H_1(L)$ as in Lemma \ref{lma:goodbasis}. Any broken $J_\infty$-holomorphic conic satisfying the given tangency conditions at $q_i$ consists of components all whose asymptotics are geodesics in homology classes of the form $k\mathbf{e}_0 \in H_1(L)$ with $k \in \Z \setminus \{0\}.$
\end{lma}
\begin{rmk}
In the above lemma we make no assumptions on the genericity of $J_\infty.$
\end{rmk}
\begin{cor}
The same is true also for an almost complex structure $J_\infty$ on $\CP^2 \setminus (L \cup L')$ obtained by stretching the neck also around an additional Lagrangian torus $L' \subset V \setminus L,$ with the only caveat that any component then also may have additional punctures asymptotic to $L'.$
\end{cor}
\begin{proof}
The statements follow by positivity of intersection with $C$ together with topological considerations in conjunction with Lemma \ref{lma:goodbasis}. Recall that, as prescribed by the tangency conditions satisfied by the broken conic, it intersects $C$ with algebraic intersection index $+4$ at the two points $q_i \in C,$ $i=1,2.$ In other words, there are precisely two intersection of this building and the unbroken conic $C,$ both which occur at the two points $q_i \in \ell_\infty.$

Consider a puncture $p$ appearing as an asymptotic of a top-level component $A$ in the building, where the asymptotic at $p$ corresponds to a geodesic in the homology class $k\mathbf{e}_0+l\mathbf{e}_1 \in H_1(L).$ Compactify $A$ to a continuous chain with boundary on $L$ in the homology class $-(k\mathbf{e}_0+l\mathbf{e}_1)\in H_1(L)$ (note the sign!).
Choose an arbitrary null-homology $B$ of $k\mathbf{e}_0+l\mathbf{e}_1 \in H_1(L)$ inside $\CP^2 \setminus \ell_\infty,$ and create a cycle $A \cup B.$

In this manner we have produced a cycle in $H_2(\CP^2)$ of degree $d=A \bullet \ell_\infty,$ where either $d=0,$ $1$ or $2.$ The intersection number of the cycle and $C$ is precisely $2d+la_1=2d.$ The first term arises from the intersections of $A$ and $\ell_\infty$ (these occur at $\{q_1,q_2\} \subset \ell_\infty$ by positivity of intersection) while the second term arises from the intersections of $B$ and $C.$ Since $a_1>0$ by Lemma \ref{lma:goodbasis} we must have $l=0$ as sought.
 
\end{proof}

\subsection{Constructing a compatible fibration of conics}
\label{sec:compatible}
 We are now ready to state the result concerning the existence of a pseudoholomorphic conic foliation compatible with a Lagrangian torus $L \subset (V,\omega_{\OP{FS}})$ that satisfies one of the assumptions of Theorem \ref{thm:main}. We postpone its proof to Subsection \ref{pf:CompatibleFibration} below. 
\begin{thm}
\label{thm:CompatibleFibration}
After a Hamiltonian isotopy of $L \subset V \setminus (\ell_\infty \cup C \cup C_{\OP{nodal}}),$ there exists a tame almost complex structure $J$ on $(\CP^2,\omega_{\OP{FS}}),$ where $J=i$ is standard near the divisor $\ell_\infty \cup C \cup C_{\OP{nodal}},$ and for which an application of Theorem \ref{thm:Lefschetz} gives rise to a fibration $f_J \colon (V,\omega_{\OP{FS}}) \to \C$ by $J$-holomorphic conics which is compatible with the Lagrangian torus in the following sense:
\begin{enumerate}
\item The restriction $f_J|_L$ makes the torus a smooth $S^1$-fibre bundle over the embedded closed curve
$$\sigma\coloneqq f_J(L) \subset \C \setminus \{0=f_J(C_{\OP{nodal}}),1=f_J(C)\}$$
which has winding number one around $1 \in \C.$ The fibres of $f_J|_L$ are, moreover, closed curves inside $L$ which are contractible inside $V$ and of Maslov index zero;
\item The fibre of $f_J^{-1}(s)$ for any $s \in \sigma$ coincides with the standard fibre $f^{-1}(s)$ outside of a compact subset; and
\item The curve $\sigma$ is either disjoint from the curve $[0,1] \subset \C$ or is of the from $s_0+i[-\delta,\delta]$ in some neighbourhood of $[0,1],$ where $s_0 \in (0,1).$ Furthermore, we may assume that $f_J=f$ holds inside the same neighbourhood.
\end{enumerate}
\end{thm}
\begin{rmk}
In the case when $L \subset \CP^2 \setminus (\ell_\infty \cup C \cup C_{\OP{nodal}})$ already is satisfied, then it is not necessary to apply a Hamiltonian isotopy to $L$ in order to achieve Parts (1) and (2) of Theorem \ref{thm:CompatibleFibration}.
\end{rmk}
In view of Part (3) of the above theorem there are two possibilities for a Lagrangian torus: the closed curve $\sigma \subset \C \setminus \{0,1\}$ over which the torus is fibred either has winding number around $0 \in \C$ equal to $w=1$ (in the case when it is disjoint from $[0,1] \subset \C$) or $w=0$ (in the case when it intersects $[0,1] \subset \C$ in a single point). When $w=1$ and $0$ we say that the torus is in in \emph{Clifford position} (see Figure \ref{fig:fibration-chekanov2}) and \emph{Chekanov position} (see Figure \ref{fig:fibration-clifford2}), respectively. These Lagrangian tori will later be shown to be Hamiltonian isotopic to standard tori of the corresponding types.

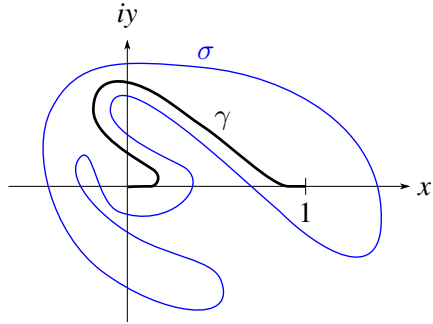
\begin{figure}[htp]
\begin{center}
\vspace{6mm}
\labellist
\pinlabel $1$ at 140 52
\pinlabel $\color{blue}\sigma$ at 93 127
\pinlabel $\gamma$ at 101 95
\pinlabel $x$ at 196 63
\pinlabel $iy$ at 56 143
\endlabellist
\includegraphics[scale=0.8]{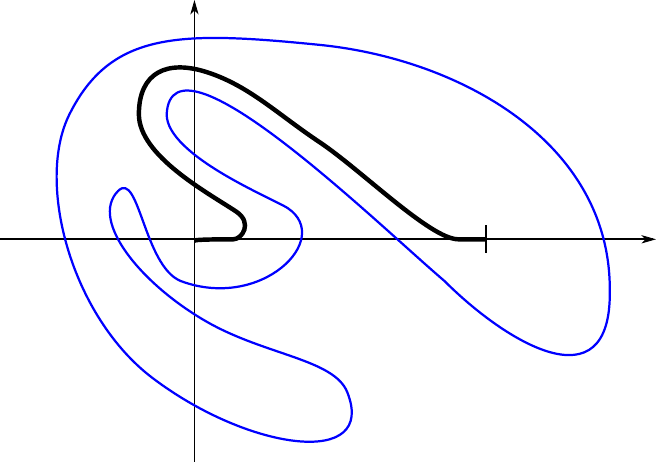}
\caption{The image of a torus in Clifford in position under a compatible fibration, before carrying out the deformation to make the fibration standard above $[0,1] \subset \C.$}
\label{fig:fibration-clifford2}
\end{center}
\end{figure}

\begin{figure}[htp]
\begin{center}
\vspace{6mm}
\labellist
\pinlabel $1$ at 140 52
\pinlabel $\color{blue}\sigma$ at 106 126
\pinlabel $\gamma$ at 101 95
\pinlabel $x$ at 196 63
\pinlabel $iy$ at 56 142
\endlabellist
\includegraphics[scale=0.8]{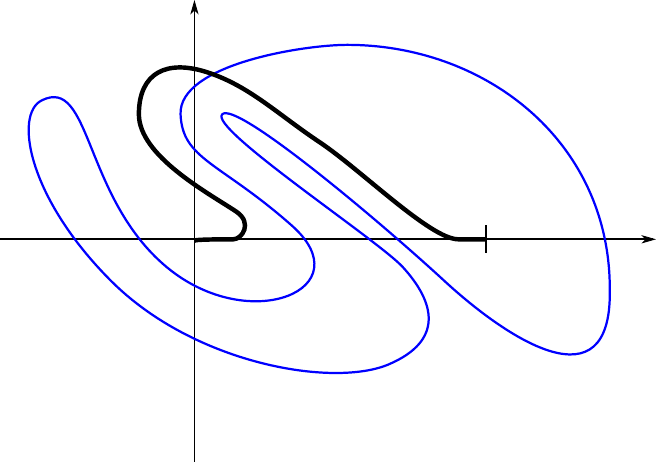}
\caption{The image of a torus in Chekanov position under a compatible fibration, before carrying out the deformation in order to make the fibration standard above $[0,1] \subset \C.$}
\label{fig:fibration-chekanov2}
\end{center}
\end{figure}
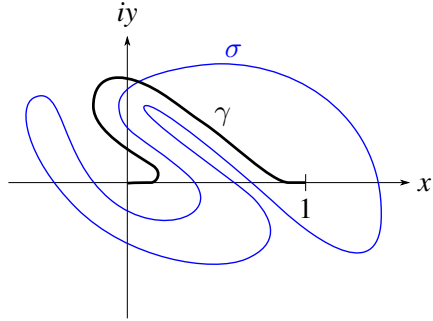

We now state some consequences of Theorem \ref{thm:CompatibleFibration}.

\begin{cor}
\label{cor:conditions}
Under the assumption of either Condition (1) or (2) of Theorem \ref{thm:main}, the Lagrangian torus $L \subset \CP^2 \setminus (\ell_\infty \cup C)$ is smoothly isotopic inside the same subset to a fibre of the Lagrangian fibration $\Pi_s \colon V \to (-1,1) \times (0,+\infty)$ through totally real tori. In particular, Conditions (1) and (2) of Theorem \ref{thm:main} are equivalent for embedded Lagrangian tori.
\end{cor}
\begin{proof}
 The smooth isotopy is readily constructed by first moving the torus within the $S^1$-family of conic fibres $\sigma$ to a Lagrangian torus that lives near, say, the point $q_1.$ The latter Lagrangian torus can then be smoothly isotoped through Lagrangian tori compatible with the standard Lefschetz fibration $f \colon \CP^2 \setminus (\ell_\infty \cup C ) \to \C$ to a standard representative. The first part of the isotopy is not necessarily through Lagrangian tori, but at least it is through \emph{totally real} tori. The reason is that, since these tori all live inside a three-dimensional hypersurface $\sigma$ foliated by pseudoholomorphic curves, while they by construction have no full tangency to any of the curves, they cannot have any complex tangencies. 

We end by noting that Conditions (1) and (2) of Theorem \ref{thm:main} are equivalent merely by the existence of this formal Lagrangian isotopy; the standard tori are obviously in the class of the generator of $H_2(V),$ and all continuous discs of nonzero Maslov index on them must intersect $C$ with a nonzero algebraic intersection number.

\end{proof}
Later we will establish that the tori actually are Lagrangian isotopy, and not just formally Lagrangian isotopic. Recall that this already is known to be the case inside the larger space $(\CP^2 \setminus \ell_\infty,\omega_{\OP{FS}}) \supset V$ by the main result of \cite{Dimitroglou:Isotopy}; constructing a Lagrangian isotopy in the complement of the divisor $C$ will turn out to require some additional work.

The property of having vanishing Maslov class for a torus satisfying either assumption of Theorem \ref{thm:main} has the following important and nontrivial consequence.
\begin{prp}
\label{prp:nondisp}
Any Lagrangian torus inside $(V,\omega_{\OP{FS}})$ which satisfies either of the assumptions of Theorem \ref{thm:main} is not Hamiltonian displaceable inside $V.$
\end{prp}
\begin{proof}
By Corollary \ref{cor:conditions} the Maslov class of the torus $L$ vanishes when considered inside $V.$

Since the expected dimension of a pseudoholomorphic disc of Maslov index zero is equal to $-1,$ there are no such somewhere injective discs when the almost complex structure is generic; this follows by a standard argument, see e.g.~\cite{McDuff:J}. Then, by Lazzarini's result \cite{Lazzarini:RelativeFrames}, it is the case that the space of nonconstant pseudoholomorphic discs with boundary on $L$ is \emph{empty} for such a generic almost complex structure. Indeed, otherwise we would be able to extract a somewhere injective pseudoholomorphic disc by alluding to Lazzarini's result.

The Hamiltonian non-displaceability is then a direct consequence of the main result of \cite{Chekanov:LagrangianIntersections}. In order to see that this result can be applied, we note that the complement of a union of positive pseudoholomorphic divisors in a closed symplectic manifold is a noncompact but tame symplectic manifold.
\end{proof}

\subsection{The proof of Theorem \ref{thm:CompatibleFibration}}
\label{pf:CompatibleFibration}

We start by performing a Hamiltonian isotopy in order to disjoin $L$ from the standard nodal conic; this is possible by Lemma \ref{lma:DisjoinNodal}.

Now we assume that $J_\infty$ is chosen generically as in Lemma \ref{lma:planecyl}, so that there are no pseudoholomorphic planes inside $\CP^2 \setminus L$ of negative index. We may assume that the nodal conic $\ell_1 \cup \ell_2$ tangent to $C$ at $\{q_1,q_2\} \subset \ell_\infty$ is $J_\infty$-holomorphic. 
\begin{prp}
\label{prp:characterisation}
Consider the basis $\langle \mathbf{e}_0,\mathbf{e}_1\rangle = H_1(L)$ as in Lemma \ref{lma:goodbasis}. Any broken conic tangent to $C$ at $q_i \in \ell_\infty,$ $i=1,2,$ which appears as the limit of pseudoholomorphic conics when stretching the neck satisfies the following: its top level consists of precisely two planes $A_i \subset \CP^2 \setminus L,$ $i=1,2,$ where
\begin{itemize}
\item $A_i \cap \ell_\infty=\{q_i\},$ where it satisfies the tangency $v_i \subset T_{q_i}\CP^2;$ and
\item $A_i$ is asymptotic to a geodesic on $L$ in the class $(-1)^{i+1} \mathbf{e}_0$ (after a suitable choice of sign for the generator $\mathbf{e}_0$);
\end{itemize}
while the remaining components consists of a single cylinder contained inside $T^*L.$ Moreover, for the above choice of sign for $\mathbf{e}_0,$ the disc $(D,\partial D) \to (V,L)$ with $[\partial D]=\mathbf{e}_0$ is of Maslov index zero and satisfies $D \bullet \ell_i=(-1)^{i}.$ See Figure \ref{fig:breaking}.
\end{prp}
\begin{proof}
 
We begin by noting that a broken conic intersects the smooth conic $C$ precisely at the two points $q_i\in \ell_\infty,$ $i=1,2$; this is a consequence of positivity of intersection \cite{McDuff:LocalBehaviour}.

First we argue that a broken conic does not contain a plane that is disjoint from $\ell_\infty.$ Under the assumption of Condition (2) of Theorem \ref{thm:main} this is immediate from Lemma \ref{lma:planecyl}.

Under the assumption of Condition (1) we must work a bit harder: By Lemma \ref{lma:geodclass} we can conclude that the plane appearing in the broken conic must be asymptotic to a geodesic in the class $k\mathbf{e}_0,$ $k \neq 0,$ and that it is disjoint from $C.$ In fact, since the plane is of index one by Lemma \ref{lma:planecyl}, we must have $k=\pm 1.$ Positivity of intersection \cite{McDuff:LocalBehaviour} together with Lemma \ref{lma:goodbasis} now implies that the moduli space containing this plane is compact, and hence homeomorphic to $S^1$; if the family of planes break in the sense of SFT-compactness \cite{Bourgeois:Compactness}, then we can readily extract a different plane in the compactification which is asymptotic to an orbit in the homology class $k'\mathbf{e}_0+l'\mathbf{e}_1$ with $l' \neq 0.$ (The non-zero intersection number with $C$ is then a contradiction.) We then argue as in \cite{Dimitroglou:Extremal}: by alluding to automatic transversality \cite{Wendl:Automatic} we can then use this moduli space to construct continuous chain inside $\CP^2 \setminus \ell_\infty$ with boundary equal to $m[L]$ for some $m \neq 0,$ i.e.~yielding a null-homology of $m[L].$ By positivity of intersection, this chain is moreover disjoint from $C,$ which is in contradiction with Condition (1).

By the above, together with Lemmas \ref{lma:planecyl} and \ref{lma:geodclass}, we thus conclude that the broken conic consists of precisely the two planes $A_i,$ $i=1,2,$ in its top level where $A_i \cap \ell_\infty=q_i.$

We then show that any disc $(D,\partial D) \to (B^4,L)$ with boundary in homology class $[\partial D]=\mathbf{e}_0 \in H_1(L)$ must satisfy $D \bullet \ell_i = (-1)^i,$ after an appropriate choice of sign for the generator. Indeed, we can complete the plane $A_i$ intersecting $q_i$ with a single puncture asymptotic to $k\mathbf{e}_0$ to a cycle in the class of a line by adding a $-k$ number of such discs, and then use the fact that
$$A_i \bullet \ell_j = \begin{cases} 2, & i=j,\\
0, & i \neq j,
\end{cases}
$$
where again we have alluded to positivity of intersection. From this we also conclude that the plane intersecting $q_i$ must be primitive and, moreover, asymptotic to a geodesic in class $(-1)^{i+1}\mathbf{e}_0.$

It now follows by a topological argument that the remaining top components, i.e.~those which are disjoint from $\ell_\infty,$ must have total symplectic area equal to
$$ \int_{\mathbf{e}_0} \lambda_{\OP{std}}-\int_{\mathbf{e}_0}\lambda_{\OP{std}}=0$$
Hence, there can be no top level components disjoint from $\ell_\infty$ and, again by positivity of intersection, the top level component consists of precisely the two planes $A_i,$ $i=1,2.$

 A final index calculation shows that $\mathbf{e}_0$ is of Maslov index zero. 
\end{proof}

We now proceed with the proof of Theorem \ref{thm:CompatibleFibration}.

(1): We run a neck stretching with $J_\infty=i$ near $\ell_\infty \cup C_{\OP{nodal}} \cup C.$ The SFT compactness theorem \cite{Bourgeois:Compactness}, \cite{Cieliebak:Compactness}, together with Theorem \ref{thm:Lefschetz} allows us to extract \emph{broken} pseudoholomorphic conics as limits from the moduli spaces $\mathcal{M}_{J_\tau}(v_1,v_2).$ In view of Proposition \ref{prp:characterisation} this conic consists of precisely two planes $A_i,$ $i=1,2,$ in the top level, where the plane $A_i$ satisfies the tangency $v_i$ at $q_i \in \ell_\infty.$ Here we must use a generic almost complex structure $J_\infty.$

The rest of the argument follows the same ideas as \cite[Section 5.4]{Dimitroglou:Isotopy}; we need to vary the components in the top level of the obtained broken conic in order to produce a whole one-parameter family of broken conics.

 By the non-existence of planes of Maslov index two inside $V \setminus L$ that are asymptotic to $L,$ as follows from Proposition \ref{prp:characterisation} (also see Lemma \ref{lma:goodbasis}), we conclude that the planes $A_i$ are contained inside compact components of its moduli spaces of planes satisfying a tangency to $C$ at $\{q_1,q_2\} \subset \ell_\infty.$ Namely, a broken configuration arising as the SFT limit of such planes would consist of at least one plane contained in the top level $\CP^2 \setminus L$ and which is disjoint from $\ell_\infty$ and $C$ by positivity of intersection. However, since the only geodesic on $L$ which is contractible inside $V$ has vanishing Maslov class by Proposition \ref{prp:characterisation}, this contradicts the genericity of the almost complex structure chosen.

Since automatic transversality is satisfied by the main result of \cite{Wendl:Automatic}, the component of the moduli space of planes tangent to $C$ at $q_i$ that contains the plane $A_i$ is diffeomorphic to $S^1$ for each $i=1,2.$ In order to achieve automatic transversality with the tangency condition, we must argue as in the proof of Lemma \ref{lma:auttrans} (but for one tangency condition instead of two). More precisely, we need to infer that the infinitesimal variations of solutions that vanish to order two at the point $q_i,$ must vanishes to order precisely two there, while vanishing nowhere else (including asymptotically at the puncture). This is a consequence of positivity of intersection together with the computation of the asymptotic intersection number made in the subsequent paragraph.

The argument from \cite[Lemma 5.13]{Dimitroglou:Isotopy}, based upon the computations of the asymptotic winding of the eigenvectors from \cite{Hind:SymplecticEmbeddings}, shows that the asymptotic evaluation map is a diffeomorphism onto the space of orbits. The computation is the same in this case, despite the fact that the Fredholm index of the plane $A_i$ is equal to five (instead of one). To that end, we use the fact that the index of this plane again is equal to one, when we consider it together with the tangency constraint at $q_i \in \ell_\infty.$

Now we recall how to exclude two planes having the same asymptotic orbit by the aforementioned calculation: Two different planes asymptotic to the same orbit have a new intersection point after a small holomorphic perturbation (the holomorphic perturbation exists in view of the aforementioned automatic transversality result). In this way, two planes with the same asymptotic and tangency at $q_i$ must coincide; otherwise we could readily construct two cycles in degree two that intersect with algebraic intersection number at least five by assembling suitable pseudoholomorphic planes asymptotic to the torus and both being tangent to $C$ at $q_i,$ $i=1,2.$

Finally, we can use the smoothing procedure from \cite[Section 5.3]{Dimitroglou:Isotopy} in order to assemble the planes in these components of the moduli space to form a $S^1$-family of closed embedded pseudoholomorphic conics, each intersecting $L$ in a closed curve, and all being tangent to $v_i$ at $q_i.$ The claim that the Maslov index of this curve is zero when considered on $L,$ is an immediate consequence of Proposition \ref{prp:characterisation}.

The existence of the global symplectic conic foliation is finally a consequence of Theorem \ref{thm:Lefschetz}.

 The claim about the winding number of $\sigma$ around the projection of the fibre $C$ is an immediate consequence of either of the conditions on $L$ made in the assumptions of Theorem \ref{thm:main}. 

(2): This is a straight forward consequence of the normalisation carried out by Theorem \ref{thm:normalise}.

(3): One can construct an embedded curve $\gamma \subset \C$ which connects $0$ to $1$ and which
\begin{itemize}
\item coincides with $[0,1] \subset \C$ near its boundary,
\item is isotopic to the latter standard embedding through embeddings of the same type, and
\item either is disjoint from $\sigma$ (when the torus is in Clifford position) or intersects it transversely in a single point (when the torus is in Chekanov position).
\end{itemize}
An application of Theorem \ref{thm:normalisepath} finally normalises the conics above the path $\gamma.$

In order to make the fibration standard inside a whole neighbourhood of $f^{-1}([0,1])$ we perform a reiteration of the entire neck stretching argument, while using almost complex structures which are standard in the aforementioned neighbourhood. (Note that the neighbourhood has been made disjoint from $L$ by the previous application of Theorem \ref{thm:normalisepath}.)
\qed

\section{Liouville forms and inflation}

Here we construct a family of Liouville forms on the complement of the holomorphic divisor $D_0 \cup D_1 \cup D_2 \cup D_3 \subset \CP^2,$ where $D_0=\ell_\infty,$ $D_1=\ell_1,$ $D_2=\ell_2,$ and $D_3=C.$ The main use of these different Liouville forms is to deform Lagrangian isotopies to Hamiltonian isotopies, by applications of the corresponding Liouville flows. This process is called inflation, and is described in Subsection \ref{sec:inflation} below. As a part of the construction, we also obtain the natural one-parameter family $\lambda_r$ of Liouville forms on $(V=\CP^2 \setminus (D_0 \cup D_3),\omega_{\OP{FS}})$ that we study in Subsection \ref{sec:cp2liouville} below.

The Fubini--Study K\"{a}hler form normalised so that $\int_{\ell_\infty}\omega_{\OP{FS}}=\pi$ can be expressed as
\begin{eqnarray*}
\lefteqn{ \omega_{\OP{FS}} \coloneqq \frac{1}{4} dd^c\log{(1/(1+\|z_1\|^2+\|z_2\|^2))}}\\
& = & -\frac{i}{2} \partial\overline{\partial}\log{(1/(1+\|z_1\|^2+\|z_2\|^2))} = \frac{i}{2} \partial\overline{\partial}\log{(1+\|z_1\|^2+\|z_2\|^2)} \\
& = & \frac{i}{2(1+\|z_1\|^2+\|z_2\|^2)}\sum_{i=1}^2dz_i \wedge d\overline{z}_i + \frac{-i}{2(1+\|z_1\|^2+\|z_2\|^2)^2}\sum_{i,j}\overline{z}_iz_jdz_i \wedge d\overline{z}_j
\end{eqnarray*}
in the affine chart $\C^2 = \CP^2 \setminus \ell_\infty.$ Here we have used the notation $d^cf(\cdot)\coloneqq df(i\cdot).$

Consider the anti-tautological line bundle $E \coloneqq \OP{Tot}(\mathcal{O}(1)) \to \CP^2$ with sheaf of holomorphic sections $\mathcal{O}(1).$ In other words, the line bundle $E^{\otimes 3}=\det(T\CP^2)$ is the anti-canonical line bundle with sheaf of holomorphic sections $\mathcal{O}(3)$ on $\CP^2.$ We endow $E$ with the Hermitian metric determined by the following condition: for the standard trivialisation over the affine chart $\C^2 = \CP^2 \setminus \ell_\infty$ this metric takes the value
$$\|s\|_E^2=\frac{\|s\|^2}{1+\|z_1\|^2+\|z_2\|^2}$$
on the locally defined section $s \colon \C^2 \to \C.$ We also consider the induced metrics 
$$ \|s\|_{E^{\otimes{n}}}^2=\frac{\|s\|^2}{(1+\|z_1\|^2+\|z_2\|^2)^n}$$
on the $n$:th tensor power $E^{\otimes{n}}$ of the line-bundle.

Take holomorphic sections $s_{\ell_\infty}$ and $s_C$ of $E$ and $E^{\otimes 2}$ respectively, which in the trivialisation over $\C^2 = \CP^2 \setminus \ell_\infty$ are given by
$$ s_{\ell_\infty}(z_1,z_2)=1 \:\:\:\: \text{and} \:\:\:\: s_C(z_1,z_2)=z_1z_2-1. $$
Similarly, there are sections $s_{\ell_i}$ of $E$ which in the trivialisation over the affine coordinate chart $\C^2 = \CP^2 \setminus \ell_\infty$ takes the form $s_{\ell_1}=z_2,$ while $s_{\ell_2}=z_1.$ Note that $(s_{\ell_i})=\ell_i.$

Using the above sections and Hermitian metrics we are now ready define a family of Liouville forms on the complement $\CP^2 \setminus (\ell_\infty \cup C \cup C_{\OP{nodal}})$ of a particular singular divisor. Let $\mathbf{r}=(r_0,\ldots,r_3) \in \R^4$ be numbers satisfying $r_0 + \ldots + r_3=1.$ Using the formula for the curvature we compute
\begin{eqnarray*}
\omega_{\OP{FS}} =\frac{1}{4}dd^c\left(r_0\log{(\|s_{\ell_\infty}\|_{E}^2)}+ r_1\log{(\|s_{\ell_1}\|_{E}^2)} + r_2\log{(\|s_{\ell_2}\|_{E}^2)} + r_3\frac{1}{2}\log{(\|s_C\|_{E^{\otimes 2}}^2)}\right)
\end{eqnarray*}
and hence
\begin{eqnarray*}
\lambda_{\mathbf{r}} \coloneqq \frac{1}{4}d^c\left(r_0\log{(\|s_{\ell_\infty}\|_{E}^2)}+ r_1\log{(\|s_{\ell_1}\|_{E}^2)} + r_2\log{(\|s_{\ell_2}\|_{E}^2)} + r_3\frac{1}{2}\log{(\|s_C\|_{E^{\otimes 2}}^2)}\right)
\end{eqnarray*}
 is a family of Liouville forms defined in the complement of the divisor $\bigcup_{\{i:\: r_i \neq 0\}} D_i.$ 

\subsection{A family of Liouville forms on $V=\CP^2 \setminus(\ell_\infty \cup C)$}

\label{sec:cp2liouville}

In this subsection we endow $(V,\omega_{\OP{FS}})$ with a family of Liouville forms. Also compare to the construction in Section \ref{sec:Liouville}, where the Liouville form is more carefully adapted to the Lagrangian fibrations considered. Setting $r_1=r_2=0,$ $ r_0=1-r,$ $ r_3=r$ in the construction above, we obtain a family
$$\lambda_r\coloneqq(1-r)\frac{1}{4}d^c\log{(\|s_{\ell_\infty}\|_{E}^2)}+r\frac{1}{8}d^c\log{(\|s_C\|_{E^{\otimes 2}}^2)}, \:\: r \in [0,1] ,$$
of Liouville forms for $\omega_{\OP{FS}}$ defined on all of $V.$ Denote by
$$\phi^t_{\lambda_r} \colon (V,\omega_{\OP{FS}}) \to (V,e^{-t}\omega_{\OP{FS}})$$
the corresponding Liouville flow.

\begin{rmk}
\label{rmk:Liouville}
It makes sense to also consider the limit cases $\lambda_r$ with $r=0$ and $r=1.$
\begin{enumerate}
\item The Liouville form $\lambda_0$ is identified with the standard radial Liouville form
is the pull-back of
$$ \lambda_{\OP{std}}=\|\widetilde{z}_1\|^2/2\,d\widetilde{\theta}_1+\|\widetilde{z}_2\|^2/2\,d\widetilde{\theta}_2=i(\widetilde{z}_1d\overline{\widetilde{z}}_1-\overline{\widetilde{z}}_1d\widetilde{z}_1)/4+i(\widetilde{z}_2d\overline{\widetilde{z}}_2-\overline{\widetilde{z}}_2d\widetilde{z}_2)/4$$
on $(B^4,\omega_0=d\lambda_{\OP{std}})$
under the symplectomorphism
$\varphi.$ By Gromov's theorem \cite{Gromov:Pseudo} there are no exact closed Lagrangian embeddings with respect to this Liouville form.
\item The Liouville $\lambda_{1}$ is defined on $(\CP^2 \setminus C,\omega_{\OP{FS}})$ and vanishes along the Lagrangian embedding
$$ \overline{\{z_2=-\overline{z_1}\}} \subset (\CP^2,\omega_{\OP{FS}})$$
of $\RP^2,$ which thus in particular is exact. Note that $\CP^2 \setminus C$ is symplectomorphic to the disc cotangent bundle of $\RP^2.$
\end{enumerate}
\end{rmk}

Different Lagrangian fibres of $\Pi_s$ become exact for different Liouville forms in the family $\lambda_r.$ Here we give a description, using the conventions from Section \ref{sec:ActionProperties} for the computation of the symplectic action. 
\begin{lma}
\label{lma:LiouvilleCP2}
\begin{enumerate}
\item The Lagrangian fibre $\Pi_s^{-1}(0,u_2)$ is exact and, in the case of the immersed sphere $u_2=1$ even strongly exact, precisely for the Liouville form $\lambda_r$ with parameter $ r=(2/\pi)A_s(0,u_2)$ in the above family (note that $ (2/\pi)A_s(0,u_2) \in (0,1)$); 
\item The backwards Liouville flow $\phi^{-t}_{\lambda_r},$ $t \ge 0,$ is complete on $V$ for all $ r \in (0,1);$ and
\item For any $A \in (0,\pi/2),$ the backwards Liouville flow $\phi^{-t}_{\lambda_{ (2/\pi)A}}$ applied to the torus fibre $\Pi^{-1}_s(u_1,u_2)$ rescales both the quantities
$$ \int_{\mathbf{e}_0} \lambda_{\OP{std}}=\pi\cdot u_1 \:\:\:\: \text{and} \:\:\:\: \int_{\mathbf{e}_1} \lambda_{\OP{std}}-A = A_s(u_1,u_2)-A$$
by multiplication with $e^{-t},$ where the function $A_s$ is as given in Lemma \ref{lma:action}. (Here we have used the flow to identify the basis in homology.)
\end{enumerate}
\end{lma}
\begin{proof}
(1): This follows from an explicit computation of the symplectic action of $\mathbf{e}_1 \in H_1(L).$ The crucial identities are
\begin{align*}
& \int_{\mathbf{e}_1} (1/4)d^c\log(\|s_{\ell_\infty}\|^2_E)=\int_{\mathbf{e}_1} \lambda_{\OP{std}}=A_s(u_1,u_2),\\
& \int_{\mathbf{e}_1} (1/8)d^c\log(\|s_C\|^2_E)=A_s(u_1,u_2)-\pi/2.
\end{align*}

(2): The Liouville vector field is equal to the gradient
$$\nabla\left(-(1-r)\frac{1}{4}\log{(\|s_{\ell_\infty}\|_E^2)}-r\frac{1}{8}\log{(\|s_C\|_{E^{\otimes 2}}^2)}\right).$$
Since this function has no critical points outside of a compact subset of $\CP^2 \setminus (\ell_\infty \cup C),$ and since the function moreover blows up along the divisor, the completeness { is now immediate}.

(3): The symplectic action of $\mathbf{e}_0$ coincides with the symplectic area of a suitable disc living inside $V;$ its rescaling properties are thus immediate from the basic property that the Liouville flow rescales the symplectic form by $e^t.$

For the rescaling properties of the symplectic action of $\mathbf{e}_1,$ one can e.g.~use the fact that the fibre $\Pi_s^{-1}(0,u)$ with $A_s(0,u)=A$ is \emph{exact} for the Liouville form $\lambda_{(3/2)A}.$ Hence, the corresponding symplectic actions for this torus, as well as its images under the Liouville flow, vanishes.

For an arbitrary torus $L$ it is now a simple matter of finding a two-dimensional chain with boundary on $L \cup \Pi_s^{-1}(0,u),$ where the boundary on $L$ moreover lives in the class $\mathbf{e}_1$ (and whose symplectic area hence coincides with the symplectic action of $\mathbf{e}_1$ on $L$).
\end{proof}

In the proof of Proposition \ref{prp:LagrangianIsotopy} we will need to use the fact that the backwards Liouville flow corresponding to $\lambda_r$ contracts $V$ into any neighbourhood of the subset $\{\|\widetilde{z}_1\|^2-\|\widetilde{z}_2\|^2=0\} \cap V.$ This is a consequence of the following.
\begin{lma}
\label{lma:LiouvilleComputation}
The Liouville vector field corresponding to $\lambda_r$ for any $r \in [0,1]$ evaluates positively (resp. negatively) on the exterior derivative
$$ d(\|\widetilde{z}_1\|^2-\|\widetilde{z}_2\|^2)=d\left(\frac{\|z_1\|^2-\|z_2\|^2}{1+\|z_1\|^2+\|z_2\|^2}\right)$$
inside the subset $\{\|\widetilde{z}_1\|^2-\|\widetilde{z}_2\|^2>0\}$ (resp. $\{\|\widetilde{z}_1\|^2-\|\widetilde{z}_2\|^2<0\}$) of $\CP^2 \setminus (\ell_\infty \cup C).$
It follows that the same is true also for the Liouville flows corresponding to $\lambda_{\mathbf{r}}$ with $\mathbf{r}=(r-\eta,\eta,\eta,1-r-\eta) \in \R_{>0}^4.$
\end{lma}
\begin{proof}
We prove the fact for the parameters $r=0$ and $r=1$; since the Liouville forms for the remaining parameters are convex interpolations of $\lambda_0$ and $\lambda_1,$ the statement then follows in general.

For the Liouville form $\lambda_0$ the statement follows by Part (1) of Remark \ref{rmk:Liouville}.

For the Liouville form $\lambda_1$ we argue as follows. By the proof of Lemma \ref{lma:LiouvilleCP2} the corresponding Liouville vector field is given as the gradient
$$ -\nabla\frac{1}{8}\log{(\|s_C\|_{E^{\otimes 2}}^2)}$$
for the Fubini--Study metric. It thus suffices to show that the gradient
$$V \coloneqq \nabla \frac{\|z_1\|^2-\|z_2\|^2}{1+\|z_1\|^2+\|z_2\|^2}=2(z_1\partial_{z_1}+\overline{z}_1\partial_{\overline{z}_1}-z_2\partial_{z_2}-\overline{z}_2\partial_{\overline{z}_2})$$
evaluates positively (resp. negatively) on
$$ d(-\|s_C\|_{E^{\otimes 2}}^2)=-d\frac{\|z_1z_2-1\|^2}{(1+\|z_1\|^2+\|z_2\|^2)^2}$$
in the subset $\{\|\widetilde{z}_1\|^2-\|\widetilde{z}_2\|^2>0\}$ (resp. $\{\|\widetilde{z}_1\|^2-\|\widetilde{z}_2\|^2<0\}$), since
$$ d(-\|s_C\|_{E^{\otimes 2}}^2)(V)=\frac{4\|z_1z_2-1\|^2(\|z_1\|^2-\|z_2\|^2)}{(1+\|z_1\|^2+\|z_2\|^2)^3}=\frac{4\|z_1z_2-1\|^2}{1+\|z_1\|^2+\|z_2\|^2}(\|\widetilde{z}_1\|^2-\|\widetilde{z}_2\|^2).$$
\end{proof}

\subsection{Inflation via the Liouville flow}
\label{sec:inflation}

 Recall that a Lagrangian isotopy is generated by a Hamiltonian if and only if the corresponding symplectic flux-path vanishes identically. Since we are considering the case of an exact symplectic manifold $(X,\omega=d\lambda),$ the latter condition can be formulated as having a Lagrangian isotopy under which the pullback of $\lambda$ to $L$ is a path of closed one-forms on $L$ that are \emph{constant in cohomology}. In cases when we can find suitable symplectic divisors in the complement of the Lagrangian isotopy, the technique of `inflation' along the divisors can in favourable situations be used to deform a Lagrangian isotopy into one with vanishing flux-path.

\begin{rmk}
 In the proof of the nearby Lagrangian conjecture for the torus \cite[Theorem 7.1]{Dimitroglou:Isotopy} the Lagrangian isotopy was turned into a Hamiltonian isotopy by applying the fibre-wise addition of suitable Lagrangian sections in $T^*\T^2$ (i.e.~deformation by a non-exact symplectomorphisms). That procedure can be reinterpreted as an inflation as well; first one compactifies $DT^*\T^2$ to $S^2 \times S^2,$ and then one inflates along appropriate lines inside the compactifying divisor.
\end{rmk}

In our case the inflation will be performed along the divisors $D_0=\ell_\infty,$ $D_1=\ell_1,$ $D_2=\ell_2,$ and $D_3=C$ living inside $\CP^2,$ with different parameters $r_0,r_1,r_2,r_3\ge 0$; one parameter for each of the divisors. In practice we found it most efficient to perform the inflation by constructing a family of Liouville forms on $\CP^2 \setminus (D_0 \cup \ldots \cup D_3)$ parametrised by $\mathbf{r}=(r_0,r_1,r_2,r_3) \in \R^4.$ Again we emphasise that it also would be possible to use a more hands-on approach as in e.g.~\cite[Section 6]{Dimitroglou:Isotopy}.

 We start by investigating where these different Liouville flows can be integrated
\begin{lma}
\label{lma:WellDefLiouville}
Assume that $r_i \ge 0$ holds for all $i=0,1,2,3.$ Then the backwards Liouville flow of $\lambda_{\mathbf{r}}$ exists for all times when restricted to $\CP^2 \setminus \cup_{\{i:\: r_i \neq 0\}} D_i.$
\end{lma}
 
The Liouville vector field that corresponds to $\lambda_{\mathbf{r}}$ is equal to the negative gradient $-\nabla f_{\mathbf{r}}/4$ (taken w.r.t.~the Fubini--Study metric) of the function
$$ f_{\mathbf{r}}\coloneqq r_0 \log{(\|s_{\ell_\infty}\|_{E}^2)} + r_1\log{(\|s_{\ell_1}\|_{E}^2)}+r_2\log{(\|s_{\ell_2}\|_{E^{\otimes 3}}^2)}+r_3(1/2)\log{(\|s_C\|_{E^{\otimes 3}}^2)}.$$
 Hence Lemma \ref{lma:WellDefLiouville} is a direct consequence of the following behaviour of $f_{\mathbf{r}}$ that we now proceed to establish. 
\begin{lma}
\label{lma:gradient}
Whenever all $r_i \ge 0$ are nonnegative, $f_{\mathbf{r}} \le C_{\mathbf{r}}$ is bounded uniformly from above by a constant depending continuously on $\mathbf{r}.$ If, moreover, $r_i > 0$ holds for some $i \in \{0,1,2,3\},$ then for any $c < 0,$ there exists a neighbourhood $U_i \supset D_i$ inside $B^4$ on which $f_{\mathbf{r}}|_{U_i \setminus (\ell_\infty \cup C \cup C_{\OP{nodal}})} \le c.$
\end{lma}
\begin{proof}
The uniform bound is straight forward to show for any of the four terms in the above expression of $f_\mathbf{r},$ using the property that all sections $s_{\ell_\infty},$ $s_{\ell_i},$ and $s_C$ are holomorphic.

Concerning the behaviour near the divisor $D_i$ in the case when $r_i>0,$ the claim is a standard consequence of the fact that the holomorphic section corresponding to the $i$:th term vanishes along that divisor, together with the above uniform bound.
\end{proof}
We will now see that the effect of the backwards Liouville flow $\phi^{-t}_{\lambda_{\mathbf{r}}}$ corresponding to $\lambda_{\mathbf{r}}$ on a Lagrangian $L \subset \CP^2 \setminus (C \cup \ell_\infty \cup C_{\OP{nodal}})$ in the complement of this singular divisor corresponds to the Lagrangian isotopy induced by performing an appropriate `inflation' along the same divisors. In other words, we need to investigate which effect the Liouville flow has on the symplectic action of the Lagrangian.

It is clear that the change of symplectic action under the flow only depends on the homotopy class of the Lagrangian. Here we consider two cases, characterised in terms of the existence of a basis satisfying certain properties. For a Lagrangian torus $L \subset \CP^2 \setminus (\ell_\infty \cup C \cup C_{\OP{nodal}})$ any basis
$$H_1(L)=\langle \mathbf{f}_1,\mathbf{f}_2 \rangle$$
is induced by a basis
$$F_i \in H_2(B^4,L), \:[\partial F_i]=\mathbf{f}_i, \:\: i=1,2,$$
of the corresponding relative homology group. In the following we are interested in two types of bases that are characterised as follows:
\begin{itemize}
\item \emph{A basis of Clifford type:} $F_i \bullet C =1$ for $i=1,2,$ while $F_i \bullet \ell_j = \delta_i^j$; or
\item \emph{A basis of Chekanov type:}
\begin{itemize}
\item $F_0 \bullet \ell_i=(-1)^i,$ $i=1,2,$ while $F_0 \bullet C=0,$ and
\item $F_1 \bullet C=1,$ while $F_1 \bullet \ell_i=0$ for $i=1,2;$
\end{itemize}
\end{itemize}
The fibres of $\Pi_s$ of Clifford and Chekanov type also admit bases of the respective types; In the case of the Clifford torus both basis elements have Maslov index two, while in the Chekanov case $F_0$ and $F_1$ have Maslov index zero and two, respectively.
\begin{rmk}
 
\label{rmk:bases}
\begin{enumerate}
\item The property of admitting a basis as above only depends on the homotopy class of the torus inside $\CP^2 \setminus (\ell_\infty \cup C \cup C_{\OP{nodal}}).$ Furthermore, the two cases are mutually exclusive for a given torus.
\item The basis of Chekanov type coincides with the basis considered in Lemmas \ref{lma:whitneyaction} and \ref{lma:basis0} for the same tori, i.e.~$\mathbf{f}_i=\mathbf{e}_i.$
\end{enumerate}
\end{rmk}

The Liouville flow has the following effect on the symplectic action of a Lagrangian torus in two aforementioned cases:
\begin{prp}
\label{prp:action}
Take some
$$\mathbf{r} \in \{(r_0,\ldots,r_3); \:\: r_i \ge 0, \: r_0 +\ldots+r_3=1\}$$
and consider the symplectic action
$$A_i(t)\coloneqq\int_{\mathbf{f}_i} \lambda_{\OP{std}}, \:\: i=1,2,$$
of the image $L_t$ of a Lagrangian torus $L \subset \CP^2 \setminus (\ell_\infty \cup C \cup C_{\OP{nodal}})$ under the time-$(-t)$ Liouville flow of $\lambda_{\mathbf{r}},$ where $L_t$ moreover is assumed to be contained inside the same subset for $t \in [0,t_0].$ Then
\begin{itemize}
\item if $\langle \mathbf{f}_1,\mathbf{f}_2 \rangle$ is a basis of Clifford type, then
\begin{align*}
& A_1(t)=\pi(r_1+r_3/2)+e^{-t}(A_1(0)-\pi(r_1+r_3/2)),\\
& A_2(t)=\pi(r_2+r_3/2)+e^{-t}(A_2(0)-\pi(r_2+r_3/2)),
\end{align*}
for all $t \in [0,t_0],$ while
\item if $\langle \mathbf{f}_1,\mathbf{f}_2 \rangle$ is a basis of Chekanov type, then
\begin{align*}
& A_1(t)=\pi(r_1-r_2)+e^{-t}(A_1(0)-\pi(r_1-r_2)),\\
& A_2(t)= \pi r_3/2 +e^{-t}(A_2(0)-\pi r_3/2),
\end{align*}
for all $t \in [0,t_0],$
\end{itemize}
where we have used the Lagrangian isotopy to make the identification $H_1(L_t) \cong H_1(L).$
\end{prp}
\begin{rmk}
In view of Lemma \ref{lma:gradient}, assuming that $r_0,\ldots,r_3>0$ is sufficient in order to guarantee that $L_t \subset \CP^2 \setminus (\ell_\infty \cup C \cup C_{\OP{nodal}})$ is fulfilled for the entire backwards Liouville flow.
\end{rmk}
\begin{proof}
The statement is a consequence of the following more general properties. Consider a surface $\Sigma \subset \CP^2$ which intersects $D_i$ transversely with algebraic intersection number equal to $n_i \in \Z,$ while having boundary contained in the complement of $D_0 \cup \ldots \cup D_3.$ We assume that all intersections $\Sigma \pitchfork D_i$ are transverse. The Liouville flow $\phi^{-t}_{\mathbf{r}}$ with $r_i>0$ acts on $\Sigma \setminus (D_0 \cup \ldots \cup D_3)$ by an isotopy, and this punctured surface can be suitably completed near $D_0 \cup \ldots \cup D_3$ to extend this isotopy to a smooth isotopy $\Sigma_t \subset \CP^2$ of surfaces, where $\Sigma_0=\Sigma,$ and $\Sigma_t \bullet D_i=n_i$ for all $t \ge 0.$ The general claim that shows the statement is that
$$ \int_{\Sigma_t}\omega_{\OP{FS}}=e^{-t}\int_\Sigma\omega_{\OP{FS}}+(1-e^{-t})(n_0r_0+n_1r_1+n_2r_2+n_3r_3/2)$$
whenever $r_i > 0.$ This holds, since the area of any such surface can be computed to be equal to
$$ \int_{\Sigma_t}\omega_{\OP{FS}}=\int_{\partial \Sigma_t} \lambda_\mathbf{r} + n_0r_0+n_1r_1+n_2r_2+n_3r_3/2,$$
while $(\phi^{-t}_{\mathbf{r}})^*\, \lambda_{\mathbf{r}}=e^{-t}\lambda_{\mathbf{r}}$ is satisfied. (The contributions for each transverse intersection with a divisor, i.e.~the latter terms, are preferably computed for a model problem consisting of a small complex disc contained in a fibre normal to the divisor.)
\end{proof}
\begin{cor}
\label{cor:flux}
 
For $\mathbf{t} \in [0,1]^k$ smoothly parametrising a family
$$\widetilde{L}_{\mathbf{t}}:=\Pi_{s(\mathbf{t})}^{-1}(u_1(\mathbf{t}),u_2(\mathbf{t})) \subset \CP^2 \setminus (\ell_\infty \cup C \cup C_{\OP{nodal}}), \:\: u_2(\mathbf{t}) \neq 1,$$
of Lagrangian standard torus fibres, there exists smooth functions
$$\mathbf{r}(\mathbf{t}) \in \{(r_0(\mathbf{t}),\ldots,r_3(\mathbf{t})); \:\: r_i(\mathbf{t})>0, \: r_0(\mathbf{t}) +\ldots+r_3(\mathbf{t})=1\}$$
and $\alpha(\mathbf{t}) \in \R_{\ge 0}$ for which the path
$$L_{\mathbf{t}}\coloneqq \phi^{-\alpha(\mathbf{t})}_{\mathbf{r}(\mathbf{t})}(\Pi_{s(0)}^{-1}(u_1(0),u_2(0))) \subset \CP^2 \setminus (\ell_\infty \cup C \cup C_{\OP{nodal}})$$
of Lagrangian tori realises the same symplectic flux-path as $\widetilde{L}_{\mathbf{t}}.$ We may moreover assume that $\alpha(\mathbf{t})=0$ holds whenever $(u_1(\mathbf{t}),u_2(\mathbf{t}))=(u_1(0),u_2(0)).$
\end{cor}
\begin{proof}
 
We prove the statement in the case of tori of tori of Chekanov type, i.e.~when $u_2(\mathbf{t}) \in (0,1).$ The case of tori of Clifford type is similar. Without loss of generality we can consider the case when $s(\mathbf{t}) \equiv \pi/2-\epsilon$ for $\epsilon>0$ sufficiently small. Indeed, we may realise the path $\widetilde{L}_{\mathbf{t}}$ of tori by tori in the same homotopy class, and of the same action, while the path $s(\mathbf{t})$ is constant.

Denote by $(A_1(\mathbf{t}),A_2(\mathbf{t}))$ the symplectic actions of the tori $\widetilde{L}_{\mathbf{t}}$ with respect to the basis of Chekanov type, and write
$$(a_1(\mathbf{t}),a_2(\mathbf{t})) \coloneqq (A_1(\mathbf{t}),A_2(\mathbf{t}))-(A_1(0),A_2(0)).$$
For a small number $\delta>0$ we choose a smooth bump function $\beta \colon \R_{\ge0} \to [-\delta^{-1},\delta]$ which satisfies $\beta(t)=t$ near $t=0,$ while $\beta(t)=-\delta^{-1}$ for all $t \ge \delta.$ Then we construct the smooth and nonnegative function
$$\alpha(\mathbf{t}) \coloneqq -\log{\left(\frac{e^{\beta(\|(a_1(\mathbf{t}),a_2(\mathbf{t}))\|^2)}}{1+e^{\beta(\|(a_1(\mathbf{t}),a_2(\mathbf{t}))\|^2)}}\right)}$$
In view of Lemma \ref{prp:action} it then suffices to take a suitable path $\mathbf{r}(\mathbf{t})$ that satisfies
\begin{eqnarray*}
\lefteqn{(\pi(r_1(\mathbf{t})-r_2(\mathbf{t})),\pi r_3(\mathbf{t})/2)=}\\
& & \left(A_1(0)+\left(1+e^{\beta(\|(a_1(\mathbf{t}),a_2(\mathbf{t}))\|^2)}\right)a_1(\mathbf{t}),A_2(0)+\left(1+e^{\beta(\|(a_1(\mathbf{t}),a_2(\mathbf{t}))\|^2)}\right)a_2(\mathbf{t})\right).
\end{eqnarray*}
That this is possible when $\delta>0$ in the construction of the above bump-function is chosen to be sufficiently small can be readily seen using Lemma \ref{lma:action}. Namely, by this lemma the values of $(A_1(\mathbf{t}),A_2(\mathbf{t}))$ and hence also
$$(A_1(0),A_2(0))+\left(1+e^{\beta(\|(a_1(\mathbf{t}),a_2(\mathbf{t}))\|^2)}\right)(a_1(\mathbf{t}),a_2(\mathbf{t})),$$
are restricted to the interior of the polytope shown in Figure \ref{fig:action}; also c.f.~Part (2) of Remark \ref{rmk:bases}.
\end{proof}

\subsection{Flux-paths that determine the Hamiltonian isotopy class}
 
A Lagrangian isotopy with a nonzero symplectic flux-path cannot be generated by a Hamiltonian isotopy. Furthermore, it is known that two Lagrangian isotopies starting at the same Lagrangian torus (i.e.~the Clifford torus) and realising the \emph{same} symplectic flux-paths still need not end at Hamiltonian isotopic tori (we may e.g.~end at either the Clifford or Chekanov torus). In contrast to this, we here show that when the symplectic flux-path can be generated by a complete negative Liouville flow, then the Hamiltonian isotopy class is indeed determined by the flux-path.
 
\begin{lma}
\label{lma:LiouvilleFlow}
Consider a symplectic manifold $(X^{2n},\omega=d\lambda),$ with a \emph{complete} negative Liouville flow $\phi^{-t}$ induced by $\lambda,$ and a smooth Lagrangian submanifold $L \subset X.$ For any Lagrangian isotopy $L_t$ with $L_0=L$ that satisfies the property that the symplectic actions of $\phi^{-t}(L)$ and $L_t$ agree for all $t \in [0,T],$ it is the case that $\phi^{-T}(L)$ and $L_T$ are Hamiltonian isotopic inside $X^{2n}.$
\end{lma}
\begin{proof}
By Weinstein's Lagrangian neighbourhood theorem one readily proves the following: the set of numbers $S \in [0,T]$ satisfying the property that $\phi^{-t}(L)$ and $L_t$ are Hamiltonian isotopic for all $t \in [0,S]$ forms an interval which is open inside $[0,T]$ and contains zero. It suffices to show that this interval also is closed, since it then is equal to $[0,T].$

Thus, we take $S_0>0$ such that the property is true for all $S < S_0.$ By the same reason as to why the previously established openness property holds, there exists some $\epsilon>0$ such that $L_{S_0-t}$ is Hamiltonian isotopic to $\phi^t(L_{S_0})$ whenever $0\le t \le \epsilon.$

By our assumptions, $\phi^{-(S_0-\epsilon)}(L)$ is Hamiltonian isotopic to $L_{S_0-\epsilon}$ and, by the previous paragraph, the Lagrangian $L_{S_0-\epsilon}$ is Hamiltonian isotopic to $\phi^{\epsilon}(L_{S_0}).$ Since Hamiltonian isotopies are preserved by the conformal symplectomorphism $\phi^{-\epsilon},$ it finally follows that $\phi^{-{S_0}}(L)=\phi^{-\epsilon} \circ \phi^{-(S_0-\epsilon)}(L)$ is Hamiltonian isotopic to $\phi^{-\epsilon}\circ \phi^\epsilon(L_{S_0})=L_{S_0}$ as sought. (The assumption of completeness is needed in order to guarantee that $\phi^{-\epsilon}$ is well-defined on the whole Hamiltonian isotopy from $\phi^{-(S_0-\epsilon)}(L)$ to $\phi^\epsilon(L_{S_0})$.)
\end{proof}

\section{The Hamiltonian isotopy to a standard torus}
\label{sec:LagrangianIsotopy}

In this section we prove Theorem \ref{thm:main} in embedded case. In other words, given a Lagrangian torus $L \subset (V,\omega_{\OP{FS}})$ satisfying the assumptions made in that theorem, we produce a Hamiltonian isotopy to a standard torus. The first step of the proof consists of applying Theorem \ref{thm:CompatibleFibration}. From now on we hence assume that $L$ has been placed in either Clifford or Chekanov position by the Hamiltonian isotopy produced by the aforementioned result and thus, in particular, that we have $L \subset \CP^2 \setminus (\ell_\infty \cup C \cup C_{\OP{nodal}}).$ The proof of Theorem \ref{thm:main} for $L$ in Clifford position is given in Section \ref{sec:LagIsoClifford}, and in Section \ref{sec:LagIsoChekanov} when $L$ is in Chekanov position.

\subsection{Classification up to Hamiltonian isotopy for tori in Clifford position}
\label{sec:LagIsoClifford}

The argument presented in Sections \ref{sec:LagIsoChekanov} and \ref{sec:LagToHam} below, and which concerns with the case of a torus in Chekanov position, applies also to the present case of tori in Clifford position after minor modifications. However, we here choose a different path with a somewhat simpler argument. This argument also turns out to be useful later when we classify the immersed Lagrangian spheres.

After the application of Theorem \ref{thm:CompatibleFibration} the torus can be assumed to be contained inside $V \setminus f^{-1}[0,1].$ Since $f^{-1}[0,1]$ deformation retracts onto $f^{-1}(0),$ it follows from Corollary \ref{cor:conditions} that the torus is homologically essential in the same subset. The existence of the Hamiltonian isotopy is now a direct consequence of Theorem \ref{thm:nearby} combined with the following lemma:
\begin{lma}
\label{lma:ConvexSubsets}
There exists a symplectomorphism
$$\Phi \colon (V \setminus f^{-1}[0,1],\omega_{\OP{FS}}) \xrightarrow{\cong} \T^2 \times U\subset (T^*\T^2,d\lambda_{\T^2})$$
 
where
$$ U \coloneqq \{(p_1,p_2);\: |p_1|< \pi, \max{(0,p_1)} < p_2 < (p_1+\pi)/2\} \subset \R^2$$
is the bounded convex subset shown in Figure \ref{fig:action2}. 
This symplectomorphism can moreover be taken to satisfy
\begin{enumerate}
\item for any fixed $(u,v) \in (-1,1) \times (1,+\infty],$ the torus fibres $\Pi_s^{-1}(u,v)$ are all mapped standard tori $\Phi(\Pi_s^{-1}(u,v))=\T^2 \times \{\mathbf{p}_{(u,v)}\}$ whenever $s>0$ is sufficiently small, and
\item the punctured Lagrangian disc
$$ \{(z,-\overline{z})\in \C^2\} \cap f^{-1}(-\infty,0) \subset V \setminus f^{-1}[0,1]$$
(i.e.~the standard Lefschetz thimble with a point removed) is mapped into the Lagrangian annulus
$$(S^1 \times \{0\}) \times (\{0\} \times \R) \subset T^*\T^2.$$
\end{enumerate}
\end{lma}
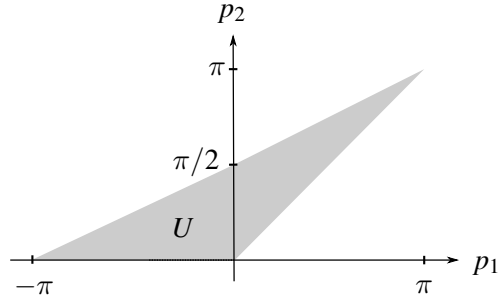
\begin{figure}[htp]
\begin{center}
\vspace{6mm}
\labellist
\pinlabel $p_2$ at 84 101
\pinlabel $\pi$ at 78 80
\pinlabel $\pi/2$ at 70 45
\pinlabel $U$ at 65 20
\pinlabel $p_1$ at 180 6
\pinlabel $-\pi$ at 9 -2
\pinlabel $\pi$ at 156 -2
\endlabellist
\includegraphics{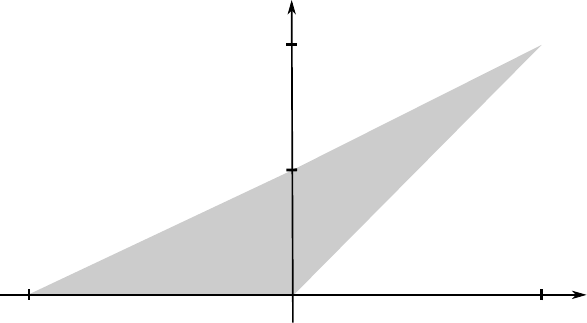}
\vspace{3mm}
\caption{ The image $U \subset \R^2$ of the values $(p_1,p_2)$ of the symplectic action class evaluated on the pair $(\mathbf{e}_0,\widetilde{\mathbf{e}}_1)$ of basis vectors of $H_1(\Pi_s^{-1}(u_1,u_2))$ with $u_2>1.$}
\label{fig:action2}
\end{center}
\end{figure}
\begin{proof}
The symplectomorphism is constructed by considering the Lagrangian torus fibration $\Pi_0$ on $V \setminus f^{-1}[0,1]$ from Section \ref{sec:LagrangianFibrationsIntro} obtained as the `limit' of the fibrations $\Pi_s$ as $s \to 0.$ Using the classical Arnol'd--Liouville theorem \cite[Theorem 2.3]{Symington:FourDimensions} we can find locally defined fibre preserving symplectomorphisms from this torus fibration into the standard fibration $\T^2 \times \{\mathbf{p}\}.$ Since the space of fibres is simply connected (even contractible), the maps can be patched together to give a well-defined global symplectomorphism. This symplectic map is an embedding, as follows from the action considerations in Lemma \ref{lma:action} (the Lagrangian tori in the fibration are determined uniquely by their symplectic actions).

To see that the imaged is $\T^2 \times U \subset T^*\T^2$ as claimed, we choose action--angle coordinates $(p_1,p_2)$ that correspond to the following global continuous choice of basis of $H_1(L)$ for the tori of Clifford type in the fibration. First choose the global choice of basis vector $\mathbf{e}_0$ as specified in Lemma \ref{lma:homologybasis}. Then choose the second basis vector $\widetilde{\mathbf{e}}_1$ by following the recipe in the case $u_1<0$ in Part (2.b) the same lemma. The shape of the image is now a consequence of the action computations made in Lemma \ref{lma:action}; also c.f. Figures \ref{fig:action} and \ref{fig:action2}. (Observe that, while the obtained basis $\langle \mathbf{e}_0,\widetilde{\mathbf{e}}_1\rangle$ of course satisfies $\widetilde{\mathbf{e}}_1=\mathbf{\mathbf{e}}_1$ in the region $u_1<0,$ it satisfies $\widetilde{\mathbf{\mathbf{e}}_1}=\mathbf{e}_1+\mathbf{e}_0$ in the region $u_1 \ge 0$.)

(1): This property is follows from the construction of the fibration $\Pi_0$ in Section \ref{sec:LagrangianFibrationsIntro}. 

(2): The Lagrangian punctured disc inside $V$ under consideration can alternatively be described as
$$\Pi_s^{-1}(\{0\} \times (1,+\infty)) \cap f^{-1}(-\infty,0),$$
for any $s \ge 0.$ The action properties of the tori in the family $\Pi_s^{-1}(\{0\} \times (1,+\infty))$ implies that, after an appropriate choice of coordinates on $T^*\T^2,$ the image of the punctured disc lives inside the subset
$$ (S^1 \times S^1) \times (\{0\} \times \R) \subset T^*\T^2,$$
while it intersects the Lagrangian tori $S^1 \times S^1 \times \{(0,t)\}$ cleanly in embedded closed curves. Hence, using the Lagrangian condition for the punctured disc, its image can be seen to be of the form $(S^1 \times \{\theta_0\}) \times (\{0\} \times \R)$ for some fixed $\theta_0 \in S^1.$
\end{proof}

\subsection{Classification up to Lagrangian isotopy for tori in Chekanov position}
\label{sec:LagIsoChekanov}
In the case of a torus in Chekanov position we must start by constructing a mere \emph{Lagrangian} isotopy inside the subset $\CP^2 \setminus (\ell_\infty \cup C \cup C_{\OP{nodal}})$ that deforms it to a standard torus.
\begin{prp}
\label{prp:LagrangianIsotopy}
Any Lagrangian torus $L \subset (\CP^2 \setminus (\ell_\infty \cup C \cup C_{\OP{nodal}}),\omega_{\OP{FS}})$ in Chekanov position is Lagrangian isotopic inside the same subset to a standard torus.
\end{prp}

\begin{proof}
 
Consider a sufficiently small neighbourhood $U \subset V = \CP^2 \setminus C$ of the subset $\{\|\widetilde{z}_1\|^2-\|\widetilde{z}_2\|^2=0\} \cap V.$ Using standard techniques, such as the symplectic neighbourhood theorem, one readily constructs a symplectic embedding $\psi \colon (U,\omega_{\OP{FS}}) \hookrightarrow (W,d\lambda)$ into the neighbourhood of the self-plumbing of the sphere as constructed in Section \ref{sec:SelfPlumbing}. Recall that there exists a symplectic involution $I$ of $(W,d\lambda)$ which interchanges tori of Chekanov and Clifford type; see Proposition \ref{prp:involution}. We may furthermore require that
\begin{itemize}
\item the nodal conic $C_{\OP{nodal}}$ is identified with the symplectic nodal surface $\{p_1=\pm q_2, p_2=\mp q_1\}$ which is fixed setwise by $I$ (while its two sheets are interchanged) and
\item the image $\psi(U) \subset W$ is invariant under the symplectic involution $I.$
\end{itemize}

Using Lemma \ref{lma:LiouvilleComputation} we can see that the backwards Liouville-flow of $\lambda_{(1/2-\eta,\eta,\eta,1/2-\eta)}$ for some small $\eta \in (0,1)$ is a Lagrangian isotopy of $L$ confined to $\CP^2 \setminus (\ell_\infty \cup C \cup C_{\OP{nodal}})$ that produces a Lagrangian isotopy from $L$ to a torus
$$ \phi_{\lambda_{(1/2-\eta,\eta,\eta,1/2-\eta)}}^{-T}(L) \subset U \setminus C_{\OP{nodal}}.$$
An application of Theorem \ref{thm:CompatibleFibration} to the torus $L' \coloneqq \psi^{-1}\circ I\circ \psi \circ \phi^{-T}(L) \subset V \setminus C_{\OP{nodal}}$ can readily be seen to make it compatible with a conic fibration while being in \emph{Clifford position}; see Proposition \ref{prp:involution} together with the homotopical consideration in Part (1) of Remark \ref{rmk:bases}. In Section \ref{sec:LagIsoClifford} we constructed a Hamiltonian isotopy from $L'$ to a standard torus of Clifford type. There is a Lagrangian isotopy contained inside $V \setminus C_{\OP{nodal}}$ from that torus to a standard torus of Clifford type that moreover contained inside $U \setminus C_{\OP{nodal}}.$ Denote by $L'_t$ this entire Lagrangian isotopy.

The sought Lagrangian isotopy is then obtained by an application of the Liouville flow of $\lambda_{(1/2-\eta,\eta,\eta,1/2-\eta)}$ to the Lagrangian isotopy $L'_t,$ with the purpose of making it confined to the subset $U\setminus C_{\OP{nodal}},$ followed by an application of the involution $\psi^{-1} \circ I \circ \psi$ which is defined in the same subset. Indeed, note that this involution interchanges standard tori of Clifford and Chekanov type (at least up to Hamiltonian isotopy) by Proposition \ref{prp:involution}.
\end{proof}

\subsection{From Lagrangian to Hamiltonian isotopy for tori in Chekanov position}
\label{sec:LagToHam}

Proposition \ref{prp:LagrangianIsotopy} gives a Lagrangian isotopy $L_t \subset (\CP^2 \setminus (\ell_\infty \cup C \cup C_{\OP{nodal}}),\omega_{\OP{FS}})$ from $L_0=L$ in Chekanov position to a standard torus. In order to correct the symplectic flux, it is sufficient to find a standard torus of Chekanov type of the same action as $L,$ as the following lemma shows.

\begin{lma}
\label{lma:hamiso}
 
Let $L_t \subset (\CP^2 \setminus (\ell_\infty \cup C \cup C_{\OP{nodal}}),\omega_{\OP{FS}})$ be a Lagrangian isotopy of tori in Chekanov position from $L_0=L$ to $L_1=\Pi_s^{-1}(u_1,u_2),$ where $s \in (0,\pi/2)$ and $(u_1,u_2) \in (-1,1) \times (0,1).$ Assume that there exists a Lagrangian isotopy
$$L_t^{\OP{std}} \coloneqq \Pi_{s(t)}^{-1}(u_1(t),u_2(t)) \subset \CP^2 \setminus (\ell_\infty \cup C \cup C_{\OP{nodal}}), \:\: u_2(t) \in (0,1),$$
of standard tori satisfying $s(1)=s$ and $(u_1(1),u_2(1))=(u_1,u_2),$ and such that the symplectic action of $L_t^{\OP{std}}$ and $L_t$ coincides for all $t$ (when using the identifications induced by the isotopies). Then $L$ is Hamiltonian isotopic inside $\CP^2 \setminus (\ell_\infty \cup C \cup C_{\OP{nodal}})$ to the standard torus $\Pi_{s(0)}^{-1}(u_1(0),u_2(0)).$ 
\end{lma}
\begin{proof}
 By our assumptions, after concatenating the Lagrangian isotopy $L_t$ with the Lagrangian isotopy $L^{\OP{std}}_{1-t},$ we may produce a Lagrangian isotopy $L_t$ of the same type, but that satisfies $L_0=L$ and $L_1=\Pi_{s(0)}^{-1}(u_1(0),u_2(0)).$ In particular, the symplectic actions of $L_0=L$ and the standard torus $L_1$ now coincide.

In the remainder of this proof we construct a deformation $L_t$ of the path relative endpoints $t=0,1,$ to a Lagrangian isotopy with vanishing symplectic flux-path by applying suitable negative Liouville flows; the main ingredient is Corollary \ref{cor:flux}.

Let $L^{\OP{std}}_t \coloneqq \Pi_{s(t)}^{-1}(u_1(t),u_2(t))$ be the path of standard tori being of the same action as (the new version of) the tori $L_t.$ Note that, since (the new version of) the Lagrangian isotopy $L_t$ starts and ends at a Lagrangian of the same symplectic action, $L^{\OP{std}}_t$ is actually a \emph{loop} of standard tori, i.e.~$L^{\OP{std}}_0=L^{\OP{std}}_1.$ Since $u_2(t)\in(0,1),$ this loop is moreover contractible; i.e.~there exists a homotopy $L_{t,\tau}^{\OP{std}} \coloneqq \Pi_{s(t,\tau)}^{-1}(u_1(t,\tau),u_2(t,\tau))$ of loops of standard tori satisfying
\begin{align*}
& \Pi_{s(t,0)}^{-1}(u_1(t,0),u_2(t,0))=\Pi_{s(t)}^{-1}(u_1(t),u_2(t)),\\
& \Pi_{s(t,1)}^{-1}(u_1(t,1),u_2(t,1)) \equiv \Pi_{s(0)}^{-1}(u_1(0),u_2(0)),\\
& \Pi_{s(i,\tau)}^{-1}(u_1(i,\tau),u_2(i,\tau)) \equiv \Pi_{s(0)}^{-1}(u_1(0),u_2(0)), \:\: i=0,1,
\end{align*}
and where $u_2(t,\tau) \in (0,1)$ moreover is satisfied. Using Corollary \ref{cor:flux} we then find smooth paths
\begin{align*}
& \mathbf{r}(t,\tau) \in \{ \R^4; \:\: r_i >0, \: r_0+\ldots+r_3=1\},\\
& \alpha(t,\tau) \ge 0, \:\: \alpha(0,\tau) \equiv \alpha(1,\tau) \equiv 0,
\end{align*}
for which $\phi^{-\alpha(t,\tau)}_{\mathbf{r}(t,\tau)}(\Pi_{s(t)}^{-1}(u_1(t),u_2(t))),$ and hence $\phi^{-\alpha(t,\tau)}_{\mathbf{r}(t,\tau)}(L_t),$ realises the same symplectic flux-path as $L^{\OP{std}}_{t,\tau}.$ In particular, $\phi^{-\alpha(t,1)}_{\mathbf{r}(t,1)}(L_t)$ is the sought Hamiltonian isotopy from the torus $\phi^{-\alpha(0,1)}_{\mathbf{r}(0,1)}(L_0)=L_0=L$ to the standard torus
$$\phi^{-\alpha(1,1)}_{\mathbf{r}(1,1)}(\Pi_{s(0)}^{-1}(u_1(0),u_2(0)))=\Pi_{s(0)}^{-1}(u_1(0),u_2(0)).$$
\end{proof}
The nondisplaceability result in Proposition \ref{prp:nondisp} can now be used to show that the assumptions of the previous lemma indeed are satisfied.
\begin{lma}
\label{lma:HamisoAssump}
 
Let $L_t \subset (\CP^2 \setminus (\ell_\infty \cup C \cup C_{\OP{nodal}}),\omega_{\OP{FS}})$ be a Lagrangian isotopy of tori in Chekanov position from $L_t=L$ to $L_1=\Pi_s^{-1}(u_1,u_2),$ where $s \in (0,\pi/2)$ and $(u_1,u_2) \in (-1,1) \times (0,1).$ Then the assumptions of Lemma \ref{lma:hamiso} can be assumed to hold. In particular $L$ is Hamiltonian isotopic to a standard torus.
\end{lma}
\begin{proof}
 
It suffices to show that there exists a standard torus of Chekanov type that is of the same symplectic action as $L_t$ for all $t \in [0,1].$ This is the case for the image of $L_t$ under the negative Liouville flow $\phi^{-T}_{1/2-\eta,\eta,\eta,1/2-\eta}$ for some small $\eta>0$ and sufficiently large $T \gg 0,$ as can be seen by alluding to the symplectic action computations in Proposition \ref{prp:action}.

Since $\phi^{-T}_{1/2-\eta,\eta,\eta,1/2-\eta}(\Pi_s^{-1}(u_1,u_2))$ is Hamiltonian isotopic to a standard torus by Lemma \ref{lma:LiouvilleFlow}, we can in fact conclude the following by alluding to Lemma \ref{lma:hamiso}: The image $\phi^{-T}_{1/2-\eta,\eta,\eta,1/2-\eta}(L)$ of our torus is Hamiltonian isotopic to a standard torus $\Pi_{s'}^{-1}(u_1',u_2')$ of Chekanov type.

From this we will now deduce that there exists a standard torus of the same symplectic action as $L_t.$ Consider the symplectic embedding
$$ \iota_{s'} \colon (V,\omega_{\OP{FS}}) \hookrightarrow (\widehat{W},d\lambda)$$
into its Liouville completion as provided by Proposition \ref{prp:embedding}, which takes torus fibres to torus fibres. Since $\iota_{s'}\circ \phi^{-T}_{1/2-\eta,\eta,\eta,1/2-\eta}(L)$ is Hamiltonian isotopic to a standard torus fibre $\hat{\pi} \colon \widehat{W} \to \R^2,$ we can use Lemma \ref{lma:LiouvilleFlow} to conclude that $\widetilde{L} \coloneqq \phi^T_\lambda \circ \iota_{s'}\circ \phi^{-T}_{1/2-\eta,\eta,\eta,1/2-\eta}(L)$ is Hamiltonian isotopic to a fibre $\hat{\pi}^{-1}(a,b)$ as well. Since $\phi^t_\lambda \circ \iota_{s'}\circ \phi^{-t}_{1/2-\eta,\eta,\eta,1/2-\eta}(L)$ is a Hamiltonian isotopy from $\iota_{s'}(L)$ to $\widetilde{L},$ we conclude that the latter fibre $\hat{\pi}^{-1}(a,b)$ must be contained in the image of $\iota_{s'}$ as sought. If not, then we have managed produced a Hamiltonian displacement of $\iota_{s'}(L),$ which is in contradiction with Proposition \ref{prp:nondisp}.
\end{proof}

\section{The Hamiltonian isotopy to a standard immersed sphere}

\label{sec:proof-whitney}
In this section we let $\widetilde{L} \subset (B^4,\omega_0)$ be an immersed Lagrangian sphere that satisfies one of the conditions in the assumption of Theorem \ref{thm:main}. The goal is to prove the theorem by finding a Hamiltonian isotopy to a standard sphere. We begin with a preliminary lemma which holds for arbitrary such spheres.
\begin{lma}
For any immersed Lagrangian sphere $\widetilde{L} \subset (B^4,\omega_0=d\lambda_{\OP{std}})$ with a single transverse double point, a primitive $f \colon \widetilde{L} \to \R$ of $\lambda_{\OP{std}}|_{T\widetilde{L}}$ satisfies the property that the difference of its values at the two preimages $\{p,q\}$ of the double point is bounded by $|f(p)-f(q)| \le \pi/2.$
\end{lma}
\begin{proof}
This is a consequence of \cite[Corollary 1.3]{Cieliebak:PuncturedHolomorphic}, by which the minimal positive symplectic action of a Lagrangian torus satisfies the same upper bound $\pi/2.$ Indeed, performing a ``zero-area'' Lagrangian surgery on $\widetilde{L}$ (this can always be done in an arbitrarily small neighbourhood of the double point) we obtain a monotone Lagrangian torus that is contained inside $B^4$ and whose minimal value of the modulus of the nonzero symplectic action is equal to precisely $|f(p)-f(q)|.$ (For such a surgery, the symplectic area of a disc with a corner is \emph{the same} as the corresponding area of the disc with boundary on the resolved Lagrangian.)
\end{proof}
In view of the above lemma, combined with Part (1) Proposition \ref{prp:mainprp}, when taking $s=|f(p)-f(q)|$ we conclude that $\widetilde{L}$ and $L_{\OP{Wh}}(s)$ both are strongly exact for the Liouville form $\lambda_{3s/\pi}$ on $(V,\omega_{\OP{FS}})$; c.f.~Lemma \ref{lma:LiouvilleCP2}.

Consider the two non-monotone tori $T_\pm^0 \coloneqq \Pi_s^{-1}(\pm \epsilon,1)$ fibred over the closed curve $f(T_\pm^0) =\Psi_s^{-1}(1+e^{i\theta}) \subset \C$ in the standard conic Lefschetz fibration. The subset of the fibres bounded by these two Lagrangian tori is a one-parameter family
$$ A^0(\theta) \coloneqq f^{-1}(\Psi_s^{-1}(1+e^{i\theta})) \cap \{ |\|\widetilde{z}_1\|^2-\|\widetilde{z}_2\|^2| \le \epsilon \}, \:\: \theta \in S^1,$$
of holomorphic annuli, whose boundaries thus provide foliations of the two tori. All annuli are embedded and smooth except for the `nodal' annulus $A^0(\pi),$ which consists of two embedded discs intersecting transversely in a single point. In addition, the annuli $\{A^0(\theta)\}_{\{\theta \neq \pi\}}$ intersect the smooth part (i.e.~the complement of the double point) of the Whitney immersion $L_{\OP{Wh}}(s)$ in a foliation by closed curves, while the double point of $L_{\OP{Wh}}(s)$ intersects $A^0(\pi)$ precisely in its node.

A basis $\langle \mathbf{e}^\pm_0,\mathbf{e}^\pm_1\rangle=H_1(T_\pm^0)$ with the properties in Lemma \ref{lma:goodbasis} can be fixed in the following manner. We let $\mathbf{e}_0^\pm$ be the boundary inside $T_\pm^0$ of either holomorphic disc in the nodal annulus $A^0(\pi)$; their orientations are moreover chosen so that the symplectic action of $\mathbf{e}_0^\pm$ satisfies $\pm\int_{\mathbf{e}_0^\pm} \lambda_{\OP{std}} > 0$ on $H_1(T_\pm^0).$ We then let $\mathbf{e}_1^\pm$ be the unique class of Maslov index two inside $\C^2$ for which $\int_{\mathbf{e}_1^\pm}\lambda_{\OP{std}}=s.$

Consider a Weinstein neighbourhood of $\widetilde{L}$ identified with a neighbourhood of the standard sphere $L_{\OP{Wh}}(s)=\Pi_s^{-1}(0,1)$ by a symplectomorphism $\Phi$ that satisfies $\Phi(L_{\OP{Wh}})=\widetilde{L}$ (here we have used Proposition \ref{prp:Weinstein}). The map $\Phi$ is constructed so that induces the identity map in homology; to that end recall that $H_i(V)\cong\Z$ for $i=0,1,$ and that $V$ deformation retracts onto the Whitney sphere by Propositions \ref{prp:W} and \ref{prp:embedding}.

Since $\widetilde{L}$ satisfies the assumption of Theorem \ref{thm:main}, the same holds also for the induced Lagrangian tori $T_\pm\coloneqq\Phi(T_\pm^0),$ which can be assumed to be arbitrarily close to $\widetilde{L}.$ Moreover, using this symplectomorphism we can produce a family $A(\theta)\coloneqq\Phi(A^0(\theta))$ of symplectic annuli containing our immersed Lagrangian sphere $\widetilde{L} \subset \bigcup_\theta A(\theta).$
Using the same identification $\Phi,$ we also obtain an induced basis $\langle \mathbf{e}_0^\pm,\mathbf{e}_1^\pm \rangle = H_1(T_\pm)$ from the choice of basis for $H_1(T_\pm^0)$ made above. We formulate its properties in a lemma.

\begin{lma}
\label{lma:orientation}
 The ordered basis $\langle \mathbf{e}_0^\pm, \mathbf{e}_1^\pm \rangle=H_1(T_\pm)$ satisfies the properties of Lemma \ref{lma:goodbasis}. Furthermore, the induced orientation on $[T_\pm] \in H_2(V)$ is equivalent to the orientation of a standard torus $[\Pi_s^{-1}(u_1,u_2)] \in H_2(V)$ that is induced by the action-angle coordinates and the standard orientation of the base $\{(u_1,u_2)\} \subset \R^2$ (more precisely, orienting the base gives an orientation of the cotangent bundle of the torus, which is canonically isomorphic to the tangent bundle by the Lagrangian condition).
\end{lma}

\begin{lma}
\label{lma:e0}
For a generic neck-stretching sequence around either of $T_\pm,$ the broken conic fibres produced by Proposition \ref{prp:characterisation} satisfy the property that the unique plane in the building which passes through $q_1$ (resp. $q_2$) is asymptotic to a closed geodesic in the class $\mathbf{e}_0^\pm$ (resp. $-\mathbf{e}_0^\pm$).
\end{lma}
\begin{proof}
 
Since the assumptions of Theorem \ref{thm:main} are satisfied for the tori $T_\pm,$ they are Hamiltonian isotopic to standard tori $\Pi_s^{-1}(u_1^\pm,u_2^\pm).$ By the above Lemma \ref{lma:orientation}, we must have $\int_{\mathbf{e}_0^\pm}(\lambda_{\OP{std}})=\pi\cdot u_1^\pm$ (the sign of the right-hand side is the nontrivial part of this equivalence, for which we must use the statement concerning the orientation induced by the basis). The claim now follows from the topological fact that the area of a plane that has intersection number $+1$ with $\ell_\infty$ is of symplectic area equal to $\pi-\int_{\mathbf{g}} \lambda_{\OP{std}}=\pi-\pi u_1^\pm,$ where $\mathbf{g} \in H_1(T^\pm)$ is the class of the geodesic to which the plane is asymptotic (note the sign!).
\end{proof}

We proceed to perform a neck-stretching \emph{simultaneously} around the two disjoint Lagrangian tori $T_+ \cup T_-.$ It can be explicitly seen that it is possible to choose the neck stretching sequence $J_\tau$ so that the above family $A(\theta)$ of symplectic annuli stay pseudoholomorphic for all $\tau \ge 0.$ To that end, it is important to use the fact that they all are nicely embedded near their boundaries, and that they intersect the tori $T_\pm$ cleanly in a foliation by embedded closed curves. Moreover, we can achieve the following for the limit almost complex structure $J_\infty$ defined on $\CP^2 \setminus (T_+ \cup T_-)$:
\begin{lma}
\label{lma:planes}
The almost complex structure $J_\infty$ on $\CP^2 \setminus (T_+ \cup T_-)$ may be taken to satisfy:
\begin{enumerate}
\item $C(\theta) \coloneqq A(\theta) \setminus (T_+ \cup T_-)$ for $\theta \neq \pi$ are embedded $J_\infty$-holomorphic cylinders (i.e.~two-punctured spheres) inside $\CP^2 \setminus (T_- \cup T_+),$ one of the punctures of which is asymptotic to a closed geodesic in the class $-\mathbf{e}_0^+$ on $T_+,$ while the other puncture is asymptotic to a closed geodesic in the class $\mathbf{e}_0^-$ on $T_-;$
\item The curve $C(\pi)$ is a nodal $J_\infty$-holomorphic cylinder consisting of two transversely intersecting pseudoholomorphic planes $P_\pm,$ where the plane $P_\pm$ is asymptotic to a closed geodesic on $T_\pm$ in the class $\mp \mathbf{e}_0^\pm;$
\item Away from the node $P_+ \cap P_-,$ the punctured spheres $C(\theta)$ provide a smooth foliation of an embedded three-manifold $\bigcup_\theta C(\theta) \setminus \{ P_+ \cap P_-\},$ and the asymptotic evaluation maps $\{C(\theta)\} \to \Gamma_{\pm \mathbf{e}_0} \cong S^1$ from the cylinders to its asymptotic orbit on either torus $T_\pm$ are diffeomorphism; and
\item The Fredholm index of each cylinder $C(\theta),$ $\theta \neq \pi,$ is equal to $\indx{C(\theta)}=0,$ while each of the two punctured planes involved in the nodal cylinder $C(\pi)$ is of index $-1.$
\end{enumerate}
\end{lma}
\begin{proof}
(1)--(3): These properties can be explicitly checked by hand.

(4): This follows easily from the formulation of the Fredholm index in terms of the Maslov index; see \cite[Section 3.1]{Dimitroglou:Isotopy}.
\end{proof}

The planes involved in $C(\pi)$ are of negative Fredholm index, and are hence clearly not transversely cut out. Neither the cylinders $C(\theta)$ have the correct dimension, since they come in a one-dimensional family while their index is equal to zero. In other words, an almost complex structure $J_\infty$ as above is never regular. For that reason, additional care must be taken in order to control the structure of the broken conics. First we need to exclude the existence of other pseudoholomorphic planes of negative index. By the following lemma, this is possible by perturbing $J_\infty$ away from $\widetilde{L}.$
\begin{lma}
\label{lma:almostholomorphic}
There exists an arbitrarily small closed neighbourhood $\Phi(U) \subset \CP^2$ of $\bigcup_\theta A(\theta) \supset T_\pm,$ where $U \subset \CP^2$ is a sufficiently small neighbourhood of $\bigcup_\theta A^0(\theta) \supset T_\pm^0,$ and an almost complex structure $J_U$ on $\Phi(U) \setminus (T_+ \cup T_-)$ for which
\begin{enumerate}
\item $J_U$ can be extended to an almost complex structure $J_\infty$ on $\CP^2 \setminus (T_+ \cup T_-)$ satisfying the properties of Lemma \ref{lma:planes};
\item $J_U$ has the property that, given an arbitrary such extension, any $J_\infty$-holomorphic plane inside $\CP^2 \setminus (T_+ \cup T_-)$ which is not equal to a branched cover of $P_\pm$ must leave the neighbourhood $\Phi(U).$
\end{enumerate}
In particular, after after choosing an extension $J_\infty$ to be generic outside of $\Phi(U),$ we may assume that the planes of negative index are precisely the planes $P_\pm$ together with their branched covers.
\end{lma}
\begin{proof}
There is a bijection between pseudoholomorphic planes $P \subset \Phi(U)\setminus (T_+ \cup T_-)$ contained inside the neighbourhood of interest, and pseudoholomorphic planes
$$\Phi^{-1}(P) \subset U \setminus (T_+^0 \cup T_-^0) \subset \CP^2 \setminus (T_+^0 \cup T_-^0),$$
given that the almost complex structures are chosen so that the symplectomorphism $\Phi$ is a biholomorphism. The complex structure satisfying the sought properties will be constructed on the `standard model' $U \setminus (T_+^0 \cup T_-^0),$ which then will be pushed forward to $\Phi(U)$ under the locally defined symplectomorphism $\Phi.$ In other words, we need to find a suitable neighbourhood $U$ of $\bigcup_\theta A^0(\theta),$ together with a suitable almost complex structure $J_\infty$ on $\CP^2 \setminus (T_-^0 \cup T_+^0),$ for which the branched covers of the two planes
$$P_\pm^0 \subset C^0(\pi) \subset \CP^2 \setminus (T_+^0 \cup T_-^0)$$
asymptotic to $T_\pm^0$ comprise all the $J_\infty$-holomorphic planes contained inside $U.$

We begin by constructing an almost complex structure $J_\infty$ on $(\CP^2\setminus (T_+^0 \cup T_-^0),\omega_{\OP{FS}})$ by deforming the standard complex structure $i$ near $T_+^0 \cup T_-^0$ in order to obtain a concave cylindrical end there. With some additional care, this construction can be performed so that
\begin{itemize}
\item $J_\infty$ is made to satisfy the analogous properties of Lemma \ref{lma:planes} (for $T_\pm^0$ instead of $T_\pm$), and
\item all $i$-holomorphic conic fibres $f^{-1}(z)$ that pass through $T_\pm^0$ give rise to three $J_\infty$-holomorphic punctured spheres $f^{-1}(z) \setminus (T_+^0 \cup T_-^0).$ The latter three components consist of two planes $A^\infty_i(\theta),$ $i=1,2,$ being tangent to $C$ at $A^\infty_i(\theta) \cap \ell_\infty=\{q_i\},$ together with a cylinder $C^0(\theta).$
\end{itemize}
Using the second property, one obtains a broken pseudoholomorphic conic by adjoining two cylinders in the bottom level to these three punctured spheres in the top level, as shown in Figures \ref{fig:whitneybreaking1} and \ref{fig:whitneybreaking2}.

We continue by arguing that the simply covered $J_\infty$-holomorphic planes inside $V \setminus (T_+^0 \cup T_-^0)$ are precisely the two planes $P_\pm^0.$ Argue by contradiction, assuming the existence of a $J_\infty$-holomorphic plane $P \subset U \setminus T_+^0$ asymptotic to $T_+^0$; the argument is analogous in the case of $T_-^0.$ The asymptotics of $P$ is a closed geodesic in the class $-k\mathbf{e}_0^+$ on $T_+^0$ for some $k>0.$ (Recall that $k < 0$ is impossible by positivity of symplectic area for pseudoholomorphic curves.) 

A topological consideration shows that
$$P \bullet (P^0_- \cup A^\infty_2(\pi))=k[P^0_+] \bullet (P^0_- \cup A^\infty_2(\pi))=k.$$
For the first equality, we use the fact that $T_+$ is disjoint from the line $P^0_- \cup A^\infty_2(\pi) \subset \CP^2$ and that the connecting homomorphism $H_2(\C^2,T_+) \xrightarrow{\cong} H_1(T_+)$ is an isomorphism. Positivity of intersection then shows that, unless $P = kP_+^0,$ the plane $P$ also must intersect at least one of the components $C^0(\theta), A^\infty_2(\theta)$ for all $\theta$ sufficiently close to $\pi.$ (This follows by positivity of intersection of $P$ and $C_\pi^0 \cup A^\infty_2(\pi)$.)

Now choose a $\theta$ as above, with the additional requirement that all asymptotics arising in the broken conic with top components $A^\infty_1(\theta) \cup C^0(\theta) \cup A^\infty_2(\theta) \subset \CP^2 \setminus T_-^0$ are different from the asymptotic of $P.$ In other words, $P \cup k \cdot A^\infty_1(\theta)$ and $A^\infty_1(\theta) \cup C^0(\theta) \cup A^\infty_2(\theta)=f^{-1}(z)$ inside $\CP^2 \setminus T_-^0$ compactify to two cycles in $\CP^2$ that intersect in a subset where they both are $J_\infty$-holomorphic. Positivity of intersection leads to the final contradiction, since one computes the intersection number
$$ (P \cup k\cdot A^\infty_1(\theta_0)) \bullet f^{-1}(z) > k[\ell_\infty] \bullet [f^{-1}(z)]=2k$$
between the cycles given as the completions of the respective buildings. Recall that the intersection at $q_1 \in \ell_\infty$ of the two buildings is a tangency that precisely gives the contribution $2k.$
\end{proof}

\begin{lma}
\label{lma:brokenlines}
Consider a neck-stretching sequence around $T_+$ (resp. $T_-)$ for which $P_\pm \subset \CP^2 \setminus T_\pm$ is the unique simply covered $J_\infty$-holomorphic plane of index $-1.$ Then the limit of lines passing through $q_1 \in \ell_\infty$ (resp. $q_2 \in \ell_\infty$) while being tangent to $C$ is a broken line containing a branched cover of $P_+$ (resp. $P_-$) in its limit.
\end{lma}
\begin{proof}
We prove the statement in the case of the line through $q_1,$ and while stretching the neck around the Lagrangian torus $T_+;$ the proof is analogous in the other case.

If the limit is a broken line, then the statement is a direct consequence of Corollary \ref{cor:conditions} together with the fact that there are no contractible geodesics on the flat $T_+;$ by the assumptions there simply are no other pseudoholomorphic planes inside $V \setminus T_+$ to which the broken line can limit except, the branched covers of $P_+.$ Note that, by elementary topological reasons, it is clear that the building indeed must contain a plane.

It thus suffices to show that the line necessarily converges to a \emph{broken} line. We argue by contradiction and assume that the limit line passing through $q_1,$ denoted by $\ell_1^{J_\infty},$ is unbroken. 

We can perturb the almost complex $J_\infty(t),$ $t \in [0,\epsilon],$ through a path of almost complex structures in order to obtain the regular almost complex structure $J_\infty(\epsilon),$ where $J_\infty(0)=J_\infty.$ The lines passing through $q_1$ being tangent to $C$ may be assumed to remain unbroken during this isotopy, giving rise to a smooth family $\ell_1^{J_\infty(t)}$ of such lines.

Then we examine the limit of the conic fibration when stretching the neck around $T_+$ in the case of the generic almost complex structure $J_\infty(\epsilon).$ By Proposition \ref{prp:characterisation} together with Lemma \ref{lma:e0}, we see that the cycle $P_+$ (which need not be pseudoholomorphic for the perturbed almost complex structure) necessarily intersects $\ell_1^{J_\infty(1)},$ and hence also $\ell_1^{J_\infty},$ with \emph{negative} intersection number. Indeed, the plane $A_1 \subset \CP^2 \setminus T_+$ in the broken conic with asymptotic Reeb orbit covering the same geodesic as the corresponding asymptotic of $P_+$ (but necessarily with the opposite orientation of the geodesic) can be completed to a cycle $A_1 \cup P_+$ in the class of a line; then we use the fact that $A_1 \bullet \ell_1^{J_\infty}=2.$ (There is a similar argument in the proof of Proposition \ref{prp:characterisation}.)

Finally, since $P_+$ and $\ell_1^{J_\infty}$ both are $J_\infty$-holomorphic, the negative intersection number $P_+ \bullet \ell_1^{J_\infty}<0$ gives the sought contradiction.
\end{proof}

\begin{figure}[htp]
\centering
\vspace{3mm}
\hspace{20mm}
\labellist
\pinlabel $\CP^2\setminus L$ at -26 58
\pinlabel $T^*L$ at -18 19
\pinlabel $q_2$ at 153 81
\pinlabel $q_1$ at 33 81
\pinlabel $0$ at 73 53
\pinlabel $\color{blue}\widetilde{L}$ at 93 83
\pinlabel $C(\theta)$ at 132 70
\pinlabel $A_2(\theta)$ at 180 65
\pinlabel $A_1(\theta)$ at 55 73
\pinlabel $5$ at 33 53
\pinlabel $5$ at 154 53
\pinlabel $2$ at 42 11
\pinlabel $2$ at 148 11
\pinlabel $\color{red}\mathbf{e}_0$ at 35 23
\pinlabel $\color{red}-\mathbf{e}_0$ at 72 23
\pinlabel $\color{red}\mathbf{e}_0$ at 116 23
\pinlabel $\color{red}-\mathbf{e}_0$ at 153 23
\pinlabel $\color{blue}T_+$ at 57 -10
\pinlabel $\color{blue}T_-$ at 138 -10
\vspace{5mm}
\endlabellist
\includegraphics{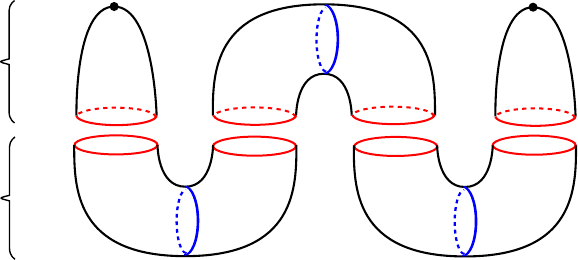}
\hspace{15mm}
\caption{The broken conic in the generic case. The two planes $A_i(\theta)$ each intersect $\ell_\infty$ transversely in the point $q_i,$ where they moreover are tangent to the smooth conic $C.$ The two planes join to form a continuous embedding of a sphere.}
\label{fig:whitneybreaking1}
\end{figure}
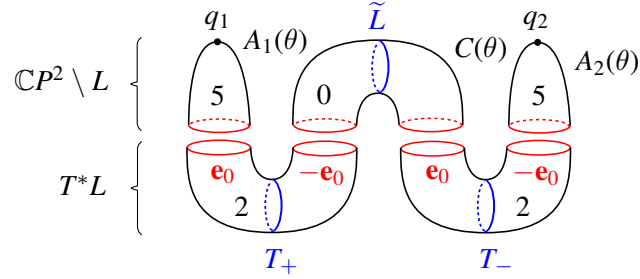

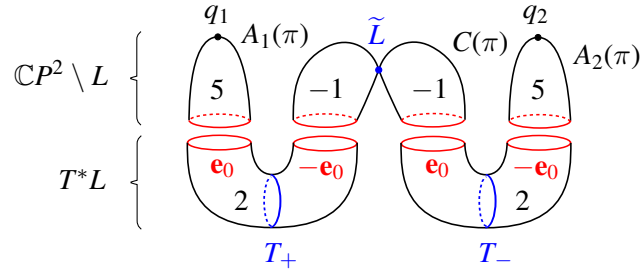
\begin{figure}[htp]
\centering
\vspace{3mm}
\hspace{20mm}
\labellist
\pinlabel $\CP^2\setminus L$ at -26 58
\pinlabel $T^*L$ at -18 19
\pinlabel $q_2$ at 153 81
\pinlabel $q_1$ at 33 81
\pinlabel $-1$ at 74 53
\pinlabel $-1$ at 113 53
\pinlabel $\color{blue}\widetilde{L}$ at 93 75
\pinlabel $C(\pi)$ at 132 70
\pinlabel $A_2(\pi)$ at 180 65
\pinlabel $A_1(\pi)$ at 55 73
\pinlabel $5$ at 33 53
\pinlabel $5$ at 154 53
\pinlabel $2$ at 42 11
\pinlabel $2$ at 148 11
\pinlabel $\color{red}\mathbf{e}_0$ at 35 23
\pinlabel $\color{red}-\mathbf{e}_0$ at 72 23
\pinlabel $\color{red}\mathbf{e}_0$ at 116 23
\pinlabel $\color{red}-\mathbf{e}_0$ at 153 23
\pinlabel $\color{blue}T_+$ at 57 -10
\pinlabel $\color{blue}T_-$ at 138 -10
\vspace{5mm}
\endlabellist
\includegraphics{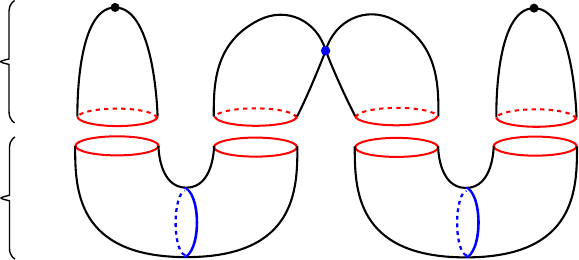}
\hspace{15mm}
\caption{The broken conic in the generic case. The two planes $A_i(\pi)$ each intersect $\ell_\infty$ transversely in the point $q_i,$ where they moreover are tangent to the smooth conic $C.$ The two planes join to form a continuous embedding of a sphere.}
\label{fig:whitneybreaking2}
\end{figure}

\begin{prp}
\label{prp:characterisation2}
For a generic almost complex structure satisfying the conclusions of Lemma \ref{lma:almostholomorphic}, there exists a family of embedded broken conics parametrised by $\theta \in S^1,$ all which are tangent to $C$ at the two points $q_1,q_2 \in \ell_\infty,$ and consisting of the following three components in the top level $\CP^2 \setminus (T_+ \cup T_-)$ varying smoothly with $\theta$:
\begin{itemize}
\item two planes $A_i(\theta) \subset \CP^2 \setminus (T_+ \cup T_-),$ $i=1,2,$ where $A_1(\theta)$ is asymptotic to $T_+$ and tangent to $C$ at $q_1 \in \ell_\infty,$ and where $A_2(\theta)$ is tangent to $C$ at $q_2 \in \ell_\infty$ and asymptotic to $T_-,$ together with
\item the cylinder $C(\theta)=\Phi(C^0(\theta))$ constructed above,
\end{itemize}
while the bottom level consists of two cylinders inside $T^*T_\pm.$ Moreover, all curves $A_i(\theta) \setminus \ell_\infty$ and $C(\theta) \setminus \ell_\infty$ are mutually disjoint and foliate the three-dimensional variety that is given by their union. See Figures \ref{fig:whitneybreaking1} and \ref{fig:whitneybreaking2} for a schematic depiction of this.
\end{prp}
\begin{proof}
Using Lemma \ref{lma:brokenlines} we conclude that the nodal conics converge to a broken conic consisting of (possibly trivial) branched covers of both planes $P_\pm$ under a neck-stretching limit. Arguing by using positivity of intersection, we conclude that these planes occur only once in each such broken line, and that they both are simply covered. The lines are thus seen to have components in the top level consisting of $P_\pm$ together with planes $A_i \subset \CP^2 \setminus (T_+ \cup T_-)$ passing through $q_i,$ for $i=1,2.$

We then argue that the planes $A_i$ necessarily live in \emph{compact} components of its moduli spaces consisting of planes tangent to $C$ at $q_i \in \ell_\infty.$ Namely, any broken building arising as an SFT limit of such planes would involve an additional plane $P_\pm.$ Since every such a plane intersects $P_\mp$ transversely, this would contradict positivity of intersection when combined with the fact that the original planes $A_i$ were disjoint from $P_+ \cup P_-$; recall that the SFT compactness theorem implies that we can extract a convergent sequence which converges uniformly in the $C^\infty$-topology on compact subsets. Finally we can argue as in the proof of Theorem \ref{thm:CompatibleFibration} (i.e.~using automatic transversality and the asymptotic intersection number) in order to show the existence of the required $S^1$-families $A_i(\theta)$ of planes.
\end{proof}

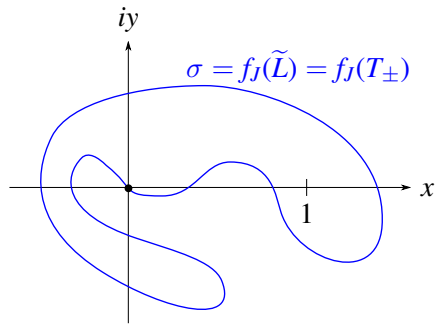
\begin{figure}[htp]
\begin{center}
\vspace{6mm}
\labellist
\pinlabel $1$ at 140 52
\pinlabel $\color{blue}\sigma=f_J(\widetilde{L})=f_J(T_\pm)$ at 136 121
\pinlabel $x$ at 197 63
\pinlabel $iy$ at 56 142
\endlabellist
\includegraphics[scale=0.8]{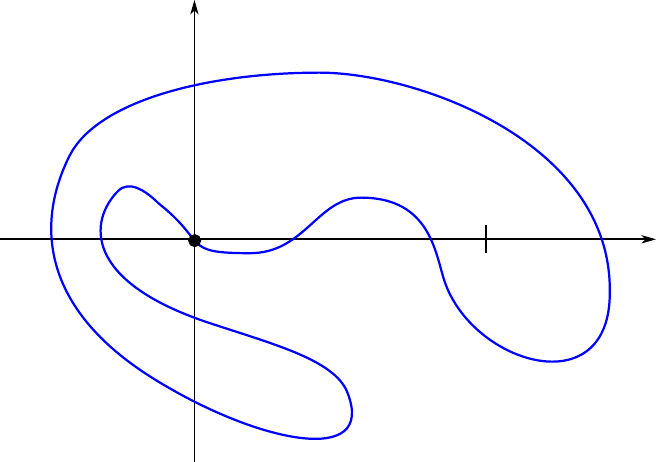}
\caption{The image of $\widetilde{L}$ and $T_\pm$ under $f_J$ coincide by construction.}
\label{fig:fibration-whitney2a}
\end{center}
\end{figure}

\subsection{The proof of Theorem \ref{thm:main} in the case of a sphere}
Proposition \ref{prp:characterisation2} above produces a family of broken conics asymptotic to $T_+ \cup T_-$ that contains the immersed sphere $\widetilde{L}.$ We then use the techniques from \cite[Section 5.3]{Dimitroglou:Isotopy} in order to smoothen the broken conics near $T_+ \cup T_-.$ Recall that this smoothing is performed so that the produced conics in the $S^1$-family
$$\overline{A_1(\theta) \cup C(\theta) \cup A_2(\theta)} \subset (\CP^2,\omega_{\OP{FS}}), \:\: \theta \in S^1,$$
each intersect the union $T_+ \cup T_-$ of tori in two embedded and disjoint closed curves that are homotopically nontrivial. In addition, we may assume that the symplectic annuli $C(\theta)$ are left undeformed by this smoothing procedure, since they already comapactify to a smooth surface with boundary on $T_\pm$ by their construction.

An application of Theorem \ref{thm:Lefschetz} combined with Theorem \ref{thm:normalise} now produces a globally defined smooth conic fibration $f_J \colon \CP^2 \setminus \ell_\infty \to \C$ satisfying the following crucial properties:
\begin{itemize}
\item $f_J^{-1}(0)$ is the nodal conic,
\item $f_J^{-1}(1)=C$ is our standard smooth conic, and
\item $f_J=f$ near $\ell_\infty.$
\end{itemize}
By construction, the immersed sphere $\widetilde{L} \subset \bigcup_\theta C(\theta)$ satisfies the property that $f_J(\widetilde{L})=\sigma \subset \C$ is a closed curve encircling $f_J(C)=1 \in \C$ with winding number one and passing through $0 \in \C$; see Figure \ref{fig:fibration-whitney2a} for an example. Note that both tori $T_\pm$ also live over the same curve $\sigma$ by construction. In the four steps below we then perform additional normalisations of the fibration $f_J.$

{\em Step I: Normalising the nodal conic.}
By construction, the nodal conic $f_J^{-1}(0)$ can be assumed to have a node modelled on the standard conic $C_{\OP{nodal}}$; recall that this nodal conic coincides with the nodal annulus $C(\pi)$ which, in turn, is symplectomorphic to the standard nodal annulus $C^0(\pi) \subset C_{\OP{nodal}}$ by construction. After an application of a suitable Hamiltonian isotopy we may also assume that $\ell_i^J \subset f_J^{-1}(0)=\ell_1^J \cup \ell_2^J$ coincides with $\ell_i$ near the node for $i=1,2.$

Arguing as in the proof of Lemma \ref{lma:DisjoinNodal}, we can readily produce a Hamiltonian isotopy that deforms $f_J^{-1}(0)$ to $C_{\OP{nodal}}$: First we consider a family of symplectic nodal conics that connects $f_J^{-1}(0)$ to $C_{\OP{nodal}}$ produced by Gromov's result Theorem \ref{thm:gromov}, for an interpolation of tame almost complex structures from $J$ to $i.$ Then we invoke Part (2) of Proposition \ref{prp:NormaliseNode} to normalise the node as well as the intersections with $\ell_\infty.$ Finally, we can apply Proposition \ref{prp:SiebertTian}. 

{\em Step II: Normalising a neighbourhood of the nodal conic.}

It is now the case that the double point of $\widetilde{L}$ intersects $C_{\OP{nodal}}$ precisely in its node. After a Hamiltonian isotopy supported in an arbitrarily small neighbourhood $f_J^{-1}B^2_{2\epsilon}$ of this node, one readily makes sure that $\widetilde{L}$ moreover coincides with $L_{\OP{Wh}}(s)$ inside $f^{-1}B^2_\epsilon$ for any fixed choice of $s \in (0,\pi/2).$ Reiterating the previous construction in this setting, we can now assume that the fibration $f_J$ produced satisfies the additional properties
\begin{itemize}
\item $f_J=f,$ and
\item $\widetilde{L}=L_{\OP{Wh}}(s)$
\end{itemize}
in some neighbourhood of $C_{\OP{nodal}}.$ The image of the immersed sphere after this modification is shown in Figure \ref{fig:fibration-whitney2b}. 

For a suitable such modification, we may in addition assume that there exists an embedded path $\gamma \subset \C$ from $0$ to $1,$ also shown in Figure \ref{fig:fibration-whitney2b}, satisfying the properties that
\begin{enumerate}
\item $\gamma \cap \sigma = \{0\},$
\item $\gamma$ coincides with $[0,1]$ near its boundary point $\{0,1\} \subset \C,$ and
\item $\gamma$ is isotopic to $[0,1]$ through paths $\gamma_t$ coinciding with $[0,1]$ near its boundary, where $\gamma_0=\gamma$ and $\gamma_1=[0,1].$
\end{enumerate}
For the last point, it might be necessary to first deform the sphere $\widetilde{L}$ near its double point by an explicit rotation via an application of the Reeb flow on the standard contact sphere. We proceed to give the details.

Consider a Hamiltonian of the form $h(\|\widetilde{\mathbf{z}}\|^2)$ defined on $(B^4 \setminus \{0\},\omega_0) \cong (\CP^2 \setminus \ell_\infty,\omega_{\OP{FS}})$ with support in a small neighbourhood of the origin. We require that $h(\|\widetilde{\mathbf{z}}\|^2)=k\pi\|\widetilde{\mathbf{z}}\|^2/4$ is satisfied near the origin for some suitable $k \in \Z.$ Note that the corresponding Hamiltonian diffeomorphisms given as the time-$1$ maps are of the form $\widetilde{\mathbf{z}} \mapsto e^{i\theta(\|\widetilde{\mathbf{z}}\|^2)}\widetilde{\mathbf{z}},$ where $\theta(\|\widetilde{\mathbf{z}}\|^2) \equiv k\pi/2$ near the origin. In particular, this Hamiltonian diffeomorphism fixes the two coordinate lines $\{\widetilde{z}_i=0\},$ $i=1,2.$

Inspecting the standard fibration $\widetilde{f}(\widetilde{z},\widetilde{z}_2)=\widetilde{z}_1\widetilde{z}_2/(1-\|\widetilde{z}_1\|^2-\|\widetilde{z}_2\|^2)$ in these coordinates, we see that the deformation of $\widetilde{L}$ by such a Hamiltonian diffeomorphism still has an image under $f_J$ which is a curve that passes through the origin, but which now spins around it with a number $k$ of additional half-turns. Here we note that the double point of the sphere is $0 \in B^4,$ and that its two sheets are tangent to the two Lagrangian discs $\{(z,\pm\overline{z})\} \subset \C^2$ there.

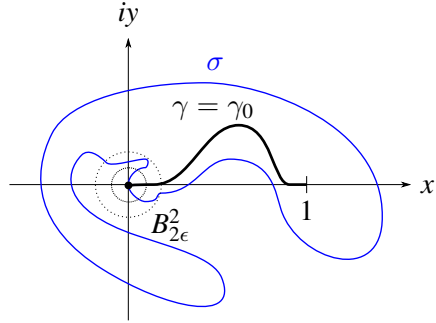
\begin{figure}[htp]
\begin{center}
\vspace{6mm}
\labellist
\pinlabel $1$ at 140 52
\pinlabel $\color{blue}\sigma$ at 97 121
\pinlabel $\gamma=\gamma_0$ at 96 99
\pinlabel $B_{2\epsilon}^2$ at 76 45
\pinlabel $x$ at 197 63
\pinlabel $iy$ at 56 143
\endlabellist
\includegraphics[scale=0.8]{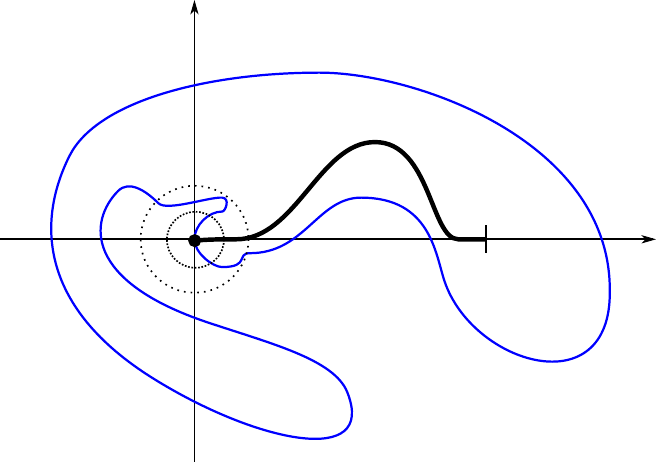}
\caption{Here we have deformed the sphere $\widetilde{L}$ inside $f^{-1}(B^2_{2\epsilon})$ in order to make it standard inside $f^{-1}(B^2_\epsilon).$}
\label{fig:fibration-whitney2b}
\end{center}
\end{figure}

{\em Step III: Normalise along a path from $0$ to $1$.}

As in the proof of Theorem \ref{thm:CompatibleFibration}, the isotopy from $\gamma$ to $[0,1]$ induces a Hamiltonian isotopy that disjoins $\widetilde{L}$ from $f^{-1}(0,1].$ A second reiteration of the whole argument in this subsection now allows us to infer that, in addition to the above, it is the case that $f_J=f$ is satisfied in some neighbourhood of $f^{-1}([-\epsilon,1]),$ and that $\widetilde{L}$ moreover coincides with a standard sphere in the same neighbourhood. See Figure \ref{fig:fibration-whitney2c} for a schematic depiction of this situation.

\begin{figure}[htp]
\begin{center}
\vspace{6mm}
\labellist
\pinlabel $1$ at 140 52
\pinlabel $\color{blue}\sigma$ at 97 121
\pinlabel $\gamma_1$ at 96 74
\pinlabel $x$ at 197 63
\pinlabel $iy$ at 56 143
\endlabellist
\includegraphics[scale=0.8]{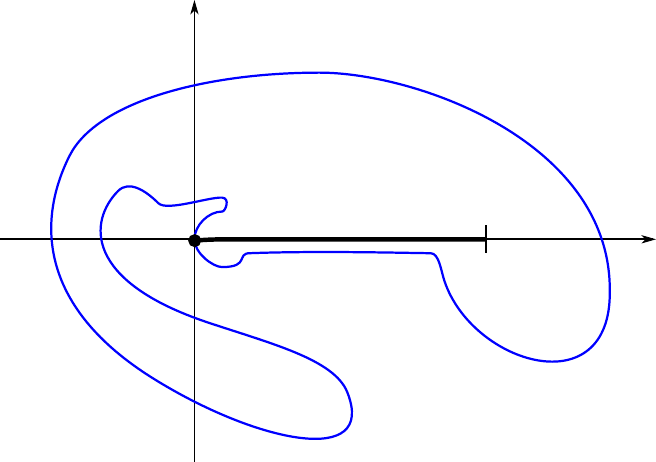}
\caption{Here we have deformed the sphere $\widetilde{L}$ to make it standard inside a neighbourhood of $f^{-1}[0,1].$}
\label{fig:fibration-whitney2c}
\end{center}
\end{figure}
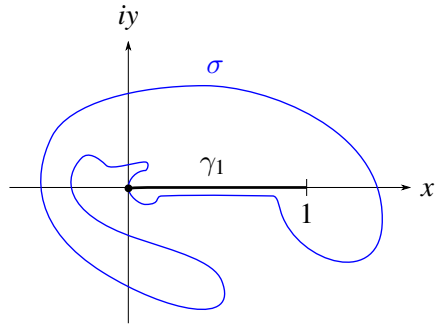

{\em Step IV: Resolving the double point.}

Finally, we perform a Lagrangian surgery in a neighbourhood of $C_{\OP{nodal}},$ replacing $\widetilde{L}$ there (where it coincides with a standard immersed sphere) with a piece of a standard torus of Clifford type; this is shown in Figure \ref{fig:fibration-whitney2d}. Denote by $L$ the obtained Lagrangian torus, which is contained in an arbitrarily small neighbourhood of $\widetilde{L}.$ Note that the standard Lagrangian disc $\{(z,-\overline{z});\:\|z\|\le\sqrt{\epsilon}\} \subset V,$ may be assumed to intersect $L$ cleanly precisely along its boundary.

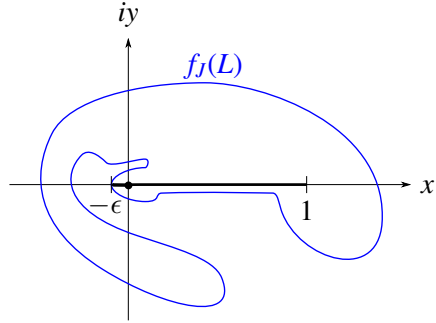
\begin{figure}[htp]
\begin{center}
\vspace{6mm}
\labellist
\pinlabel $1$ at 140 52
\pinlabel $-\epsilon$ at 45 54
\pinlabel $\color{blue}f_J(L)$ at 97 120
\pinlabel $x$ at 197 63
\pinlabel $iy$ at 56 143
\endlabellist
\includegraphics[scale=0.8]{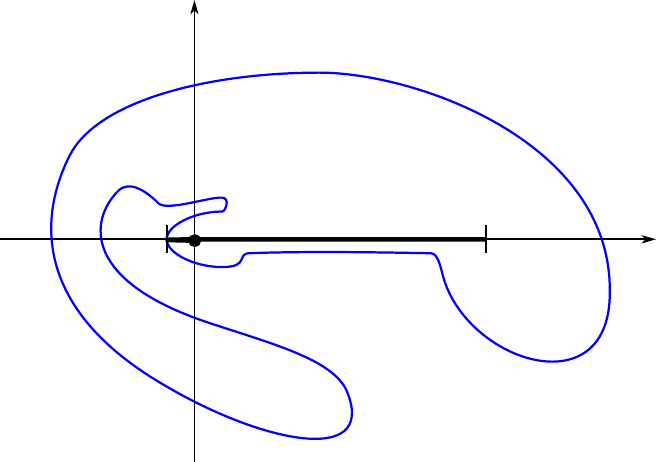}
\caption{The resolution of the double point of $\widetilde{L}$ creates a Lagrangian torus $L$ of Clifford type. The image of the Lagrangian disc $\{(z,-\overline{z}); \: \|z\|\le\sqrt{\epsilon}\},$ under $f_J$ is equal to the path $[-\epsilon,0] \subset \C.$}
\label{fig:fibration-whitney2d}
\end{center}
\end{figure}

The proof of Theorem \ref{thm:main} is finalised by a Hamiltonian isotopy from $L$ to a standard torus, where this Hamiltonian isotopy is supported away from the Lagrangian disc $\{(z,-\overline{z}),\:\:\|z\|\le \sqrt{\epsilon}\} \subset V.$ To construct this Hamiltonian isotopy we apply Parts (2) and (3) of Theorem \ref{thm:nearby}, after first using the symplectic identification of $(\CP^2 \setminus (\ell_\infty \cup f^{-1}[0,1]),\omega_{\OP{FS}})$ and a suitable subset $\R \times U \subset T^*\T^2$ as produced by Lemma \ref{lma:ConvexSubsets}.

The final Hamiltonian isotopy from $\widetilde{L}$ to a standard immersed sphere is then constructed by means of a simple modification by hand of the above Hamiltonian isotopy from $L$ to a standard Clifford torus.\qed

\section{The proof of Theorem \ref{thm:nearby}}
\label{sec:proofnearby}
In this section we prove the refined version of the nearby Lagrangian conjecture for $\T^2$ formulated in Theorem \ref{thm:nearby}. Denote by $\iota_L \colon L \hookrightarrow T^*\T^2$ the Lagrangian embedding under consideration, and let $\pi \colon T^*\T^2 \to \T^2$ be the canonical bundle projection.

We begin by deducing the existence of a Hamiltonian isotopy to a graphical Lagrangian under either of the different assumptions made on the embedding, i.e.~that $\iota_L$ either is
\begin{enumerate}
\item weakly exact, 
\item homological essential, or 
\item a Lagrangian embedding of a torus with vanishing Maslov class.
\end{enumerate}
In view of \cite[Theorem 7.1]{Dimitroglou:Isotopy}, such a Hamiltonian isotopy exists whenever $\iota_L$ is a homotopy equivalence; we thus need the following result. 
\begin{lma}
The three assumptions (1), (2), and (3) on $\iota_L$ are equivalent, and either condition is moreover equivalent to $\iota_L$ being a homotopy equivalence.
\end{lma}
\begin{proof}
First note that any closed Lagrangian $L \subset T^*\T^2$ which is homologically essential with $\Z$-coefficients is orientable, and is hence a torus by the Lagrangian adjunction formula. By \cite[Theorem 4.1.1]{Audin:SymplecticRigidity} a Lagrangian torus $L \subset T^*\T^2$ has vanishing Maslov class if and only if it is homologically essential. As shown by Arnol'd in \cite{FirstSteps}, and independently by E. Giroux in \cite{ContactTendue}, the latter is in turn equivalent to $\iota_L$ being a homotopy equivalence. In conclusion, we have shown that (2) and (3) both are equivalent to $\iota_L$ being a homotopy equivalence. It is clear from elementary topology that these conditions also imply weak exactness.

What remains is showing that (1) implies that $\iota_L$ is a homotopy equivalence. If $L$ weakly exact, Lemma \ref{lma:exact} below shows that $L$ is \emph{exact} after a translation by a section corresponding to the graph of a suitable closed one-form. By \cite{Abouzaid:NearbyLagrangians} and \cite{Kragh:Parametrized}, the exactness now implies that $L$ indeed is a torus whose inclusion is homotopy equivalent to the zero-section.
\end{proof}

Given a closed one-form $\alpha \in \Omega^1M$ we use $\tau_\alpha \colon (T^*M,d(\lambda_M+\alpha)) \to (T^*M,d(\lambda_M))$ to denote the (not necessarily exact) symplectomorphism given by fibre-wise addition with the section $\alpha.$ The identification $T^*\T^2=\T^2\times \R^2$ provides us with the `constant' one-forms $p_1d\theta_1+p_2d\theta_2$ and we write $\tau_{(p_1,p_2)}$ for the corresponding symplectomorphism.

\begin{lma}
\label{lma:exact}
If $L \subset (T^*M,d\lambda_M)$ is a closed connected Lagrangian submanifold which is weakly exact, and if $M$ is connected with $\pi_1(M)$ abelian, then the fibrewise translation $\tau_\alpha(L)$ by a suitable closed one-form $\alpha \in \Omega^1M$ is exact.
\end{lma}
\begin{proof}
Since $\pi_1(M)$ is abelian, it follows that $H_1(T^*M)=\pi_1(T^*M),$ and hence any element
$$\gamma \in \ker(H_1(L) \xrightarrow{[\iota_L]}H_1(T^*M))$$
can be lifted to $\widetilde{\gamma} \in \pi_1(L)$ that is contained in the image of the connecting homomorphism $\pi_2(T^*M,L) \to \pi_1(L).$

By the the previous paragraph, together with the assumption of weak exactness of $L,$ it now follows that $[\iota_L^*\lambda_{M}]$ vanishes on $\ker [\iota_L] \subset H_1(L;\R).$ The symplectic action class can thus be written as $[\iota_L^*\lambda_{M}]=-[(\pi \circ \iota_L)^*\alpha]$ for some closed one-form $\alpha \in \Omega^1M.$

For any section $\alpha \in \Omega^1M$ of a closed one-form, the pull-back of $\lambda_{M}$ under $\iota_L$ and its fibrewise translation $\tau_\alpha\circ \iota_{L}$ by $\alpha$ satisfies
$$ [(\tau_\alpha \circ \iota_L)^*\lambda_{M}]=[(\pi \circ \iota_L)^*\alpha]+[\iota_L^*\lambda_{M}]=0$$
in $H^1(L,\R),$ which shows the statement.
\end{proof}

We are now ready to prove Parts (1)--(3) of Theorem \ref{thm:nearby}. 

\begin{proof}[Proof of Part (1):] As already established, there is a Hamiltonian isotopy $L_t$ from the Lagrangian torus $L=L_0$ to section $L_1$ of a closed one-form, which may be assumed to be a constant section $L_1=\T^2 \times \{\mathbf{p}_0\}.$ Unless $\mathbf{p}_0 \in U \subset \R^2,$ it follows that $L$ is displaceable from itself by a Hamiltonian isotopy, which thus contradicts e.g.~Floer's original work \cite{Floer:Morse} or the result in \cite{Laudenbach:Persistance} by F. Laudenbach and J.-C. Sikorav.

In order to see that the entire Hamiltonian isotopy can be assumed to be contained inside $U,$ we note the following. First, fibrewise rescalings
\begin{gather*}
\sigma_s \colon (T^*\T^2,d\lambda_{\T^2})\to(T^*\T^2,e^{-s}d\lambda_{\T^2}),\\
(\boldsymbol{\theta},\mathbf{p}) \mapsto (\boldsymbol{\theta},e^s\mathbf{p}),
\end{gather*}
preserve exact Lagrangian submanifolds. Hence, the induced isotopy by such a rescaling of an exact Lagrangian submanifold can thus be generated by a Hamiltonian isotopy. Since $\tau_{\mathbf{p}_0}^{-1}(L)$ is exact, we now see that the isotopy $\tau_{\mathbf{p}_0} \circ \sigma_s \circ \tau_{\mathbf{p}_0}^{-1}$ acts on our Lagrangian $L$ by Hamiltonian isotopy.

Since the conformal symplectic isotopy
$$\tau_{\mathbf{p}_0} \circ \sigma_s \circ \tau_{\mathbf{p}_0}^{-1} \colon (T^*\T^2,d\lambda_{\T^2}) \to (T^*\T^2,e^{-t}d\lambda_{\T^2})$$
contracts any compact subset into an arbitrarily small neighbourhood of $\T^2 \times \{\mathbf{p}_0\}$ as $s \to -\infty,$ and since it preserves the subset $\T^2 \times U \supset \T^2 \times \{\mathbf{p}_0\}$ by the convexity assumptions, it now follows that we can deform the Hamiltonian isotopy $L_t$ to the Hamiltonian isotopy $\tau_{\mathbf{p}_0} \circ \sigma_s \circ \tau_{\mathbf{p}_0}^{-1}(L_t),$ $s \gg 0,$ which is contained entirely inside $\T^2 \times U.$ Finally, we also need to pre-concatenate with the Hamiltonian isotopy $\tau_{\mathbf{p}_0} \circ \sigma_s \circ \tau_{\mathbf{p}_0}^{-1}(L),$ which also is contained entirely inside $\T^2 \times U$ for the very same reasons.
\end{proof}

\begin{proof}[Proof of Part (2):] The geometric assumptions made on the Lagrangian enables us to construct a Hamiltonian isotopy which `frees up' space around $L$ above $S^1 \times e^{i[-\epsilon,\epsilon]} \subset \T^2.$
\begin{lma}
\label{lma:cleanup}
There exists a Hamiltonian isotopy $\phi^t_{H_t}$ with support in the complement of $\bigcup_{|s|\le \epsilon} \dot{D}_{\mathbf{p}^0}( e^{is} ),$ where
$$ \phi^1_{H_t}(L) \cap \left( S^1 \times e^{i[-\epsilon/2,\epsilon/2]} \times \R \times (-\infty,p_2^0]\right) = S^1 \times e^{i[-\epsilon/2,\epsilon/2]} \times \{\mathbf{p}^0\}$$
is satisfied for the time-1 map.
\end{lma}
See Figure \ref{fig:cleanup} for an example of the effect of the above symplectic isotopy in the case $\mathbf{p}^0=(0,0).$
\begin{proof}
We act on the subset of $L$ contained inside 
\begin{align*}
S^1 \times (-\epsilon,\epsilon) \times O \subset T^*\T^2,\:\:\:\:O := \R^2 \setminus (\{p_1^0\} \times (-\infty,p_2^0]),
\end{align*}
by the symplectic isotopy given by the fibre-wise translation $\tau_{t h(\theta_2)d\theta_2}.$ Here $h(\theta_2) \ge 0$ has is a smooth bump function with nonempty support contained inside $(-\epsilon/2,\epsilon/2).$ Here it is crucial to use the assumptions on the intersection properties of $L$ and the Lagrangian disc in order to infer that this produces a Lagrangian isotopy.

Finally we argue that the above Lagrangian isotopy has a vanishing symplectic flux-path and, hence, that it is a Hamiltonian isotopy as sought. It suffices to consider the change of symplectic action on a closed loop $\gamma \subset L$ which is homotopic to $\{\pt\} \times S^1 \times \{0\}$ inside $T^*\T^2$ (recall that $L$ is homotopic to the zero-section). Since the degree of the projection $T^*\T^2 \to S^1$ to the second factor of $\T^2=S^1 \times S^1$ restricted to $\gamma$ is one, and since the part of $\gamma$ contained inside $S^1 (-\epsilon,\epsilon) \times \{0\} \subset L$ is fixed by the deformation, one readily computes that the flux is unchanged (by degree reasons the total integral over the remaining sheets of $\gamma$ where the isotopy is supported is equal to zero).
\end{proof}

\begin{figure}[htp]
\begin{center}
\vspace{6mm}
\labellist
\pinlabel $p_2$ at 68 120
\pinlabel $p_2$ at 224 120
\pinlabel $p_1$ at 145 35
\pinlabel $p_1$ at 301 35
\pinlabel $\color{blue}L\cap\{|\theta_2|<2\delta\}$ at 28 82
\pinlabel $\color{blue}\phi^1_{H_t}(L)\cap\{|\theta_2|<2\delta\}$ at 168 104
\pinlabel $\color{red}C(e^{is})$ at 182 45
\pinlabel $\color{red}C(t,e^{is})$ at 187 17
\pinlabel $\bigcup_{|s|\le 2\delta}\dot{D}_{\mathbf{p}^0}(e^{is})$ at 109 5
\pinlabel $\bigcup_{|s|\le 2\delta}\dot{D}_{\mathbf{p}^0}(e^{is})$ at 265 5
\endlabellist
\includegraphics{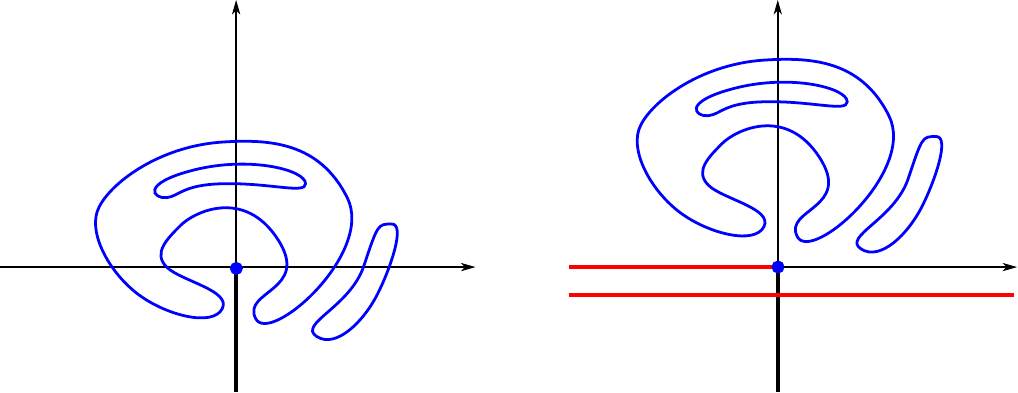}
\caption{The canonical projection of $L \cap (S^1 \times (-\epsilon,\epsilon) \times \R^2)$ to $\R^2$ before and after the deformation provided by Lemma \ref{lma:cleanup}. Here $\mathbf{p}^0=(0,0)$ and $2\delta \coloneqq \epsilon/2.$ In this subset the deformation is simply a translation of $L \setminus S^1 \times (-\epsilon,\epsilon) \times \{\mathbf{p}^0\}$ in the positive $p_2$-direction.} 
\label{fig:cleanup}
\end{center}
\end{figure}
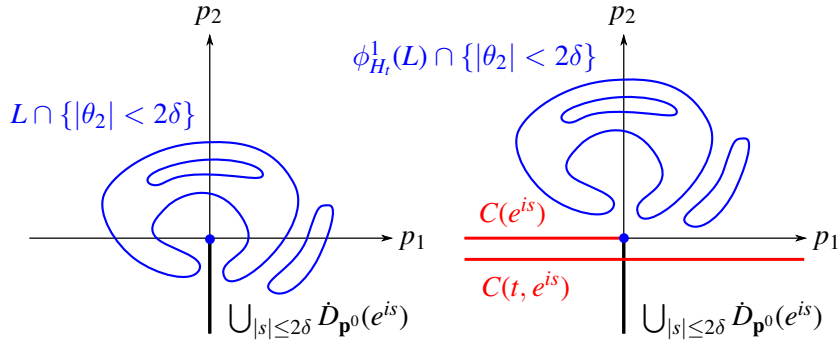

We now replace $L$ with the Lagrangian produced by Lemma \ref{lma:cleanup} and will write
$$ \delta \coloneqq \epsilon/8.$$
Under these assumptions it is possible to establish a refined Lagrangian isotopy result -- this is the only result in this section whose proof requires hard techniques.
\begin{prp}
\label{prp:LagIsoFixed}
For $L$ as above, there exists a \emph{Lagrangian} isotopy $L_t$ from $L_0=L$ to a graphical Lagrangian torus $L_1,$ where the isotopy satisfies the additional property that
$$ L_t \cap \left(S^1 \times e^{i(-2\delta,2\delta)} \times \R \times (-\infty, p_2^0] \right) = S^1 \times \{\theta_2\} \times \{p_1^t,p_2^0\}, $$
and where $p_1^t$ depends smoothly on $t \in [0,1].$
\end{prp}
We prove this proposition in Section \ref{sec:ProofLagIsoFixed} below by following the steps of the proof of \cite[Theorem 7.1]{Dimitroglou:Isotopy} while taking additional care.

Given Proposition \ref{prp:LagIsoFixed}, Part (2) of Theorem \ref{thm:nearby} is now a direct consequence of the straight-forward result Lemma \ref{lma:HamFromLag} below.
\end{proof}

\begin{lma}
\label{lma:HamFromLag}
 
Consider the Lagrangian isotopy $L_t$ produced by Proposition \ref{prp:LagIsoFixed}. There exists a smooth path of one-forms
$$\alpha_t=(p_1-p_1^t)d\theta_1+g_t(\theta_2)d\theta_2$$
supported inside $S^1 \times e^{i(-2\delta,2\delta)},$ and where $g_t=0$ moreover holds in $e^{i(-\delta,\delta)},$ for which $\tau_{\alpha_t}(L_t)$ is a Hamiltonian isotopy.

In other words, $\tau_{\alpha_t}(L_t)$ is a Hamiltonian isotopy which starts at $L_0=L,$ satisfies
$$ \tau_{\alpha_t}(L_t) \cap \left(S^1 \times e^{i(-\delta,\delta)} \times \R \times (-\infty,p_2^0]\right) = S^1 \times e^{i(-\delta,\delta)} \times \{\mathbf{p}^0\},$$
for all $t \in [0,1],$ and ends at the graphical Lagrangian torus $\tau_{\alpha_1}(L_1).$
\end{lma}
\begin{proof}
The functions $g_t(\theta_2)$ are constructed to satisfy
$$\int_{S^1} g_t(\theta_2)d\theta_2=\int_{\{e^{i0}\} \times S^1} \lambda_{\T^2}|_{TL_0}-\int_{\{e^{i0}\} \times S^1} \lambda_{\T^2}|_{TL_t}$$
for all $t \in [0,1].$ Using the fact that $L \subset T^*\T^2$ is a homotopy equivalence, the resulting deformed isotopy is now be seen to have a vanishing symplectic flux-path, as sought. \end{proof}

\begin{proof}[Proof of Part (3):]
Consider the Hamiltonian isotopy $L_t$ constructed in Part (2), where $L_1$ is a Lagrangian that is the graph of the closed one-form $\alpha$ on $\T^2.$ Since $L=L_0 \subset \T^2 \times U$ we can readily find a different closed one-form $\beta,$ coinciding with $\alpha$ near $\{\theta_2=0\},$ and whose graph is contained inside $\T^2 \times U.$ Here we have used the fact that there is a Hamiltonian isotopy from $L_0$ to a constant section inside $\T^2 \times U$ by Part (1). We then construct a Hamiltonian isotopy $\widetilde{L}_t$ from $\widetilde{L}_0=L$ to the graph of $\beta$ which satisfies the requirements of Part (2) as follows: concatenate the original path $L_t$ with a path of graphical Lagrangians from the fibrewise convex interpolation $t\beta+(1-t)\alpha.$

What remains is to deform the Hamiltonian isotopy constructed above relative its endpoints in order to make it confined to the subset $\T^2 \times U \supset L.$ This we do by arguing as in Part (1), i.e.~by considering the deformed Hamiltonian isotopy
$$\tau_{\beta} \circ \sigma_s \circ \tau_{\beta}^{-1}(\widetilde{L}_t), \:\: s \gg 0.$$
Here it is important to use the property that
$$U \setminus \{ p_2 < p_2^0 \} \subset \R^2 $$
is a \emph{star-shaped} set with respect to the point $\mathbf{p}^0,$ in order to ensure that the requirements in Part (2) are satisfied also for this new Hamiltonian isotopy.
\end{proof}

\subsection{Proof of Proposition \ref{prp:LagIsoFixed}}
\label{sec:ProofLagIsoFixed}

Here we show the existence of the Lagrangian isotopy $L_t.$ After a suitable translation by the constant section $-(p_1^0d\theta_1+p_2^0d\theta_2)$ it suffices to consider the case when $p_1^0=p_2^0=0.$

The starting point of the proof is the observation that
$$ C(e^{is}):=S^1 \times \{e^{is}\} \times (-\infty,p_1^0) \times \{p_2^0\} \subset T^*\T^2 \setminus L$$
is a smooth one-dimensional family of embedded two-punctured symplectic spheres, i.e.~cylinders, which moreover are pseudoholomorphic for the standard cylindrical almost complex structure $J_{\OP{cyl}}$ on $T^*\T^2 \setminus 0_L$ (see Section \ref{sec:neckstretch}). Here $C(e^{is})$ is parametrised by $(\theta_1,p_1),$ $p_1 < p_1^0=0,$ while the family is parametrised by $s \in (-4\delta,4\delta).$ Note that the cylinder is asymptotic to Reeb chords of $T^*\T^2 \setminus L$ at its convex (i.e.~near $+\infty$) as well as concave (i.e.~near $L$) ends. In fact this is a `trivial' cylinder inside the symplectisation $T^*\T^2 \setminus 0_{T^2} \xrightarrow{\cong} \R \times ST^*\T^2,$ in the sense that each $C(e^{is})$ is invariant under translation of the symplectisation coordinate.

In addition, for all $s \in (-4\delta,4\delta)$ and $p_2 < p_2^0,$ there are $J_{\OP{cyl}}$-holomorphic cylinders
$$ C(p_2,e^{is}):=S^1 \times \{e^{is}\} \times \R \times \{p_2\} \subset T^*\T^2 \setminus L$$
having both punctures asymptotic to Reeb orbits in the convex end (i.e.~near $+\infty$).

Of course, since $L$ only partially coincides with $0_{\T^2},$ the almost complex structure $J_{\OP{cyl}}$ is typically not defined on all of $T^*\T^2 \setminus L.$ However, we can find a well-defined compatible almost complex structure $J$ on $(T^*\T^2 \setminus L, d\lambda_{\T^2})$ which is cylindrical outside of a compact subset, and which agrees with $J_{\OP{cyl}}$ in a neighbourhood of the union of cylinders
$$ C(e^{is}) \cup C(p_2,e^{is}), \:\: |s| \le 4\delta, \:\: p_2 < p_2^0.$$
Here it is crucial that these cylinders all are disjoint from $L$; see Figure \ref{fig:cleanup}. (This was achieved by the previous application of Lemma \ref{lma:cleanup}.)

We now follow the argument given in \cite[Section 7]{Dimitroglou:Isotopy}. First, we compactify $D_RT^*\T^2$ to $S^2 \times S^2$ for some $R \gg 0$ as described in \cite[Section 7]{Dimitroglou:Isotopy}. The above cylinders $C(e^{is})$ now become one-punctured pseudoholomorphic spheres inside the latter space, i.e.~pseudoholomorphic planes, given that $J$ is carefully chosen. In fact, we may require $J$ to be equal to the standard product complex structure near the divisor
$$ C_\infty \coloneqq S^2 \times \{0,\infty\} \: \: \sqcup \:\: \{0,\infty\} \times S^2 \subset S^2 \times S^2,$$
and thus, in particular, this divisor $J$-holomorphic.

Second, arguing as in in \cite{Dimitroglou:Isotopy} by using automatic transversality, positivity of intersection, and the SFT compactness theorem, we obtain the following result, which is the core of the argument here:
\begin{prp}[Section 6 in \cite{Dimitroglou:Isotopy}]
The above family $\{C(e^{is})\}_{|s|<4\delta}$ of $J$-holomorphic planes inside $S^2 \times S^2 \setminus L$ lives inside a regularly cut out and compact component its moduli space $\mathcal{M}=\{C(\theta)\}_{\theta \in S^1} \cong S^1.$ Furthermore, the planes $C(\theta),$ $\theta \in S^1,$ foliate the hypersurface
$$\bigcup_{\theta \in S^1} C(\theta) \subset T^*\T^2 \setminus L$$
diffeomorphic to $S^1 \times C(e^{i0}),$ and the asymptotic evaluation map $\mathcal{M} \to \Gamma$ is a diffeomorphism, where we have used $\Gamma\cong S^1$ to denote the corresponding $S^1$-family of Reeb orbits on $ST^*L.$
\end{prp}

By construction, the discs $C(\theta)$ with $\theta \in e^{i(-4\delta,4\delta)},$ may be assumed to coincide with planes of the form
$$B^2 \times \{e^{i\theta}\} \subset S^2 \times S^2,$$
where $B^2 \subset S^2$ denotes the lower open hemisphere.
\begin{lma}[Section 7.1 in \cite{Dimitroglou:Isotopy}]
\label{lma:straighten}
After a deformation of $C(\theta),$ $\theta \in S^1,$ supported near the divisor $C_\infty \subset S^2 \times S^2$ we may further assume that each $C(\theta),$ $\theta \in S^1,$ coincides with the $J_{\OP{cyl}}$-holomorphic cylinder
$$B^2 \times \{e^{i\theta}\} \subset S^2 \times S^2$$
in that neighbourhood.
\end{lma}

Third, after deforming the planes $\{C(\theta)\},$ $\theta \notin e^{i(-3\delta,3\delta)},$ as in \cite[Section 5.3]{Dimitroglou:Isotopy}, we may assume that the compactifications $D(\theta) \subset (S^2 \times S^2,L)$ for all $\theta \in S^1$ yield a smoothly embedded solid torus
$$(\widetilde{\mathcal{T}} \coloneqq D^2 \times S^1,S^1 \times S^1) \hookrightarrow (S^2 \times S^2,L).$$
Using the initial identification
$$S^2 \times S^2 \setminus C_\infty \xrightarrow{\cong} D_RT^*\T^2$$
we extend the image of $\widetilde{\mathcal{T}}$ to a proper embedding
$$((D^2 \setminus \{0\}) \times S^1,S^1 \times S^1) \xrightarrow{\cong} (\dot{\mathcal{T}},\partial \dot{\mathcal{T}}) \subset (T^*\T^2,L)$$
foliated by symplectic punctured discs $\dot{D}(\theta).$ Using the existence of our initial punctured pseudoholomorphic spheres $C(e^{is})$ and $C(p_2,e^{is}),$ we can make sure that this `punctured solid torus' is of a standard form.
\begin{lma}
\label{lma:cleanup2}
In addition to the above properties, we may assume that
\begin{enumerate}
\item The punctured disc leaves $\dot{D}(\theta)$ for $\theta \in e^{i(-3\delta,3\delta)}$ coincide with
$$ S^1 \times \{\theta\} \times \R_{\le 0} \times \{0\},$$
while $\dot{D}(\theta) \setminus D_RT^*\T^2$ all are contained inside the latter subset; and
\item The intersection
$$ \dot{\mathcal{T}} \cap \left(S^1 \times e^{i(-3\delta,3\delta)} \times \R \times \R_{<0}\right) = \emptyset$$
is empty.
\end{enumerate}
\end{lma}
\begin{proof}
(1): This property can be achieved by Lemma \ref{lma:straighten}.

(2): This is a consequence of positivity of intersection with the cylinders $C(p_2,s^{is})$ with $p_2<0$ which can be assumed to remain pseudoholomorphic during all of the steps taken above.
\end{proof}

Lastly, we need to correct the monodromy map induced by the characteristic distribution of the solid torus. Recall that this monodromy is a symplectomorphism
$$\phi \colon (S^1 \times \{e^{i0}\} \times \R_{\le0} \times \{0\},d\lambda_{\T^2}) \to (S^1 \times \{e^{i0}\} \times \R_{\le0} \times \{0\},d\lambda_{\T^2}) $$ 
obtained by integrating the line field
$$\ker( d\lambda_{\T^2}|_{T\dot{\mathcal{T}}})\subset T\dot{\mathcal{T}}$$
which clearly is transverse to all symplectic leaves $\dot{D}(\theta).$ The goal of this deformation is to make this symplectomorphism preserve the foliation
$$ \gamma_p \coloneqq S^1 \times \{e^{i0}\} \times \{p\} \times \{0\} \subset S^1 \times \{e^{i0}\} \times \R_{\le0} \times \{0\} =\dot{D}(e^{i0})$$
by simple closed curves. Observe that this is already the case for the curve $\gamma_0$ as well as $\gamma_{-P}$ for all $P \gg 0,$ as follows from the Lagrangian condition of $L$ and Part (1) of Lemma \ref{lma:cleanup2}, respectively.

The deformation of the family of annuli is performed by a so-called symplectic suspension of a suitable Hamiltonian isotopy $\phi^t_{H_t}$ of the symplectic punctured disc
$$ \dot{D}(e^{i0})=S^1 \times \{e^{i0}\} \times \R_{\le 0} \times \{0\} \subset (T^*\T^2,d\lambda_{\T^2}).$$

The argument for the existence of Hamiltonian isotopies of the disc in \cite[Lemma 7.4]{Dimitroglou:Isotopy} carries over to show the following.
\begin{lma}
\label{lma:HamIso}
There exists a Hamiltonian isotopy
$$\phi^t_{H_t} \colon (\dot{D}(e^{i0}) \cong D^2 \setminus \{0\},\omega) \xrightarrow{\cong} (\dot{D}(e^{i0}) \cong D^2 \setminus \{0\},\omega), \:\: t \in [0,\delta],$$
where $H_t=0$ for all $t \notin [\delta/2,\delta)$ which
\begin{itemize}
\item is generated by a non-negative time-dependent Hamiltonian $H_t \ge 0$ that has compact support, and satisfies $H_t|_{\partial D^2}\equiv 0$ for all $t,$ and
\item satisfies the property that $\phi^{\delta}_{H_t} \circ \phi$ preserves the leaves of the foliation $\{\gamma_p\}$ by embedded closed curves.
\end{itemize}
\end{lma}
\begin{proof}
 
We begin by recalling the standard fact that any symplectomorphism of an annulus that preserves the boundary set-wise is Hamiltonian isotopic, while fixing the boundary set-wise, to one which preserves the standard foliation by meridian curves. Whenever a boundary component is fixed point-wise by the symplectomorphism, this Hamiltonian isotopy may moreover be assumed to fix the same component pointwise. To see this claim, first we point out that a Hamiltonian isotopy can be explicitly constructed in order to deform the symplectomorphism to one which is the identity on both boundary components. We can then use \cite[Proposition A.4]{Cieliebak:Hamiltonian}, together with the facts that the standard Dehn twist (and hence its powers) are compactly supported symplectomorphisms that preserve the standard foliation by meridians.

After adding a time-dependent constant we may assume that the Hamiltonian constructed above satisfies $H_t|_{\partial D^2} \equiv 0.$ In order to make it non-negative everywhere, we can simply add a suitable autonomous Hamiltonian which is constant along any meridian in the foliation. (The generated isotopy now performs an additional rotation along the meridian curves.) Note that the same technique also can be used to make $H_t$ vanish near the origin $0 \in D^2,$ as opposed to being merely constant there.
\end{proof}

The symplectic suspension that we use is the locally defined symplectomorphism
\begin{gather*}
 (z,(e^{is},p_2)) \mapsto \left(\phi^{s+3\delta}_{H_t}(z),e^{is},p_2-H_{s+3\delta}\left(\phi^{s+3\delta}_{H_t}(z)\right)\right),\\
z \in D(e^{i0}), \:\: s \in [-3\delta,-2\delta].
\end{gather*}
When applying this symplectomorphism to the family $S^1 \times e^{i(-3\delta,-2\delta)} \times \R_{\le 0} \times \{0\}$ of punctured discs, the properties in Lemma \ref{lma:cleanup2} satisfied by $\dot{\mathcal{T}}$ implies the following: the deformation of $\dot{\mathcal{T}}$ obtained by excising the above family and replacing it by its suspension $\dot{\mathcal{T}}'$ is embedded, and has a symplectic monodromy that preserves all leaves $\{\gamma_p\},$ $p \le 0,$ of the foliation of $\dot{D}(e^{i0}).$

Given the above, the Lagrangian isotopy is finally constructed to be contained completely inside the hypersurface $\dot{\mathcal{T}}'$ foliated by punctured symplectic discs. More precisely, we consider the Lagrangian tori $L_t$ in $\dot{\mathcal{T}}'$ uniquely determined by the property that they intersect the disc $\dot{D}(0)$ precisely in the curve $\gamma_{-Ct}$ for some $C \gg 0$ sufficiently large. Indeed, by the assumption that all annuli are standard outside of a compact subset of $T^*\T^2,$ it necessarily follows that these tori all are graphical for $t=1.$
\qed

\bibliographystyle{gtart}
\bibliography{references}

\def\cprime{$'$} \def\cprime{$'$} \def\cprime{$'$} \def\cprime{$'$}
\begin{thebibliography}{}
\providecommand\bibmarginpar{\leavevmode\marginpar}
\def\urlstyle#1{{\tt #1}}

\bibitem{Abouzaid:NearbyLagrangians}
\textbf{M Abouzaid}, \href{http://dx.doi.org/10.1007/s00222-011-0365-0}
  {\emph{Nearby {L}agrangians with vanishing {M}aslov class are homotopy
  equivalent}}, Invent. Math. 189 (2012) 251--313

\bibitem{FirstSteps}
\textbf{V\,I Arnol{\cprime}d}, \emph{The first steps of symplectic topology},
  Uspekhi Mat. Nauk 41 (1986) 3--18, 229

\bibitem{Audin:SymplecticRigidity}
\textbf{M Audin}, \textbf{F Lalonde}, \textbf{L Polterovich},
  \href{http://dx.doi.org/10.1007/s10107-007-0158-9} {\emph{Symplectic
  rigidity: {L}agrangian submanifolds}}, from ``Holomorphic curves in
  symplectic geometry'', Progr. Math. 117, Birkh\"auser, Basel (1994)  271--321

\bibitem{Auroux:SpecialLagrangian}
\textbf{D Auroux}, \href{http://dx.doi.org/10.4310/SDG.2008.v13.n1.a1}
  {\emph{Special {L}agrangian fibrations, wall-crossing, and mirror symmetry}},
  from ``Surveys in differential geometry. {V}ol. {XIII}. {G}eometry, analysis,
  and algebraic geometry: forty years of the {J}ournal of {D}ifferential
  {G}eometry'', Surv. Differ. Geom. 13, Int. Press, Somerville, MA (2009)
  1--47

\bibitem{Bourgeois:Compactness}
\textbf{F Bourgeois}, \textbf{Y Eliashberg}, \textbf{H Hofer}, \textbf{K
  Wysocki}, \textbf{E Zehnder}, \href{http://dx.doi.org/10.2140/gt.2003.7.799}
  {\emph{Compactness results in symplectic field theory}}, Geom. Topol. 7
  (2003) 799--888

\bibitem{Chekanov:LagrangianTori}
\textbf{Y\,V Chekanov}, \href{http://dx.doi.org/10.1007/PL00004278}
  {\emph{Lagrangian tori in a symplectic vector space and global
  symplectomorphisms}}, Math. Z. 223 (1996) 547--559

\bibitem{Chekanov:LagrangianIntersections}
\textbf{Y\,V Chekanov}, \href{http://dx.doi.org/10.1215/S0012-7094-98-09506-0}
  {\emph{Lagrangian intersections, symplectic energy, and areas of holomorphic
  curves}}, Duke Math. J. 95 (1998) 213--226

\bibitem{Cieliebak:Compactness}
\textbf{K Cieliebak}, \textbf{K Mohnke},
  \href{http://projecteuclid.org/euclid.jsg/1154467631} {\emph{Compactness for
  punctured holomorphic curves}}, J. Symplectic Geom. 3 (2005)
  589--654Conference on Symplectic Topology

\bibitem{Cieliebak:PuncturedHolomorphic}
\textbf{K Cieliebak}, \textbf{K Mohnke},
  \href{http://dx.doi.org/10.1007/s00222-017-0767-8} {\emph{Punctured
  holomorphic curves and {L}agrangian embeddings}}, Invent. Math. 212 (2018)
  213--295

\bibitem{Cieliebak:Hamiltonian}
\textbf{K Cieliebak}, \textbf{M Schwingenheuer}, \emph{Hamiltonian
  unknottedness of certain monotone {L}agrangian tori in {$S^2\times S^2$}}To
  appear in \emph{Pacific J. Math.} (2018).

\bibitem{Dimitroglou:Extremal}
\textbf{G {Dimitroglou Rizell}}, \emph{Uniqueness of extremal {L}agrangian tori
  inside the four-dimensional disc}, from ``Proceedings of the G\"{o}kova
  Geometry-Topology Conference 2015'', G\"{o}kova Geometry/Topology Conference
  (GGT), G\"{o}kova (2016)  151--167

\bibitem{Dimitroglou:Isotopy}
\textbf{G Dimitroglou~Rizell}, \textbf{E Goodman}, \textbf{A Ivrii},
  \href{http://dx.doi.org/10.1007/s00039-016-0388-1} {\emph{Lagrangian isotopy
  of tori in {$S^2\times S^2$} and {$\Bbb{C}P^2$}}}, Geom. Funct. Anal. 26
  (2016) 1297--1358

\bibitem{LagCaps}
\textbf{Y Eliashberg}, \textbf{E Murphy},
  \href{http://dx.doi.org/10.1007/s00039-013-0239-2} {\emph{Lagrangian caps}},
  Geom. Funct. Anal. 23 (2013) 1483--1514

\bibitem{Eliashberg:LocalLagrangian}
\textbf{Y Eliashberg}, \textbf{L Polterovich},
  \href{http://dx.doi.org/10.2307/2118583} {\emph{Local {L}agrangian
  {$2$}-knots are trivial}}, Ann. of Math. (2) 144 (1996) 61--76

\bibitem{Eliashberg:ProblemLagrangian}
\textbf{Y Eliashberg}, \textbf{L Polterovich}, \emph{The problem of
  {L}agrangian knots in four-manifolds}, from ``Geometric topology ({A}thens,
  {GA}, 1993)'', AMS/IP Stud. Adv. Math. 2, Amer. Math. Soc., Providence, RI
  (1997)  313--327

\bibitem{Floer:Morse}
\textbf{A Floer},
  \href{http://projecteuclid.org/getRecord?id=euclid.jdg/1214442477}
  {\emph{Morse theory for {L}agrangian intersections}}, J. Differential Geom.
  28 (1988) 513--547

\bibitem{Gadbled:OnExotic}
\textbf{A Gadbled}, \href{http://projecteuclid.org/euclid.jsg/1384282840}
  {\emph{On exotic monotone {L}agrangian tori in {$\Bbb{CP}^2$} and {$\Bbb
  S^2\times\Bbb S^2$}}}, J. Symplectic Geom. 11 (2013) 343--361

\bibitem{ContactTendue}
\textbf{E Giroux}, \href{http://www.numdam.org/item?id=ASENS_1994_4_27_6_697_0}
  {\emph{Une structure de contact, m\^eme tendue, est plus ou moins tordue}},
  Ann. Sci. \'Ecole Norm. Sup. (4) 27 (1994) 697--705

\bibitem{Gromov:Pseudo}
\textbf{M Gromov}, \href{http://dx.doi.org/10.1007/BF01388806}
  {\emph{Pseudoholomorphic curves in symplectic manifolds}}, Invent. Math. 82
  (1985) 307--347

\bibitem{Hind:LagrangianSpheres}
\textbf{R Hind}, \href{http://dx.doi.org/10.1007/s00039-004-0459-6}
  {\emph{Lagrangian spheres in {$S^2\times S^2$}}}, Geom. Funct. Anal. 14
  (2004) 303--318

\bibitem{Hind:SymplecticEmbeddings}
\textbf{R Hind}, \textbf{S Lisi},
  \href{http://dx.doi.org/10.1007/s00029-013-0146-2} {\emph{Symplectic
  embeddings of polydisks}}, Selecta Mathematica  (2014) 1--22

\bibitem{Hofer:OnGenericity}
\textbf{H Hofer}, \textbf{V Lizan}, \textbf{J-C Sikorav},
  \href{http://dx.doi.org/10.1007/BF02921708} {\emph{On genericity for
  holomorphic curves in four-dimensional almost-complex manifolds}}, J. Geom.
  Anal. 7 (1997) 149--159

\bibitem{Hofer:IV}
\textbf{H Hofer}, \textbf{K Wysocki}, \textbf{E Zehnder}, \emph{Properties of
  pseudoholomorphic curves in symplectisation. {IV}. {A}symptotics with
  degeneracies}, from ``Contact and symplectic geometry ({C}ambridge, 1994)'',
  Publ. Newton Inst. 8, Cambridge Univ. Press, Cambridge (1996)  78--117

\bibitem{Hofer:I}
\textbf{H Hofer}, \textbf{K Wysocki}, \textbf{E Zehnder},
  \href{http://dx.doi.org/10.1016/S0294-1449(98)80034-6} {\emph{Properties of
  pseudoholomorphic curves in symplectisations. {I}. {A}symptotics}}, Ann.
  Inst. H. Poincar\'e Anal. Non Lin\'eaire 13 (1996) 337--379

\bibitem{Kragh:Parametrized}
\textbf{T Kragh}, \href{http://dx.doi.org/10.2140/gt.2013.17.639}
  {\emph{Parametrized ring-spectra and the nearby {L}agrangian conjecture}},
  Geom. Topol. 17 (2013) 639--731

\bibitem{Laudenbach:Persistance}
\textbf{F Laudenbach}, \textbf{J-C Sikorav},
  \href{http://dx.doi.org/10.1007/BF01388807} {\emph{Persistance d'intersection
  avec la section nulle au cours d'une isotopie hamiltonienne dans un fibr\'e
  cotangent}}, Invent. Math. 82 (1985) 349--357

\bibitem{Lazzarini:RelativeFrames}
\textbf{L Lazzarini}, \href{http://dx.doi.org/10.1007/s11784-010-0004-1}
  {\emph{Relative frames on {$J$}-holomorphic curves}}, J. Fixed Point Theory
  Appl. 9 (2011) 213--256

\bibitem{McDuff:LocalBehaviour}
\textbf{D McDuff},
  \href{http://projecteuclid.org/getRecord?id=euclid.jdg/1214446994} {\emph{The
  local behaviour of holomorphic curves in almost complex {$4$}-manifolds}}, J.
  Differential Geom. 34 (1991) 143--164

\bibitem{McDuff:Introduction}
\textbf{D McDuff}, \textbf{D Salamon}, \emph{Introduction to symplectic
  topology}, second edition, Oxford Mathematical Monographs, The Clarendon
  Press Oxford University Press, New York (1998)

\bibitem{McDuff:J}
\textbf{D McDuff}, \textbf{D Salamon}, \emph{{$J$}-holomorphic curves and
  symplectic topology}, volume~52 of \emph{American Mathematical Society
  Colloquium Publications}, second edition, American Mathematical Society,
  Providence, RI (2012)

\bibitem{McLean:LefschetzFibrations}
\textbf{M McLean}, \href{http://dx.doi.org/10.2140/gt.2009.13.1877}
  {\emph{Lefschetz fibrations and symplectic homology}}, Geom. Topol. 13 (2009)
  1877--1944

\bibitem{Pascaleff:FloerCohomology}
\textbf{J Pascaleff}, \href{http://dx.doi.org/10.1215/00127094-2804892}
  {\emph{Floer cohomology in the mirror of the projective plane and a binodal
  cubic curve}}, Duke Math. J. 163 (2014) 2427--2516

\bibitem{Siebert:OnTheHolomorphicity}
\textbf{B Siebert}, \textbf{G Tian},
  \href{http://dx.doi.org/10.4007/annals.2005.161.959} {\emph{On the
  holomorphicity of genus two {L}efschetz fibrations}}, Ann. of Math. (2) 161
  (2005) 959--1020

\bibitem{Symington:FourDimensions}
\textbf{M Symington}, \href{http://dx.doi.org/10.1090/pspum/071/2024634}
  {\emph{Four dimensions from two in symplectic topology}}, from ``Topology and
  geometry of manifolds ({A}thens, {GA}, 2001)'', Proc. Sympos. Pure Math. 71,
  Amer. Math. Soc., Providence, RI (2003)  153--208

\bibitem{Vianna:first}
\textbf{R Vianna}, \href{http://dx.doi.org/10.2140/gt.2014.18.2419} {\emph{On
  exotic {L}agrangian tori in {$\Bbb{CP}^2$}}}, Geom. Topol. 18 (2014)
  2419--2476

\bibitem{Vianna:second}
\textbf{R Vianna}, \href{http://dx.doi.org/10.1112/jtopol/jtw002}
  {\emph{Infinitely many exotic monotone {L}agrangian tori in {$\Bbb{CP}^2$}}},
  J. Topol. 9 (2016) 535--551

\bibitem{Weinstein:Lagrangian}
\textbf{A Weinstein}, \emph{Symplectic manifolds and their {L}agrangian
  submanifolds}, Advances in Math. 6 (1971) 329--346 (1971)

\bibitem{Wendl:Automatic}
\textbf{C Wendl}, \href{http://dx.doi.org/10.4171/CMH/199} {\emph{Automatic
  transversality and orbifolds of punctured holomorphic curves in dimension
  four}}, Comment. Math. Helv. 85 (2010) 347--407

\end{thebibliography}

\end{document}